\documentclass[12pt]{amsart}
\usepackage{graphicx} 
\usepackage{caption}
\usepackage{subcaption}
\usepackage{url} 
\usepackage[hidelinks]{hyperref}
\usepackage{amsfonts,amsthm,latexsym,amsmath,amssymb,amscd,amsmath,epsf,a4wide} 
\usepackage{verbatim}
\usepackage{tikz}
\usetikzlibrary{arrows,shapes,matrix}
\usetikzlibrary{shapes.geometric}
\usetikzlibrary{positioning}
\usetikzlibrary{matrix,arrows}
\usetikzlibrary{through}
\usetikzlibrary{patterns}
\usetikzlibrary{calc}

\theoremstyle{plain}
\newtheorem{theorem}{\bf Theorem}[section]
\newtheorem{lemma}[theorem]{\bf Lemma}
\newtheorem{proposition}[theorem]{\bf Proposition}
\newtheorem{corollary}[theorem]{\bf Corollary}

\newtheorem{idea}{\bf Idea}

\theoremstyle{definition}
\newtheorem{definition}[theorem]{Definition}

\theoremstyle{remark}
\newtheorem{remark}[theorem]{Remark}

\numberwithin{equation}{section}

\makeindex

\newcommand{\PP}{\mathbb{P}}
\newcommand{\NN}{\mathbb{N}}
\newcommand{\II}{\mathbb{I}}

\newcommand{\K}{\mathcal{K}}
\newcommand{\kk}{\mathbf{k}}

\newcommand{\onetotwo}{\xi}

\newcommand{\T}{\mathcal{T}}
\newcommand{\DT}{\widehat{\mathcal{T}}}
\newcommand{\A}{\mathcal{A}}
\newcommand{\B}{\mathcal{B}}
\newcommand{\C}{\mathcal{C}}
\renewcommand{\P}{\mathcal{P}}
\renewcommand{\S}{\mathcal{S}}

\newcommand{\SP}{\mathcal{SP}}
\newcommand{\NC}{\mathcal{NC}}
\newcommand{\CM}{\mathcal{CM}}
\newcommand{\PT}{\mathcal{PT}}
\newcommand{\IT}{\mathcal{IT}}

\newcommand{\DDP}{\widehat{\mathcal{DP}}}
\newcommand{\DP}{\mathcal{DP}}
\newcommand{\Dyck}{\mathcal{D}yck}
\newcommand{\AP}{\mathcal{AP}}

\newcommand{\PF}{\mathcal{PF}}

\newcommand{\intvert}{\mathnormal{IntVert}}

\newcommand{\IdentityPath}{D^\emptyset}
\newcommand{\wcomp}{\mathcal{WC}omp}
\newcommand{\comp}{\mathcal{C}omp}
\newcommand{\Par}{\mathcal{P}ar}
\newcommand{\sym}{\mathfrak{S}}

\DeclareMathOperator{\leftDes}{left} 

\def\newop#1{\expandafter\def\csname #1\endcsname{\mathop{\rm #1}\nolimits}}

\newop{length}
\newop{signat}
\newop{caverns}
\newop{Area}
\newop{area}
\newop{dom}
\newop{child}
\newop{indeg}

\title[]{signature Catalan combinatorics}

\author[C. Ceballos]{Cesar Ceballos}
\address{Faculty of Mathematics, University of Vienna, Vienna, Austria}
\email{cesar.ceballos@univie.ac.at}

\author[R. S. Gonz\'alez D'Le\'on]{Rafael S. Gonz\'alez D'Le\'on}
\address{Escuela de Ciencias Exactas e Ingenier\'ia, Universidad Sergio Arboleda, Bogot\'a, 
Colombia}
\email{rafael.gonzalezl@usa.edu.co}

\thanks{C. Ceballos was supported by the Austrian Science Foundation FWF, grant F 5008-N15, in the framework
of the Special Research Program Algorithmic and Enumerative Combinatorics"; he was also partially supported by York University and a Banting Postdoctoral Fellowship of the Government of Canada. R.~S.~Gonz\'alez~D'Le\'on was supported during this project by University of Kentucky, York University and Universidad Sergio Arboleda and he is grateful for their support.}

\begin{document}
\begin{abstract}
The Catalan numbers constitute one of the most important sequences in combinatorics. Catalan objects have been generalized in various directions, including the classical Fuss-Catalan objects and the rational Catalan generalization of Armstrong-Rhoades-Williams. We propose a wider generalization of these families indexed by a composition $s$ which is motivated by the combinatorics of planar rooted trees; when $s=(2,...,2)$ and $s=(k+1,...,k+1)$ we recover the classical Catalan and Fuss-Catalan combinatorics, respectively. Furthermore, to each pair $(a,b)$ of relatively prime numbers we can associate a signature that recovers the combinatorics of rational 
Catalan objects.
We present explicit bijections between the resulting $s$-Catalan objects, and a fundamental recurrence that generalizes the fundamental recurrence of the classical Catalan numbers.
Our framework allows us to define signature generalizations of parking functions which coincide with the generalized parking functions studied by Pitman-Stanley and Yan, as well as generalizations of permutations which coincide with the notion of Stirling multipermutations introduced by Gessel-Stanley. 
Some of our constructions differ from the ones of Armstrong-Rhoades-Williams, however as a byproduct of our extension, we obtain the additional notions of rational permutations and rational trees.
\end{abstract}

\maketitle
\tableofcontents
\section{Introduction}\label{section:introduction}

A \emph{permutation} $\sigma$ of the set $[n]:=\{1,2,\dots,n\}$ is a bijection 
$\sigma:[n]\rightarrow [n]$. A permutation can be represented in the one-line notation 
$\sigma_1\sigma_2\cdots \sigma_n$ where $\sigma_i:=\sigma(i)$. We denote the 
set of permutations 
of $[n]$ by $\sym_n$. This set has the structure of a group under composition of permutations and 
it is commonly known as the \emph{symmetric group}. It is an introductory exercise in enumerative
combinatorics to show that the cardinality of  $\sym_n$ is given by the factorial numbers 
$n!:=1\cdot 2 \cdots (n-1)\cdot n$. 
There are other families of objects that are in bijection with permutations. For example, there 
is a bijection between $\sym_n$  and the set~$\IT_n$ of binary trees drawn in the plane with a distinguished vertex or \emph{root} and with 
labels in the internal nodes that are increasing 
when walking away from the root. The trees in~$\IT_n$ are also known as 
\emph{increasing rooted planar binary trees} (see \cite[Chapter 1]{Stanley2012} for the bijection 
and Figure \ref{figure:example_type_A_combinatorics} for an example of the bijection when $n=3$). 
We will be denoting by $\S_n$ a generic family of objects that is in bijection with $\sym_n$, that 
is, $|\S_n|=n!$.

We say that a permutation $\sigma$ is 
\emph{$312$-avoiding} if there are no indices $i<j<k$ such that $\sigma_j < \sigma_k < \sigma_i$. 
We denote by $\sym_n(312)$ the set of $312$-avoiding permutations in $\sym_n$. It is known  that 
the set $\sym_n(312)$ is in bijection with the set $\DP_n$ of lattice paths in $\NN\times\NN$ from 
$(0,0)$ to $(n,n)$ taking only north steps $(0,1)$ and east steps $(1,0)$, and such that at 
any given point the number of north steps taken is greater than or equal to the number of east steps 
taken (see Figure \ref{figure:example_type_A_combinatorics} for an example of the bijection when 
$n=3$). These lattice paths are also known as \emph{Dyck paths}. Any family of objects in bijection 
with either $\sym_n(312)$ or~$\DP_n$ is known as a Catalan family named after the Belgian 
mathematician Eug\`{e}ne Charles Catalan who studied them. Catalan objects are probably the most 
intriguing objects in combinatorics since they appear in connection with numerous fields of 
mathematics in many different forms. Many mathematicians have studied 
enumerative, geometric and algebraic occurrences of the Catalan families. In particular, Richard 
Stanley has collected and curated a selection of these families during years of study 
\cite{Stanley2015}. Let us denote by $\C_n$ a generic family of Catalan objects. It is known 
that $|\C_n|=\frac{1}{n+1}\binom{2n}{n}=\frac{1}{2n+1}\binom{2n+1}{n}=\frac{(2n)!}{(n+1)!n!}$ what 
is known as the $n$-th 
\emph{Catalan number}.

From objects in $\DP_n$ we can create a different family $\DDP_n$ of objects obtained by decorating the north steps of a Dyck path with a permutation of $[n]$ in such a way that 
the consecutive sequences of north steps (also known as \emph{runs}) have increasing values from 
bottom to top (see Figure \ref{figure:example_type_A_combinatorics}). The objects obtained in this 
way are known as \emph{decorated Dyck paths} and happen to be in bijection with another famous family of combinatorial objects
known as 
\emph{parking functions}. The set  $\PF_n$ of parking functions of $n$ consists of sequences of $n$ 
nonnegative integers $(p_1,p_2,\dots,p_n)$ with the property that after being rearranged in weakly 
increasing order $p_{j_1}\le p_{j_2}\le \cdots \le p_{j_n}$ they satisfy $p_{j_i}<i$ for all 
$i\in[n]$. Figure \ref{figure:example_type_A_combinatorics} shows an example of the bijection 
between decorated Dyck paths and parking functions for $n=3$. Parking functions were studied by 
Konheim and Weiss \cite{KonheimWeiss1966} using a different description that explains their name. 
There are other popular families of objects in bijection with $\PF_n$, for example, the 
set of (nonplanar nonrooted) trees on the vertex set $\{0\}\cup [n]$ is in bijection with $\PF_n$ (see 
\cite{DeOliveiraLasVergnas2010,Kreweras1980,PerkinsonYangYu2016}). Let us denote any generic family 
of objects in bijection with $\PF_n$ by $\P_n$. It is also known that $|\P_n|=(n+1)^{n-1}$, see for 
example \cite{Stanley1997}.

The families counted by the sequences $n!$, $\frac{1}{n+1}\binom{2n}{n}$ and $(n+1)^{n-1}$ for 
$n\ge0$ have intrigued mathematicians for centuries and in a nutshell intuitively we can describe 
their relation by the following idea.

\begin{idea}\label{idea:firstidea}
A family $\C_n$ is obtained as a sub-family of $\S_n$ satisfying a suitable restriction. A family $\P_n$ can be obtained by extending a family $\C_n$ with a decoration of its objects with the elements of $[n]$, also subject to a suitable restriction on the decoration.
\end{idea}

 \begin{figure}
\centering
\input{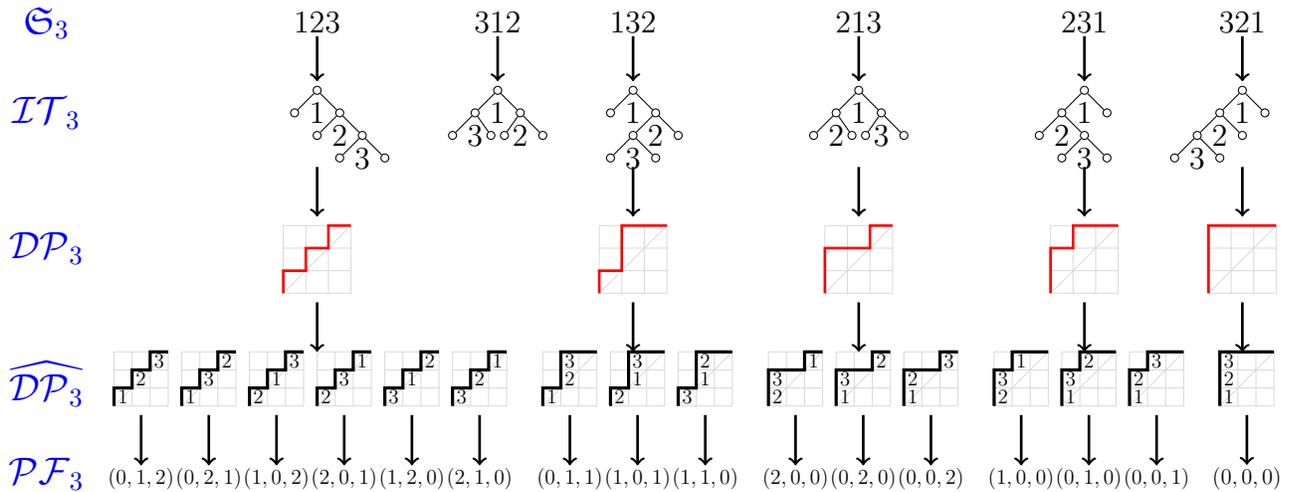}
\caption{Maps between classical combinatorial objects in type A}
\label{figure:example_type_A_combinatorics}
\end{figure}

Idea \ref{idea:firstidea} is the underlying idea in the story of the classical combinatorial 
objects mentioned above. As mentioned in \cite{ArmstrongRhoadesWilliams2013}, there are two general directions 
in which we can generalize this story. The first considers that all the families of objects 
discussed, $\S_n$, $\C_n$ and $\P_n$, are related directly to the combinatorics of the symmetric 
group that happens to be the Weyl group of Coxeter type A. Hence, we can wonder if there are 
corresponding objects for other Coxeter types. Work in this direction has been done by some 
authors like Reiner~\cite{Reiner1997}, Athanasiadis~\cite{Athanasiadis2005}, Fomin and 
Reading~\cite{FominReading2005}, Armstrong~\cite{Armstrong2009}, Williams~\cite{Willams2013} and 
others, and it is currently an active area of research.
The second direction to generalize these families of objects is to consider Catalan 
objects as a phenomenon that depends on two relatively prime numbers $a,b$ such that a family 
$\C_{a,b}$ parametrized by these numbers has cardinality $\frac{1}{a+b}\binom{a+b}{b}$, also known 
as the 
\emph{rational Catalan number}. When $a=n$ and $b=n+1$ we recover the classical Catalan numbers,
and when $a=n$ and $b=kn+1$ we obtain another classical generalization known as the 
\emph{Fuss-Catalan numbers}. Even thought ingredients of the rational Catalan story have appeared 
previously in the literature, Armstrong, Rhodes and Williams started a more systematic study of the 
rational Catalan objects $\C_{a,b}$ in \cite{ArmstrongRhoadesWilliams2013}. In this generalization 
of the classical story there are rational versions $\P_{a,b}$  of the parking objects in 
$\P_n$ (see for example \cite{ArmstrongLoehrWarrington2016}),
but it is not clear what is the generalization for the permutation objects in $\S_n$ to a rational 
version $\S_{a,b}$.

In this paper we start a systematic study of a further generalization of the rational Catalan objects. Our generalization is indexed by a composition $s=(s_1,s_2,\dots,s_{a})$ that we call a \emph{signature}. 
When $s=(2,\dots,2)$ and $s=(k+1,\dots,k+1)$ we recover the classical Catalan and Fuss-Catalan objects respectively. The rational Catalan objects are obtained by the signature whose $i$-th entry is the number of boxes in the $i$-th row of an $b \times a$ grid that are crossed by the main diagonal of the grid. 

The central idea for our generalization of Catalan objects relies on the combinatorics of planar rooted trees. For a given 
signature $s$ we associate a family $\T_s$ of planar rooted trees that we call 
\emph{$s$-trees}. Based on their structural properties one can obtain generalizations of other classical Catalan objects such as Dyck paths, $312$-avoiding permutations, noncrossing partitions, complete noncrossing matchings, triangulations of a polygon, and parenthesizations.
We also propose generalizations of permutations and parking functions objects in this general set up. 
The generalized permutations are encoded by a family of increasingly labeled planar rooted trees. These labeled trees are in bijection with a generalization of the Stirling permutations studied by Gessel and Stanley back in the seventies~\cite{GesselStanley1978} (see also 
\cite{Park1994-1,Park1994-2,Park1994-3,GrahamKnuthPatashnik1994,JansonKubaPanholzer2011,KubaPanholzer2011,NovelliThibon2014,Dleon2015,RemmelWilson2015}).  
The generalized parking functions are obtained as appropriate decorations of $s$-Catalan objects as it is classically done in the case of rational parking functions~\cite{ArmstrongLoehrWarrington2016}. These also coincide with the ones originally studied by Pitman and Stanley~\cite{StanleyPitman2002} and Yan~\cite{Yan2000,Yan2001} in a slightly more general form.

\section{Preliminaries}
\subsection{Compositions and weak compositions}
We denote by $\PP$ the set of positive integers and by $\NN$ the set of nonnegative integers. A \emph{weak composition} is a finite sequence 
$\mu=(\mu(1),\mu(2),\dots, \mu(\ell))$ of numbers $\mu(i) \in \NN$. For a weak composition $\mu$ we 
define its \emph{sum} $|\mu|:=\sum_i \mu(i)$ and its \emph{length} $\ell(\mu):=\ell$. For example 
for $\mu=(2,0,3,4,0,1)$ we have that $\ell(\mu)=6$ and $|\mu|=10$.
If $|\mu|=n$ for some $n \in \NN$, we say that $\mu$ is a \emph{weak composition of $n$}. We denote 
by $\wcomp$ the set of weak compositions and $\wcomp_n$ the set of weak compositions of $n$. 
A \emph{composition} is a weak composition $\mu$ such that $\mu(i)\ne 0$ for all $i\in [\ell]$, in 
other words, a composition is a finite sequence of entries in $\PP$. We denote by $\comp$ the set 
of 
compositions and $\comp_n$ the set of compositions of~$n$. 
An \emph{(integer) partition} $\lambda$ of $n$ (denoted $\lambda \vdash n$) is a composition of $n$ 
whose entries are nonincreasing, i.e., $\lambda=(\lambda(1) \ge \lambda(2) \ge \cdots)$. We denote 
by $\Par$ the set of partitions and $\Par_n$ the set of partitions of $n$.

We can define several partial orders on the set of (weak) compositions. We introduce the two 
orderings that we will be using in this article.
For  $\mu,\nu \in \wcomp$ we say that $\mu$ is a \emph{refinement} of $\nu$ if $\nu$ can be 
obtained 
from $\mu$ by adding adjacent parts. For example, $(3,2,1,1,5,2,2)$ is a refinement of $(6,1,7,2)$ 
since $(6,1,7,2)=(3+2+1,1,5+2,2)$. Refinement defines a partial order in $\wcomp$ (also in 
$\comp$) and we say that 
$\mu\le\nu$ if $\mu$ is a refinement of $\nu$. For $\mu, \nu \in \wcomp$ we say that $\mu 
\le_{\dom} \nu$ in \emph{dominance order} if $\mu(1)+\mu(2)+\cdots+\mu(i) \le 
\nu(1)+\nu(2)+\cdots+\nu(i)$ for all $i$.  For example $(1,1,4,2) \le_{\dom} (1,2,3,3)$ 
since $1\le 1$, $1+1\le1+2$, $1+1+4\le1+2+3$ and $1+1+4+2\le1+2+3+3$. For $\mu, \nu\in\wcomp_n$ for 
some $n$, such 
that $\mu \le_{\dom} \nu$, the \emph{dominance difference} $\nu\setminus_{\dom}\mu$ of 
$\nu$ 
and $\mu$ is the weak composition of length $\ell(\nu)-1$ defined
$(\nu\setminus_{\dom}\mu)(i):=(\nu(1)+\cdots+\nu(i))- (\mu(1)+\cdots+\mu(i))$ for 
$i=1,\dots,\ell(\nu)-1$. 
We also consider a natural operation in $\wcomp$. For $\mu,\nu 
\in \wcomp$ we let $\mu\oplus\nu$ be the weak composition formed by the \emph{concatenation} of 
$\mu$ and 
$\nu$. For example,  $(3,2,1)\oplus(1,5,2,2)=(3,2,1,1,5,2,2)$. Sometimes we will use the notation
$\mu^{+i}$ or $\mu^{-i}$ to denote the weak composition obtained from $\mu$ by adding or 
substracting $i$ to the last part of $\mu$ respectively.

\subsection{Catalan objects}
The Catalan numbers constitute one of the most important sequences in combinatorics:
\[
1, 1, 2, 5, 14, 42, 132, 429, 1430, 4862, 16796, 58786, \dots
\]
The $n$-th term of this sequence is commonly denoted by $C_n$. It is given by the simple formula $C_n= \frac{1}{n+1}\binom{2n}{n}=\frac{1}{2n+1}\binom{2n+1}{n}=\frac{(2n)!}{(n+1)!n!}$, and is completely determined by the fundamental recurrence 
\[
C_{n+1}=\sum_{k=0}^nC_{k}C_{n-k}, \quad C_0=1.
\]
The Catalan numbers are known to count a great variety of objects in mathematics~\cite{Stanley2015}. Among the most remarkable ones are:
\begin{enumerate}
\item planar binary trees with $n$ internal nodes;
\item Dyck paths in an $n\times n $ grid;
\item $312$-avoiding permutations of $[n]$;
\item noncrossing partitions of $[n]$;
\item noncrossing matchings of $[2n]$;
\item triangulations of an $(n+2)-gon$;
\item parenthesizations of $n+1$ consecutive letters.
\end{enumerate}

Some of these families of combinatorial objects have been generalized in various directions. This includes families that are counted by the Fuss-Catalan numbers $\frac{1}{mn+1}\binom{(m+1)n}{n}$, or more generally, by the rational Catalan numbers $C_{a,b}=\frac{1}{a+b}{a+b \choose a}$. 
 
The main purpose of this paper is to start a systematic study of a wider generalization indexed by a composition $s=(s_1,s_2, \dots,s_a)$. The members of these generalized families will be referred to as $s$-Catalan objects.

\subsection{$s$-Catalan objects}
The $s$-Catalan objects corresponding to the Catalan families mentioned above are described very naturally in terms of any composition $s$, which encodes certain combinatorial information. For instance, planar binary trees can be generalized to arbitrary planar rooted trees, where the signature encodes the number of children of each internal node of the tree when read in certain order (preorder). Another example is that of Dyck paths, which can be generalized to lattice paths that lie weakly above a certain ribbon shape determined by $s$. 
The precise definitions and motivations for the generalized $s$-Catalan families will be presented in the coming sections. As an insight, some examples of the resulting combinatorial objects are illustrated in Figure~\ref{fig_sCatalanZoo}. 

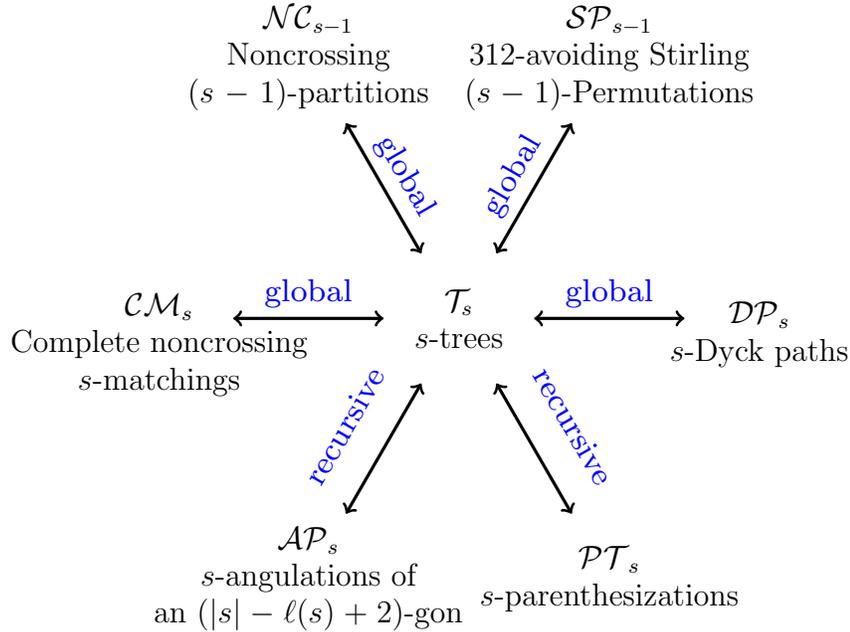
\begin{figure}[htb]
\begin{tikzpicture}[block_center/.style ={text width=0.3\textwidth, text centered}]

\tikzstyle{every node}=[circle,inner sep=-6pt, minimum width=0pt]

\node [block_center] (permutations) at (60:4) { $\SP_{s-1}$\\$312$-avoiding Stirling $(s-1)$-Permutations};
\node [block_center,inner sep=-26pt] (trees) at (0:0) { $\T_s$\\$s$-trees};
\node [block_center] (paths) at (-3:4){ $\DP_s$\\$s$-Dyck paths};
\node [block_center] (matchings)at (185:4) {$\CM_{s}$\\Complete noncrossing \\$s$-matchings };
\node [block_center] (partitions) at (120:4){ $\NC_{s-1}$\\Noncrossing \\$(s-1)$-partitions}; 
\node [block_center] (angulations) at (240:4){ $\AP_{s}$\\$s$-angulations of\\an 
$(|s|-\ell(s)+2)$-gon}; 
\node [block_center] (parenthesizations) at (300:4){ $\PT_{s}$\\$s$-parenthesizations}; 

\draw [<->,very thick] (300:1)--(300:3)  node [midway, above, sloped,scale=1.5] (TextNode) {\color{blue} \tiny recursive};
\draw [<->,very thick] (180:1)--(180:3)  node [midway, above, sloped,scale=1.5] (TextNode) {\color{blue} \tiny global};
\draw [<->,very thick] (0:1)--(0:3)  node [midway, above, sloped,scale=1.5] (TextNode) {\color{blue} \tiny global};
\draw [<->,very thick] (120:1)--(120:3) node [midway, above, sloped,scale=1.5] (TextNode) {\color{blue} \tiny global};
\draw [<->,very thick] (240:1)--(240:3)  node [midway, above, sloped,scale=1.5] (TextNode) {\color{blue} \tiny recursive};
\draw [<->,very thick] (60:1)--(60:3) node [midway, above, sloped,scale=1.5] (TextNode) {\color{blue} \tiny global};

\end{tikzpicture}
\caption{Bijections between $s$-Catalan objects.}
\label{fig:bijections}
\end{figure}

\begin{figure}[htbp]
\begin{center}
\begin{tikzpicture}[scale=1]
\draw (0,0) grid[step=(30:10)] (17.35,15);
\draw (0,15)--(0,20)--(17.35,20)--(17.35,15);
\tikzstyle{every node}=[scale=0.7]
\node[color=red,scale=1.5] () at (9,20.5) {\Large \bf The $s$-Catalan Zoo};

\node[color=blue,align=left,scale=1.2] () at (12,15.5) {\Large \bf $s$-Trees};

\node[color=blue,align=left] () at (7,10.5) {\Large \bf $s$-Dyck paths};

\node[color=blue,align=left,text width=132] () at (7,5.5) {\Large \bf Noncrossing $(s-1)$-partitions};

\node[color=blue,align=left,text width=200] () at (15,5.5) {\Large \bf Complete noncrossing $s$-matchings};

\node[color=blue,align=left,text width=180] () at (7,0.5) {\Large \bf  $s$-Angulations of a polygon};


\node[color=blue,align=left] () at (15,0.5) {\Large \bf  $s$-Parenthesizations};

\node[color=blue,align=left,text width=180] () at (14.5,10.5) {\Large \bf  $312$-avoiding Stirling $(s-1)$-permutations};
\begin{scope}[xshift=225,yshift=470,thick,scale=0.4]
\tikzstyle{every node}=[circle, draw,
                        inner sep=0.5pt, minimum width=4pt,font=\small]
\node (a1) at (2,6){};
\node (b1) at (-1,4){};
\node (b2) at (2,4){};
\node (b3) at (5,4){};

\node (c1) at (-2.5,2) {};
\node (c2)  at (-1.5,2){};
\node (c3)  at (-0.5,2){};
\node (c4) at (0.5,2){};
\node (d1)  at  (4.5,2){};
\node (d2) at (5.5,2) {};

\node (f1) at (-2,0){};
\node (f2)  at (-1,0){};
\node (f3)  at (0,0){};
\node (f4) at (1,0){};
    
\node (e1) at (2.5,0){};
\node (e2)  at (3.5,0){};
\node (e3)  at (4.5,0){};
\node (e4) at (5.5,0){};
\node (e5) at (6.5,0){};

\draw (a1)--(b1);
\draw (a1)--(b2);
\draw	 (a1)--(b3);
\draw (b1)--(c1);
\draw (b1)--(c2);
\draw (b1)--(c3);
\draw (b1)--(c4);
\draw (b3)--(d1);
\draw (b3)--(d2);
\draw (c3)--(f1);
\draw (c3)--(f2);
\draw (c3)--(f3);
\draw (c3)--(f4);
\draw (d1)--(e1);
\draw (d1)--(e2);
\draw (d1)--(e3);
\draw (d1)--(e4);
\draw (d1)--(e5);

\end{scope}


\begin{scope}[xshift=45,yshift=330,scale=0.4]
\edef \signature{3,4,4,2,5}

\def\row{0}
\def\col{0}

\foreach \part in \signature{
\pgfmathsetmacro \newcol {\col+\part-1}
\pgfmathparse{\col+\part-1}
\global\let\newcol\pgfmathresult
\pgfmathparse{\row+1}
\global\let\newrow\pgfmathresult
\draw[fill, color=gray!10] (\col,\row) rectangle (\newcol+1,\newrow);
\global\let\row\newrow
\global\let\col\newcol
}
\draw[very thin] (0, 0) grid (\col+1, \row);

\def\rowpa{0}
\def\colpa{0}
\edef \pa{0,2,6,0,6}
\foreach \var in \pa{
\pgfmathparse{\colpa+\var}
\global\let\newcolpa\pgfmathresult
\pgfmathparse{\rowpa+1}
\global\let\newrowpa\pgfmathresult
\draw [line width=3, color=red]
(\colpa, \rowpa)--(\colpa,\newrowpa)--(\newcolpa,\newrowpa);

\global\let\rowpa\newrowpa
\global\let\colpa\newcolpa
}


\end{scope}
\begin{scope}[xshift=120,yshift=220,scale=0.5]
 
   \tikzstyle{every node}=[inner sep=0pt, minimum width=4pt]
   \def\n{13}
   
   \def\inicio{3} 
   \def\col{blue}
      \path[fill=\col!50] (90-360/\n*\inicio+360/\n+1:2.8cm)
      \foreach \x in {3,4,5}{ -- (90-360/\n*\x+360/\n+1:2.8cm)};
      
    \def\inicio{1} 
    \def\col{red}
      \path[fill=\col!50] (90-360/\n*\inicio+360/\n+1:2.8cm)
      \foreach \x in {1,2,6,7,8}{ -- (90-360/\n*\x+360/\n+1:2.8cm)};
      
      \def\inicio{9} 
    \def\col{gray}
      \path[fill=\col!50] (90-360/\n*\inicio+360/\n+1:2.8cm)
      \foreach \x in {9,10,11,12,13}{ -- (90-360/\n*\x+360/\n+1:2.8cm)};
   
   \foreach \x in {1,2,...,13}
   \draw {(90-360/\n*\x+360/\n+1:3cm) node {\x}};
      
\end{scope}
\begin{scope}[xshift=370,yshift=220,scale=0.5]
 
   \tikzstyle{every node}=[inner sep=0pt, minimum width=4pt]
   \def\n{18}
   
   \def\inicio{2} 
   \def\col{blue}
      \path[fill=\col!50] (90-360/\n*\inicio+360/\n+1:2.8cm)
      \foreach \x in {2,3,4,9}{ -- (90-360/\n*\x+360/\n+1:2.8cm)};
      
    \def\inicio{1} 
    \def\col{red}
      \path[fill=\col!50] (90-360/\n*\inicio+360/\n+1:2.8cm)
      \foreach \x in {1,10,11}{ -- (90-360/\n*\x+360/\n+1:2.8cm)};
      
      \def\inicio{5} 
    \def\col{gray}
      \path[fill=\col!50] (90-360/\n*\inicio+360/\n+1:2.8cm)
      \foreach \x in {5,6,7,8}{ -- (90-360/\n*\x+360/\n+1:2.8cm)};
      
      \def\inicio{12} 
    \def\col{brown}
      \path[draw,very thick, fill=\col!50,\col!50] (90-360/\n*\inicio+360/\n+1:2.8cm)
      \foreach \x in {12,18}{ -- (90-360/\n*\x+360/\n+1:2.8cm)};

 \def\inicio{13} 
    \def\col{green}
      \path[fill=\col!50] (90-360/\n*\inicio+360/\n+1:2.8cm)
      \foreach \x in {13,14,15,16,17}{ -- (90-360/\n*\x+360/\n+1:2.8cm)};

   \foreach \x in {1,2,...,\n}
   \draw {(90-360/\n*\x+360/\n+1:3cm) node {\x}};
   
\end{scope}

\begin{scope}[xshift=120,yshift=75,thick,scale=0.5]
 
   \tikzstyle{every node}=[inner sep=0pt, minimum width=4pt,scale=1.6]
   \def\n{15}
   
   \foreach \x in {1,2,...,\n}
   \draw {(90-360/\n*\x+360/\n+1:3cm) node (n\x)[label={90-360/\n*\x+360/\n+1:\tiny \x}]{}};

   
\draw[dashed, thick,red] (n1) -- (n2)node [midway, above, sloped,red] (TextNode) {$e$};
\draw (n2) -- (n3);
\draw (n3) -- (n4);
\draw (n4) -- (n5);
\draw (n5) -- (n6);
\draw (n6) -- (n7);
\draw (n7) -- (n8);
\draw (n8) -- (n9);
\draw (n9) -- (n10);
\draw (n10) -- (n11);
\draw (n11) -- (n12);
\draw (n12) -- (n13);
\draw (n13) -- (n14);
\draw (n14) -- (n15);
\draw (n15) -- (n1);
\draw (n1) -- (n9);
\draw (n2) -- (n8);
\draw (n3) -- (n8);
\draw (n10) -- (n14);
\end{scope}


\begin{scope}[xshift=370,yshift=70,thick,scale=0.32]
\node at (0,0) {\LARGE \bf $(\star\star(\star\star\star\star)\star)\star((\star\star\star\star\star)\star)$};
\end{scope}

\begin{scope}[xshift=370,yshift=360,thick,scale=0.32]
\node at (0,0) {\Huge \bf $2233321155554$};
\end{scope}
\end{tikzpicture}

\caption{Examples of $s$-Catalan objects for $s=(3,4,4,2,5)$.}

\label{fig_sCatalanZoo}
\end{center}
\end{figure}

As we shall see, these generalized families have the same intrinsic combinatorial structure (Sections~\ref{sec_sCatalanzooAndBijections} and~\ref{sec_fundamentalRecurrence}). They are all counted by the same determinantal formula (Section~\ref{sec_enumeration}), and are determined by the same fundamental recurrence of $s$-Catalan structures: 
\[
C_s=\sum_{(s)}C_{s_1}C_{s_2}\cdots C_{s_{s(1)}}, 
\]
where the sum is over all sequences $(s_1,s_2,\dots,s_{s(1)})$ of compositions such that 
$s=(s(1))\oplus s_1\oplus s_2\oplus\cdots \oplus s_{s(1)}$ and where $C_{\emptyset}=1$ (see Section~\ref{sec_fundamentalRecurrence}).

The classical Catalan and Fuss-Catalan combinatorics can be recovered when $s=(2,2,\dots , 2)$ and $s=(k+1,k+1,\dots, k+1)$ respectively. The rational Catalan combinatorics corresponds to the signature $s$, where $s(i)$ is the number of boxes in row $i$ in a~$b\times a$ grid that are crossed by the main diagonal. 

One of our main results is to show that the generalized $s$-Catalan families are bijectively equivalent (see Figure 
\ref{fig:bijections}).

\begin{theorem}\label{thm_sCatalanFamiliesEquivalent}
The following generalized familes are in bijective correspondence:
 \begin{enumerate}
\item $s$-trees; \label{item_trees}
\item $s$-Dyck paths; \label{item_Dyckpaths}
\item $312$-avoiding Stirling $(s-\bf{1})$-permutations; \label{item_permutations}
\item noncrossing $(s-\bf{1})$-partitions; \label{item_noncrossingpartitions}
\item Complete noncrossing $s$-matchings; \label{item_matchings}
\item $s$-angulations of an $(|s|-\ell(s)+2)$-gon; \label{item_angulations}
\item $s$-parenthesizations; \label{item_parenthesizations}
\end{enumerate}
In items (\ref{item_permutations}) and (\ref{item_noncrossingpartitions}) it is assumed that $s(i)\geq 2$ for all $i$.
\end{theorem}

The bijections between $s$-trees~\eqref{item_trees} and items~\eqref{item_Dyckpaths}-\eqref{item_matchings} are presented in Section~\ref{sec_sCatalanzooAndBijections}. The idea behind these bijections is always essentially the same: we assign some labels to the tree and then read them in some order (preorder). Figure~\ref{fig_bijections_sTrees} presents a quick descriptive illustration. We include detailed proofs explaining facts about these bijections for completeness, but some readers may like to skip some of the proofs and convince themselves from the description in Figure~\ref{fig_bijections_sTrees}. 
The bijections between $s$-trees~\eqref{item_trees} and items~\eqref{item_angulations}-\eqref{item_parenthesizations} are explained using a fundamental recurrence in Section~\ref{sec_fundamentalRecurrence}.
Bijections with $s$-Dyck paths are explained in the Appendix~\ref{sec_bijectionsWithDyckPaths}, see Figure~\ref{fig_bijections_sDyckpaths} for a descriptive illustration. 

\begin{figure}[htbp]
\begin{center}
\input{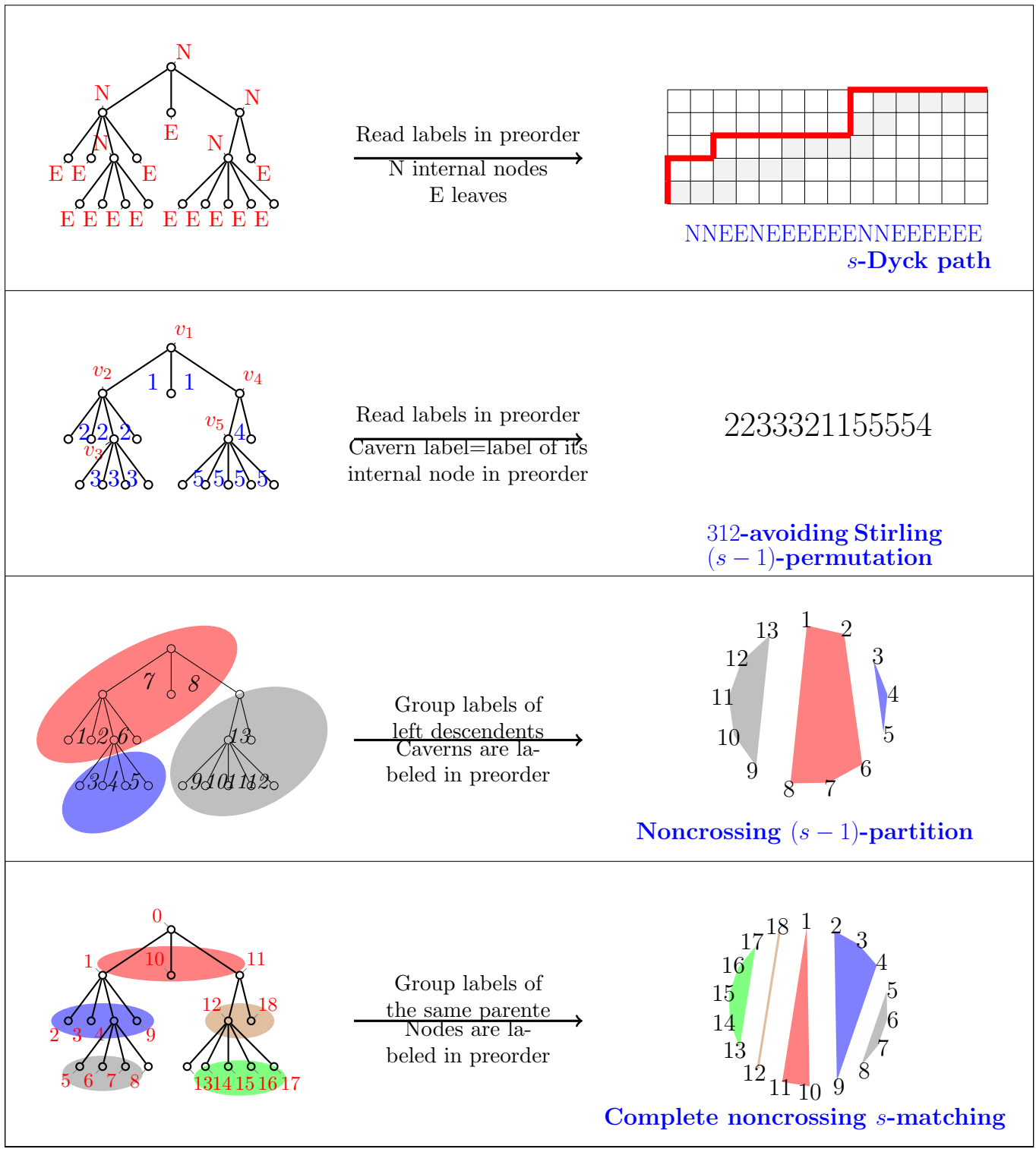}
\caption{Examples of the bijections with $s$-trees.}
\label{fig_bijections_sTrees}
\end{center}
\end{figure}

\section{The $s$-Catalan zoo and bijections}
\label{sec_sCatalanzooAndBijections}

\subsection{Planar rooted trees}\label{section:}

A tree is a graph that has no loops or cycles. We say that a tree is 
rooted if 
one of its nodes is specially marked and called the \emph{root}. For two nodes $x$ and
$y$ on a rooted tree $T$, $x$ is said to be the \emph{parent} of $y$ (and $y$ the 
\emph{child} of $x$) if $x$ is the node that follows $y$ in the unique path from $y$ to the 
root. We call $y$ a \emph{descendant} of~$x$ if~$x$ belongs to the unique path from $y$ to the root. A node is called a \emph{leaf} if it has no children, otherwise is said to be 
\emph{internal}. The \emph{degree} $\deg(x)$ of a node $x$ in a rooted tree $T$ is defined to be 
the cardinality of the set of children of $x$. A rooted tree $T$ is said to be planar if for every 
internal node $x$ of $T$ the 
set of children of $x$ is totally ordered. For $n \ge 1$ we denote by $\T_n$ the set of all planar 
rooted trees with $n$ internal nodes and by $\T$ the set of all planar rooted trees. There is a 
unique tree $\bullet$ that has only one node that is at the same time its leaf and its root, we 
call this tree the \emph{identity tree}. For trees $T_1,T_2,\dots,T_k \in \T$ with roots 
$r_1,r_2,\dots,r_k$ respectively, we denote by $[T_1,T_2,\dots,T_k]$ the tree in $\T$ 
constructed by adding a new root $r$ and attaching the trees $T_1,T_2,\dots,T_k$ to $r$ in a way 
that the 
order of the children of $r$ is given by $r_i$ being the $i$th children for $i \in [k]$ (see 
Figure~\ref{fig:examplewedge}). 
We are going to fix an order in which to read the nodes of a tree $T=[T_1,T_2,...,T_k] \in \T$. 
Let $r$ be the root of $\T$, we say that $T$ is read or traversed in \emph{preorder} if we visit 
first the root $r$ and then we visit in preorder each of the trees 
$T_i$ in a way that $T_i$ is traversed before $T_j$ if $i<j$ for $i,j \in [k]$.

\begin{figure}[h]
 \centering

\begin{tikzpicture}[scale=0.8]

\tikzstyle{every node}=[draw,scale=1]

    \draw [circle,inner sep=1pt,color=black] (2,2)  node (r){$r$};
  
\tikzstyle{every node}=[inner sep=1pt, minimum width=14pt,scale=1]

    \draw (0,0)  node (t1){$T_1$};
    \draw (1,0)  node (t2){$T_2$};
    \draw (3,0)  node (t3){$T_{k-1}$};
    \draw (4,0)  node (t4){$T_{k}$};
    
    \draw (r) --  (t1) ;
    \draw (r) --  (t2) ;
    \draw (r) --  (t3) ;
    \draw (r) --  (t4) ;
    \draw [dotted] (1.5,0.5) --  (2.5,0.5) ;

\end{tikzpicture}

\caption{Example of $[T_1,T_2,\dots, T_{k}]$}
\label{fig:examplewedge}
\end{figure}
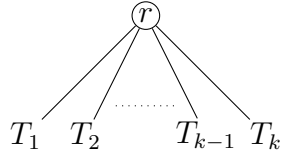

\noindent
 Let $v_1,v_2,\dots,v_a$ be the internal nodes of $T$ listed in preorder, we define the 
\emph{signature} of $T$ to be the composition 
$$\signat(T):=(\deg(v_1),\deg(v_2),\dots,\deg(v_a)).$$
The signature of the identity tree is defined 
as the empty composition, i.e., $\signat(\bullet):=\emptyset$.

\begin{definition}[$s$-trees]
\label{def_sTrees}
For any $s \in \comp$, 
an $s$-tree is a planar rooted tree with signature~$s$.
We denote by 
\[\T_s=\{T\in \T \, \mid \, \signat(T)=s\}\]
the set of \emph{$s$-trees}. 
\end{definition}

Figure~\ref{fig:examplestree} illustrates an $s$-tree $T$ with signature $s=(3,4,4,2,5)$.

\begin{figure}
\centering
\begin{tikzpicture}[thick,scale=0.6]
\tikzstyle{every node}=[circle, draw,
                        inner sep=0.5pt, minimum width=4pt,font=\small]
\node (a1) at (2,6)[pin={[color=red]135:$v_1$}]{};
\node (b1) at (-1,4)[pin={[color=red]135:$v_2$}]{};
\node (b2) at (2,4){};
\node (b3) at (5,4)[pin={[color=red]45:$v_4$}]{};

\node (c1) at (-2.5,2) {};
\node (c2)  at (-1.5,2){};
\node (c3)  at (-0.5,2)[pin={[color=red]45:$v_3$}]{};
\node (c4) at (0.5,2){};
\node (d1)  at  (4.5,2)[pin={[color=red]135:$v_5$}]{};
\node (d2) at (5.5,2) {};

\node (f1) at (-2,0){};
\node (f2)  at (-1,0){};
\node (f3)  at (0,0){};
\node (f4) at (1,0){};
    
\node (e1) at (2.5,0){};
\node (e2)  at (3.5,0){};
\node (e3)  at (4.5,0){};
\node (e4) at (5.5,0){};
\node (e5) at (6.5,0){};

\draw (a1)--(b1);
\draw (a1)--(b2);
\draw	 (a1)--(b3);
\draw (b1)--(c1);
\draw (b1)--(c2);
\draw (b1)--(c3);
\draw (b1)--(c4);
\draw (b3)--(d1);
\draw (b3)--(d2);
\draw (c3)--(f1);
\draw (c3)--(f2);
\draw (c3)--(f3);
\draw (c3)--(f4);
\draw (d1)--(e1);
\draw (d1)--(e2);
\draw (d1)--(e3);
\draw (d1)--(e4);
\draw (d1)--(e5);

\end{tikzpicture}
\caption{Example of a tree in $\T_{(3,4,4,2,5)}$}
\label{fig:examplestree}
\end{figure}
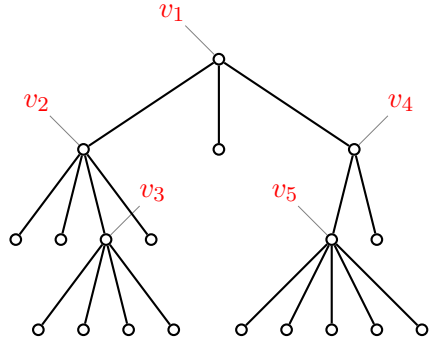

\subsection{Dyck paths}
We consider \emph{lattice paths} in $\NN \times \NN$ starting at $(0,0)$ and taking only \emph{east 
steps} E (in the direction $(1,0)$) and \emph{north steps}~N (in the direction $(0,1)$). For 
positive relatively prime integers $a$ and~$b$, a \emph{rational Dyck path} or an 
\emph{$(a,b)$-Dyck path} is a lattice path with endpoint $(b,a)$ wich stays above the 
diagonal $y=\frac{a}{b}x$. See Figure~\ref{fig:examplerationalpath} for an example of a 
rational Dyck path with $(a,b)=(5,8)$. 

Rational Dyck paths were introduced by 
Armstrong, Rhoades and Williams in~\cite{ArmstrongRhoadesWilliams2013} where they started the 
systematic combinatorial study of rational Catalan combinatorics. This definition has as special 
cases the classical \emph{Dyck paths} that occur for $n\ge 1$ when $a=n$ and $b=n+1$ and more 
generally the \emph{Fuss-Catalan} generalization of Dyck paths that occur when $a=n$ and $b=kn+1$ 
for a fixed $k \in \PP$. 

Note that in the $b \times a$ grid the diagonal $y=\frac{a}{b}x$ crosses certains cells. Those 
cells define a ribbon shape that cannot be crossed by any rational Dyck path in order to stay 
above the diagonal (see the gray shape in Figure~\ref{fig:examplerationalpath}). Let $s \in \comp$ 
be defined such that $s(i)$ is the number of cells in row $i$ crossed by the diagonal for 
each $i\in [a]$, we call $s$ the \emph{degree sequence} associated to the pair $(a,b)$. 
For any rational pair $(a,b)$ with residue $r=b-\lfloor \frac{b}{a} \rfloor a$ it is not difficult 
to check that the degree sequence is given by
\begin{align}\label{equation:definitionofrationalsignature}
   s(i) =  \left\lceil\frac{b}{a}\right\rceil + \chi(0 < ir \text{ (mod a)}< r ),
\end{align}
where $\lfloor x \rfloor$ and $\lceil x \rceil$ are the floor and ceiling functions
respectively, and $\chi$ is the function
\[
   \chi(S) = \left\{
     \begin{array}{lr}
       1  : &S \text{ is true }\\
       0  : &\text{otherwise}.
     \end{array}
   \right.
\]

For example for $(5,8)$ the degree sequence is $(2,3,2,3,2)$ since $\lceil\frac{8}{5}\rceil=2$, 
$r=8-\lfloor\frac{8}{5}\rfloor 5=3$ and $3(1,2,3,4,5)=(3,1,4,2,0) \text { (mod 5)}$ 
(see Figure~\ref{fig:examplerationalpath}).

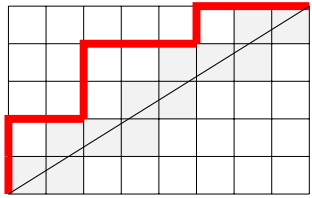
\begin{figure}
\centering
\begin{tikzpicture}[scale=0.5]
\edef \signature{2,3,2,3,2}

\def\row{0}
\def\col{0}

\foreach \part in \signature{
\pgfmathsetmacro \newcol {\col+\part-1}
\pgfmathparse{\col+\part-1}
\global\let\newcol\pgfmathresult
\pgfmathparse{\row+1}
\global\let\newrow\pgfmathresult
\draw[fill, color=gray!10] (\col,\row) rectangle (\newcol+1,\newrow);
\global\let\row\newrow
\global\let\col\newcol
}
\draw[very thin] (0, 0) grid (\col+1, \row);

\def\rowpa{0}
\def\colpa{0}
\edef \pa{0,2,0,3,3}
\foreach \var in \pa{
\pgfmathparse{\colpa+\var}
\global\let\newcolpa\pgfmathresult
\pgfmathparse{\rowpa+1}
\global\let\newrowpa\pgfmathresult
\draw [line width=3, color=red]
(\colpa, \rowpa)--(\colpa,\newrowpa)--(\newcolpa,\newrowpa);

\global\let\rowpa\newrowpa
\global\let\colpa\newcolpa
}

\draw (0,0)--(\col+1,\row);

\end{tikzpicture}
\caption{Example of a $(5,8)$ rational path}
\label{fig:examplerationalpath}
\end{figure}

Every lattice path can be encoded with a string of the letters $N$ and~$E$ representing the 
north steps and east steps respectively, where the reading order of the path starts at the origin 
$(0,0)$. For example the path~$D$ of Figure~\ref{fig:examplerationalpath} can be represented as 
$D=NNEENNEEENEEE$. If the lattice path is a rational Dyck path then it must have exactly~$a$ 
north steps and it must start with a north step and finish with an east step, so it is of the form 
$D=NE^{i_1}NE^{i_2}\cdots NE^{i_a}E$ where $E^{i}$ indicates a string of $i$ occurrences of the 
letter $E$ and then it can be uniquely associated with the weak composition 
$\mu(D):=(i_1,i_2,\dots,i_a)$. So we can think about $(a,b)$-Dyck paths as weak compositions of 
$b-1$ of length~$a$. For the example in Figure~\ref{fig:examplerationalpath} we have that 
$\mu(D)=(0,2,0,3,2)$ (note that the last~$E$ is not counted).  There is one special rational Dyck 
path, the \emph{identity path}~$\IdentityPath=E$, 
with exactly one east step and without any north steps. By convention, this is considered as the 
unique $(0,1)$-Dyck path. We also have that $\mu(\IdentityPath)=\emptyset$ (the empty weak 
composition). It is not hard to see that 
in dominance order there is a maximal $(a,b)$-Dyck path $\mathcal{D}$ that is defined by the 
boundary of 
the forbidden ribbon shape. This path has $\mu(\mathcal{D})=s-\mathbf{1}:=(s(1)-1,s(2)-1,\dots, 
s(a)-1)$. An equivalent definition of a rational Dyck path is a path $D$ such that $\mu(D)\le_{dom} 
\mu(\mathcal{D})=s-\mathbf{1}$. 

There is really nothing special about the rational ribbon shape 
apart from the fact that it is defined using the diagonal of a $b\times a$ rectangle. The 
information of the numbers $a$ and $b$ can be obtained from the degree sequence $s$. Indeed, 
$a=\ell(s)$ and $b=|s|-\ell(s)+1$. 

\begin{definition}[$s$-Dyck paths]
\label{def_sDyckPath}
For any $s \in \comp$, that we will call a \emph{signature}, we 
define an \emph{$s$-Dyck path} $D$ to 
be a lattice path starting at $(0,0)$ and ending at $(|s|-\ell(s)+1,\ell(s))$ such that~$\mu(D)\le_{dom} 
s-\mathbf{1}$. We denote by $\DP_{s}$ the set of $s$-Dyck paths, and refer to an 
$s$-Dyck path as a Dyck path when $s$ is clear from the context.
\end{definition}
  
An example of a $(3,4,4,2,5)$-Dyck path $D$ with 
$\mu(D)=(0,2,6,0,5)$ is illustrated in Figure~\ref{fig:examplespath}. We will say that a signature $s$ is \emph{rational} when $s$ is obtained as 
in equation \ref{equation:definitionofrationalsignature} for a pair $(a,b)$ of relatively prime 
numbers.

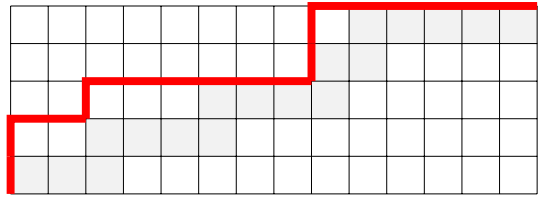
\begin{figure}
\centering
\begin{tikzpicture}[scale=0.5]
\edef \signature{3,4,4,2,5}

\def\row{0}
\def\col{0}

\foreach \part in \signature{
\pgfmathsetmacro \newcol {\col+\part-1}
\pgfmathparse{\col+\part-1}
\global\let\newcol\pgfmathresult
\pgfmathparse{\row+1}
\global\let\newrow\pgfmathresult
\draw[fill, color=gray!10] (\col,\row) rectangle (\newcol+1,\newrow);
\global\let\row\newrow
\global\let\col\newcol
}
\draw[very thin] (0, 0) grid (\col+1, \row);

\def\rowpa{0}
\def\colpa{0}
\edef \pa{0,2,6,0,6}
\foreach \var in \pa{
\pgfmathparse{\colpa+\var}
\global\let\newcolpa\pgfmathresult
\pgfmathparse{\rowpa+1}
\global\let\newrowpa\pgfmathresult
\draw [line width=3, color=red]
(\colpa, \rowpa)--(\colpa,\newrowpa)--(\newcolpa,\newrowpa);

\global\let\rowpa\newrowpa
\global\let\colpa\newcolpa
}


\end{tikzpicture}
\caption{Example of a $(3,4,4,2,5)$-Dyck path}
\label{fig:examplespath}
\end{figure}

We define the \emph{area vector} of an $s$-Dyck path $D$ to 
be the weak composition $\Area(D):=(0)\oplus(s-\mathbf{1})\setminus_{\dom}\mu(D)$. 
Its entries are the number of boxes, counted by rows, between $D$ and the ribbon shape determined 
by 
$s$. The \emph{area} of $D$ is $\area(D):=|\Area(D)|$. In the example of Figure 
\ref{fig:examplespath} $\Area(D)=(0,2,3,0,1)$ and $\area(D)=5$. 

\subsubsection{Bijection between $s$-trees and $s$-Dyck paths}
\label{sec:trees-DyckPaths}

Let $T$ be an $s$-tree. Label the internal nodes of $T$ with $N$'s and the leaves with $E$'s. The 
lattice path $\onetotwo(T)$ associated to $T$ is obtained by reading the sequence of $N$'s and $E$'s on 
the nodes of $T$ in preorder. The function $\onetotwo$ maps the identity tree~$\bullet$ to the identity 
path $E$.
An example of this bijection is illustrated in Figure~\ref{fig:treenortheast}. 

\begin{figure}[htb]
\centering
 \begin{tikzpicture}[]
\begin{scope}[thick,scale=0.45]
\tikzstyle{every node}=[circle, draw,
                        inner sep=0.5pt, minimum width=4pt,font=\small,scale=0.9]
\node (a1) at (2,6)[pin={[color=red,pin distance=3pt]135:$1$}]{N};
\node (b1) at (-1,4)[pin={[color=red,pin distance=3pt]135:$2$}]{N};
\node (b2) at (2,4)[pin={[color=red,pin distance=3pt]135:$11$}]{E};
\node (b3) at (5,4)[pin={[color=red,pin distance=3pt]45:$12$}]{N};

\node (c1) at (-2.5,2) [pin={[color=red,pin distance=3pt]225:$3$}]{E};
\node (c2)  at (-1.5,2)[pin={[color=red,pin distance=3pt]225:$4$}]{E};
\node (c3)  at (-0.5,2)[pin={[color=red,pin distance=3pt]225:$5$}]{N};
\node (c4) at (0.5,2)[pin={[color=red,pin distance=3pt]45:$10$}]{E};
\node (d1)  at  (4.5,2)[pin={[color=red,pin distance=3pt]135:$13$}]{N};
\node (d2) at (5.5,2) [pin={[color=red,pin distance=3pt]45:$19$}] {E};

\node (f1) at (-2,0)[pin={[color=red,pin distance=3pt]225:$6$}]{E};
\node (f2)  at (-1,0)[pin={[color=red,pin distance=3pt]225:$7$}]{E};
\node (f3)  at (0,0)[pin={[color=red,pin distance=3pt]225:$8$}]{E};
\node (f4) at (1,0)[pin={[color=red,pin distance=3pt]230:$9$}]{E};
    
\node (e1) at (2.7,0)[pin={[color=red,pin distance=3pt]310:$14$}]{E};
\node (e2)  at (3.5,0)[pin={[color=red,pin distance=3pt]315:$15$}]{E};
\node (e3)  at (4.5,0)[pin={[color=red,pin distance=3pt]315:$16$}]{E};
\node (e4) at (5.5,0)[pin={[color=red,pin distance=3pt]315:$17$}]{E};
\node (e5) at (6.5,0)[pin={[color=red,pin distance=3pt]315:$18$}]{E};

\draw (a1)--(b1);
\draw (a1)--(b2);
\draw	 (a1)--(b3);
\draw (b1)--(c1);
\draw (b1)--(c2);
\draw (b1)--(c3);
\draw (b1)--(c4);
\draw (b3)--(d1);
\draw (b3)--(d2);
\draw (c3)--(f1);
\draw (c3)--(f2);
\draw (c3)--(f3);
\draw (c3)--(f4);
\draw (d1)--(e1);
\draw (d1)--(e2);
\draw (d1)--(e3);
\draw (d1)--(e4);
\draw (d1)--(e5);
\end{scope}
\begin{scope}[xshift=5cm,scale=0.35]
\edef \signature{3,4,4,2,5}

\def\row{0}
\def\col{0}

\foreach \part in \signature{
\pgfmathsetmacro \newcol {\col+\part-1}
\pgfmathparse{\col+\part-1}
\global\let\newcol\pgfmathresult
\pgfmathparse{\row+1}
\global\let\newrow\pgfmathresult
\draw[fill, color=gray!10] (\col,\row) rectangle (\newcol+1,\newrow);
\global\let\row\newrow
\global\let\col\newcol
}
\draw[very thin] (0, 0) grid (\col+1, \row);

\def\rowpa{0}
\def\colpa{0}
\edef \pa{0,2,6,0,6}
\foreach \var in \pa{
\pgfmathparse{\colpa+\var}
\global\let\newcolpa\pgfmathresult
\pgfmathparse{\rowpa+1}
\global\let\newrowpa\pgfmathresult
\draw [line width=3, color=red]
(\colpa, \rowpa)--(\colpa,\newrowpa)--(\newcolpa,\newrowpa);

\global\let\rowpa\newrowpa
\global\let\colpa\newcolpa
}

\node[color=blue]  at (7,-1) {\small NNEENEEEEEENNEEEEEE};

\end{scope}

\draw[thick, ->](3.5,1)--node [above] {$\xi$}(4.5,1);
\draw[thick, <-](3.5,0.8)--node [below] {$\zeta$}(4.5,0.8);

\end{tikzpicture}
\caption{Example of the bijection between $s$-trees and $s$-Dyck paths.}
\label{fig:treenortheast}
\end{figure}
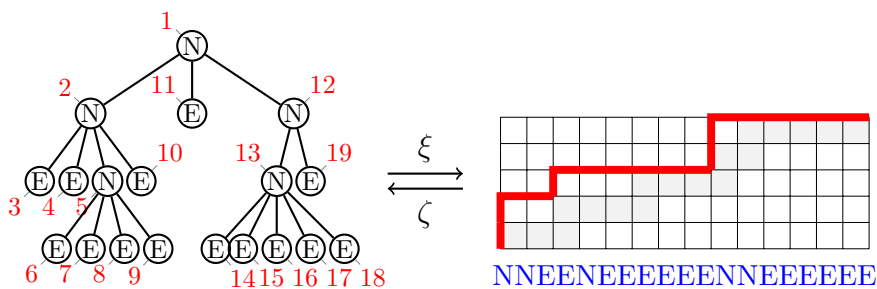

\begin{figure}[htb]
\centering
 \begin{tikzpicture}[]
\begin{scope}[thick,scale=0.45]
\tikzstyle{every node}=[circle, draw,
                        inner sep=0.5pt, minimum width=4pt,font=\small,scale=0.9]
\node (a1) at (2,6)[pin={[color=red,pin distance=3pt]135:$0$}]{};
\node (b1) at (-1,4)[pin={[color=red,pin distance=3pt]135:$2$}]{};
\node (b2) at (2,4)[pin={[color=red,pin distance=3pt]135:$1$}]{};
\node (b3) at (5,4)[pin={[color=red,pin distance=3pt]45:$0$}]{};

\node (c1) at (-2.5,2) [pin={[color=red,pin distance=3pt]225:$5$}]{};
\node (c2)  at (-1.5,2)[pin={[color=red,pin distance=3pt]225:$4$}]{};
\node (c3)  at (-0.5,2)[pin={[color=red,pin distance=3pt]225:$3$}]{};
\node (c4) at (0.5,2)[pin={[color=red,pin distance=3pt]45:$2$}]{};
\node (d1)  at  (4.5,2)[pin={[color=red,pin distance=3pt]135:$1$}]{};
\node (d2) at (5.5,2) [pin={[color=red,pin distance=3pt]45:$0$}] {};

\node (f1) at (-2,0)[pin={[color=red,pin distance=3pt]225:$6$}]{};
\node (f2)  at (-1,0)[pin={[color=red,pin distance=3pt]225:$5$}]{};
\node (f3)  at (0,0)[pin={[color=red,pin distance=3pt]225:$4$}]{};
\node (f4) at (1,0)[pin={[color=red,pin distance=3pt]230:$3$}]{};
    
\node (e1) at (2.7,0)[pin={[color=red,pin distance=3pt]310:$5$}]{};
\node (e2)  at (3.5,0)[pin={[color=red,pin distance=3pt]315:$4$}]{};
\node (e3)  at (4.5,0)[pin={[color=red,pin distance=3pt]315:$3$}]{};
\node (e4) at (5.5,0)[pin={[color=red,pin distance=3pt]315:$2$}]{};
\node (e5) at (6.5,0)[pin={[color=red,pin distance=3pt]315:$1$}]{};

\draw (a1)--(b1);
\draw (a1)--(b2);
\draw	 (a1)--(b3);
\draw (b1)--(c1);
\draw (b1)--(c2);
\draw (b1)--(c3);
\draw (b1)--(c4);
\draw (b3)--(d1);
\draw (b3)--(d2);
\draw (c3)--(f1);
\draw (c3)--(f2);
\draw (c3)--(f3);
\draw (c3)--(f4);
\draw (d1)--(e1);
\draw (d1)--(e2);
\draw (d1)--(e3);
\draw (d1)--(e4);
\draw (d1)--(e5);
\end{scope}

\end{tikzpicture}
\caption{The area labeling of a tree $T$. The area vector of $D=\onetotwo(T)$ is the sequence of
labels of the internal nodes listed preorder.}
\label{fig:area_from_tree}
\end{figure}
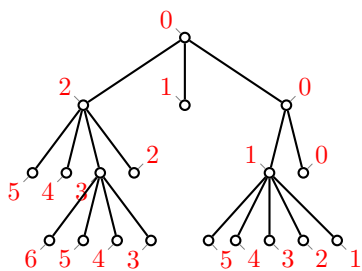

We define the \emph{area labeling} of a tree as the labeling of the nodes such that the root is 
labeled by 0 and the $i$th child of an internal node $x$, from right to left, is labeled by $i-1$ 
plus the label of~$x$. This is illustrated for our example in 
Figure~\ref{fig:area_from_tree}.
It turns out that the area vector of $D=\onetotwo(T)$ can be read directly from~$T$.

\begin{lemma}\label{lem:area_from_tree}
The entries of the area vector of the Dyck path $D=\onetotwo(T)$ are the labels in 
the area labeling of $T$ of the internal nodes listed in preorder.  
\end{lemma}

\begin{proof}
The root of the tree corresponds to the first north step of the path, which contributes a zero 
entry 
in the area vector. Let $x$ and $y$ be two internal nodes of $T$ such that $y$ is the $i$th child 
from right to left of $x$. Let $a(x)$ and $a(y)$ be the entries of the area vector of $D$ of the 
corresponding north steps in $D$.
We need to show that $a(y)=a(x)+i-1$. Let $u_1,\dots, u_j$ be the internal nodes strictly between 
$x$ and $y$ in $T$ in preorder, and denote $u_0=x$. Then, 
$$a(y)= a(x) + \sum_{k=0}^j (\deg(u_k)-1) - E(x,y),$$
where the sum of the degrees minus one is the increase of area determined by the ribbon shape and 
$E(x,y)$ denotes the number of east steps (or leafs) between $x$ and $y$. Since 
$E(x,y)=\sum_{k=0}^j \deg(u_k) - j - i$, then $a(y)=a(x)-(j+1)+j+i= a(x)+i-1$ as desired. 
\end{proof}

\begin{lemma}
The path $\onetotwo(T)$ is an $s$-Dyck path.
\end{lemma}

\begin{proof}
By definition, all the entries of the area labeling of the tree $T$ are nonnegative. 
Lemma~\ref{lem:area_from_tree} then implies that the area vector of $\onetotwo(T)$ has nonnegative 
entries 
and therefore that it is an $s$-Dyck path. 
\end{proof}

We will see below that the map $\onetotwo$ is actually a bijection between $s$-trees and $s$-Dyck paths. 
The inverse of this map is described below. Let $D$ be an $s$-Dyck path and~$\mu=\mu(D)$. Define 
the 
tree $\zeta(D)$ associated to $D$ recursively as follows.
\begin{enumerate}
\item First, start with a root $r=v_1$ with $s(1)$ children.
\item Attach an internal node $v_2$ with $s(2)$ children at the $(\mu(1)+1)$th leaf, ordered from 
left to right.    
\item In general, at the $i$th step of the process we have a tree $\tau_i$ with internal nodes 
$v_1,\dots,v_i$ in preorder. The $(i+1)$th tree $\tau_{i+1}$ is obtained from $\tau_i$ by attaching 
an internal node $v_{i+1}$ with $s(i+1)$ children at the $(\mu(i)+1)$th leaf that appear after 
$v_i\in \tau_i$ in preorder.
\item The process finishes after attaching the $a$ internal nodes $v_1,\dots ,v_a$, where 
$a=\ell(s)$.
\end{enumerate}

\begin{figure}[htb]
\centering
 \begin{tikzpicture}[]
\begin{scope}[thick,scale=0.35]
\tikzstyle{every node}=[circle, draw,
                        inner sep=0.5pt, minimum width=4pt,font=\small,scale=0.9]
\node (a1) at (2,6)[pin={[color=red,pin distance=3pt]135:$v_1$}]{};

\node (b1) at (-1,4){};
\node (b2) at (2,4){};
\node (b3) at (5,4){};



    

\draw (a1)--(b1);
\draw (a1)--(b2);
\draw	 (a1)--(b3);
\end{scope}


\begin{scope}[xshift=3.2cm,thick,scale=0.35]
\tikzstyle{every node}=[circle, draw,
                        inner sep=0.5pt, minimum width=4pt,font=\small,scale=0.9]
\node (a1) at (2,6)[pin={[color=red,pin distance=3pt]135:$v_1$}]{};

\node (b1) at (-1,4)[pin={[color=red,pin distance=3pt]135:$v_2$}]{};
\node (b2) at (2,4){};
\node (b3) at (5,4){};

\node (c1) at (-2.5,2) {};
\node (c2)  at (-1.5,2){};
\node (c3)  at (-0.5,2){};
\node (c4) at (0.5,2){};


    

\draw (a1)--(b1);
\draw (a1)--(b2);
\draw	 (a1)--(b3);
\draw (b1)--(c1);
\draw (b1)--(c2);
\draw (b1)--(c3);
\draw (b1)--(c4);
\end{scope}


\begin{scope}[xshift=6.5cm,thick,scale=0.35]
\tikzstyle{every node}=[circle, draw,
                        inner sep=0.5pt, minimum width=4pt,font=\small,scale=0.9]
\node (a1) at (2,6)[pin={[color=red,pin distance=3pt]135:$v_1$}]{};

\node (b1) at (-1,4)[pin={[color=red,pin distance=3pt]135:$v_2$}]{};
\node (b2) at (2,4){};
\node (b3) at (5,4){};

\node (c1) at (-2.5,2) {};
\node (c2)  at (-1.5,2){};
\node (c3)  at (-0.5,2)[pin={[color=red,pin distance=5pt]45:$v_3$}]{};
\node (c4) at (0.5,2){};


\node (f1) at (-2,0){};
\node (f2)  at (-1,0){};
\node (f3)  at (0,0){};
\node (f4) at (1,0){};
    

\draw (a1)--(b1);
\draw (a1)--(b2);
\draw	 (a1)--(b3);
\draw (b1)--(c1);
\draw (b1)--(c2);
\draw (b1)--(c3);
\draw (b1)--(c4);
\draw (c3)--(f1);
\draw (c3)--(f2);
\draw (c3)--(f3);
\draw (c3)--(f4);
\end{scope}


\begin{scope}[xshift=9.8cm, thick,scale=0.35]
\tikzstyle{every node}=[circle, draw,
                        inner sep=0.5pt, minimum width=4pt,font=\small,scale=0.9]
\node (a1) at (2,6)[pin={[color=red,pin distance=3pt]135:$v_1$}]{};

\node (b1) at (-1,4)[pin={[color=red,pin distance=3pt]135:$v_2$}]{};
\node (b2) at (2,4){};
\node (b3) at (5,4)[pin={[color=red,pin distance=3pt]45:$v_4$}]{};

\node (c1) at (-2.5,2) {};
\node (c2)  at (-1.5,2){};
\node (c3)  at (-0.5,2)[pin={[color=red,pin distance=5pt]45:$v_3$}]{};
\node (c4) at (0.5,2){};

\node (d1)  at  (4.5,2){};
\node (d2) at (5.5,2) {};

\node (f1) at (-2,0){};
\node (f2)  at (-1,0){};
\node (f3)  at (0,0){};
\node (f4) at (1,0){};
    

\draw (a1)--(b1);
\draw (a1)--(b2);
\draw	 (a1)--(b3);
\draw (b1)--(c1);
\draw (b1)--(c2);
\draw (b1)--(c3);
\draw (b1)--(c4);
\draw (b3)--(d1);
\draw (b3)--(d2);
\draw (c3)--(f1);
\draw (c3)--(f2);
\draw (c3)--(f3);
\draw (c3)--(f4);
\end{scope}


\begin{scope}[xshift=13.4cm,thick,scale=0.35]
\tikzstyle{every node}=[circle, draw,
                        inner sep=0.5pt, minimum width=4pt,font=\small,scale=0.9]
\node (a1) at (2,6)[pin={[color=red,pin distance=3pt]135:$v_1$}]{};

\node (b1) at (-1,4)[pin={[color=red,pin distance=3pt]135:$v_2$}]{};
\node (b2) at (2,4){};
\node (b3) at (5,4)[pin={[color=red,pin distance=3pt]45:$v_4$}]{};

\node (c1) at (-2.5,2){};
\node (c2)  at (-1.5,2){};
\node (c3)  at (-0.5,2)[pin={[color=red,pin distance=5pt]45:$v_3$}]{};
\node (c4) at (0.5,2){};

\node (d1)  at  (4.5,2)[pin={[color=red,pin distance=3pt]135:$v_5$}]{};
\node (d2) at (5.5,2) {};

\node (f1) at (-2,0){};
\node (f2)  at (-1,0){};
\node (f3)  at (0,0){};
\node (f4) at (1,0){};
    
\node (e1) at (2.7,0){};
\node (e2)  at (3.5,0){};
\node (e3)  at (4.5,0){};
\node (e4) at (5.5,0){};
\node (e5) at (6.5,0){};

\draw (a1)--(b1);
\draw (a1)--(b2);
\draw	 (a1)--(b3);
\draw (b1)--(c1);
\draw (b1)--(c2);
\draw (b1)--(c3);
\draw (b1)--(c4);
\draw (b3)--(d1);
\draw (b3)--(d2);
\draw (c3)--(f1);
\draw (c3)--(f2);
\draw (c3)--(f3);
\draw (c3)--(f4);
\draw (d1)--(e1);
\draw (d1)--(e2);
\draw (d1)--(e3);
\draw (d1)--(e4);
\draw (d1)--(e5);

\end{scope}


\end{tikzpicture}
\caption{The recursive construction of the tree $\zeta(D)$ for the Dyck path with signature 
$s(D)=(3,4,4,2,5)$ and composition $\mu(D)=(0,2,6,0,6)$.}
\label{fig:map_path_to_tree}
\end{figure}
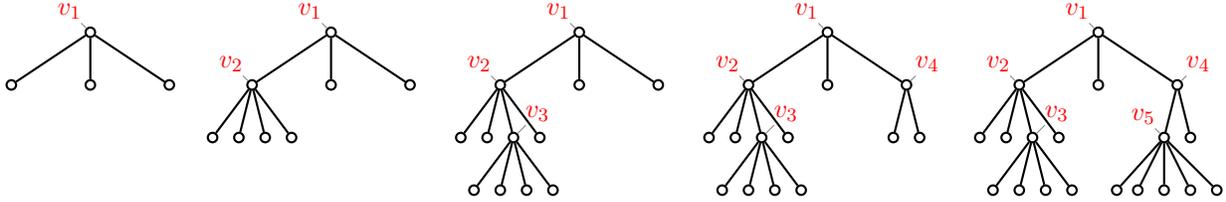

\begin{lemma}
Let $D$ be an $s$-Dyck path. The map $\zeta$ is well defined and $\zeta(D)$ is an $s$-tree.
\end{lemma}

\begin{proof}
In order to show that the map $\zeta$ is well defined we need to show that for $i<a$ there are at 
least~$\mu(i)+1$ leafs in the tree $\tau_i$ that appear after the internal node $v_i\in \tau_i$ in 
preorder. This guarantees that a new node $v_{i+1}$ can be added at the right place and that the 
tree~$\tau_{i+1}$ in 
the recursive definition can be constructed. For $i=1$ this is clear since $v_1$ has $s(1)\geq 
\mu(1)+1$ children. In general, since~$D$ is a path that lies above the ribbon shape defined by 
$s$, 
for~$i<a$ we have 
\begin{equation}
\label{eq:lemma_proof_pathtotree}
\mu(1)+\dots + \mu(i) \leq (s(1)-1) + \dots +(s(i)-1).
\end{equation}
The total number of leafs in $\tau_i$ is equal to $\sum_{k=1}^i (s(k) -1)+1$, from which 
$\sum_{k=1}^{i-1} \mu(k)$ are before $v_i$ in preorder. Combining these two equations with 
the inequality \eqref{eq:lemma_proof_pathtotree} we obtain that the number of leafs in $\tau_i$ 
after $v_i$ in 
preorder is
\[
\sum_{k=1}^i (s(k) -1)+1 - \sum_{k=1}^{i-1} \mu(k)  \geq \mu(i)+1.
\]
This shows that the map is well defined. The signature of the tree $\zeta(D)$ is by construction 
equal to $s=(s(1),\dots, s(a))$. 
\end{proof}

\begin{theorem}\label{thm:bijection_trees_paths}
The map $\onetotwo:\T_s\rightarrow\DP_s$ is a bijection between $s$-trees and $s$-Dyck paths. The inverse of~$\onetotwo$ is $\zeta$.
\end{theorem}

\begin{proof}
We will show that $D=\onetotwo(T)$ if and only if $T=\zeta(D)$. This implies that both maps are 
bijections 
and that they are inverses to each other. 

Let $D=\onetotwo(T)$ and~$\mu=\mu(D)$. Denote by $v_1,\dots,v_a$ the internal nodes of $T$ in preorder. 
By 
the definition of $\onetotwo$, there are exactly $\mu(i)$ leafs between $v_i$ and $v_{i+1}$ traveling 
$T$ in preorder. This 
condition completely characterizes $T$ and is satisfied by $\zeta(D)$. Therefore,~$T=\zeta(D)$. 

On the other hand, if $T=\zeta(D)$ then $T$ is the unique tree with internal nodes $v_1,\dots,v_a$ 
in preorder such that there are exactly $\mu(i)$ leafs between $v_i$ and $v_{i+1}$, for $i<a$. 
Therefore, $D=\onetotwo(T)$. 
\end{proof}

\subsection{$312$-avoiding Stirling permutations}

Let $s=(s(1),s(2),\dots, s(a))$ be a composition. An \emph{$s$-permutation} is any 
multiset permutation (or \emph{multipermutation}) of the multiset $\{1^{s(1)},2^{s(2)},\dots, 
a^{s(a)}\}$, where there are $s(i)$ copies of each $i$. Let us denote the set of $s$-permutations 
as $\sym_s$, then we have that

$$|\sym_s|=\binom{|s|}{s}:=\dfrac{|s|!}{s(1)!s(2)!\cdots s(a)!}.$$

For multipermutations $\tau$ and $\sigma$ we say that $\sigma$ contains the 
pattern $\tau$ if there is a subword of $\sigma$ where the elements have the 
same relative order as the elements in $\tau$. For example, the 
multipermutation $\sigma=1132235544$ contains the pattern $\tau=212$ because the elements in the 
subword $323$ of $\sigma$  are in the same relative order than the elements 
in $\tau$. A multipermutation can contain many patterns and multiple or no occurrences 
of a particular pattern. For example $\sigma=1132235544$ contains no occurrence of the pattern 
$321$. We say that $\sigma$ \emph{avoids}~$\tau$ or is  \emph{$\tau$-avoiding} if it does not 
contain any occurrence of the pattern $\tau$. 

A particular example of pattern avoiding multipermutations are the $212$-avoiding 
$s$-permutations introduced by Gessel and Stanley in \cite{GesselStanley1978} when they studied a 
generalization of Euler's formula for Eulerian polynomials. They considered the descent 
enumerator of $212$-avoiding multipermutations of $\{1,1,2,2,\cdots,n,n\}$, i.e., multipermutations 
in which all the numbers between the two occurrences of any fixed number $m$ are larger than $m$. To 
this family, belongs for example the permutation $12234431$ but not the permutation $11322344$ since 
$2$ is less than $3$ and $2$ is between the two occurrences of $3$.
They called this family of multipermutations \emph{Stirling Permutations} and they 
have been studied by many other authors, see for example
\cite{KubaPanholzer2011,JansonKubaPanholzer2011,Park1994-1,Park1994-2,Park1994-3, 
GrahamKnuthPatashnik1994,Dleon2015,RemmelWilson2015}.

The set of $s$-permutations have also already appeared in the work on Hopf algebras of 
Novelli and Thibon~\cite{NovelliThibon2014}. They study generalizations of the Malvenuto-Reutenauer Hopf 
algebra of permutations \cite{MalvenutoReutenauer1995} to multipermutations of $\{1^{m},2^{m},\dots, n^{m}\}$, and of the Loday-Ronco Hopf 
algebra of planar binary trees \cite{LodayRonco1998} to $(m+1)$-ary trees. The authors introduce a notion of 
metasylvester congruence on permutations that allows them to obtain Hopf algebras based on 
decreasing trees. The element representatives of the congruence classes are $121$-avoiding s-permutations for $s=(m,m,\dots,m)$, see~\cite[Proposition~3.10]{NovelliThibon2014}. 
These $121$-avoiding $s$-permutations and associated lattices called metasylvester lattices are 
also studied by Pons in~\cite{Pons2015}, where they are also related to the combinatorics of $m$-Tamari lattices introduced by F. Bergeron.

\begin{definition}[$312$-avoiding Stirling $s$-permutations]
\label{def_sStirlingPermutations}
A Stirling $s$-permutation is a multipermutation of 
$\{1^{s(1)},2^{s(2)},\dots, a^{s(a)}\}$
that avoids the pattern $212$. 
A $312$-avoiding Stirling $s$-permutation is a multipermutation of the same multiset that avoids both patterns $212$ and $312$. We denote by $\SP_s$ the set of Stirling $s$-permutations, and by $\SP_s(312)$ the set of $312$-avoiding Stirling $s$-permutations.
\end{definition}

For instance, 
$1223321$ and $2332211$ are 312-avoiding Stirling $(2,3,2)$-permutations, while $3312221$ is a 
Stirling $(2,3,2)$-permutation that is not $312$-avoiding. See Figure 
\ref{figure:exampleof212312avoiding} for the complete example of $\SP_{(2,3,2)}$ and $\SP_{(2,3,2)}(312)$.

\begin{figure}[htbp]
\centering
\includegraphics[width=0.45\textwidth]{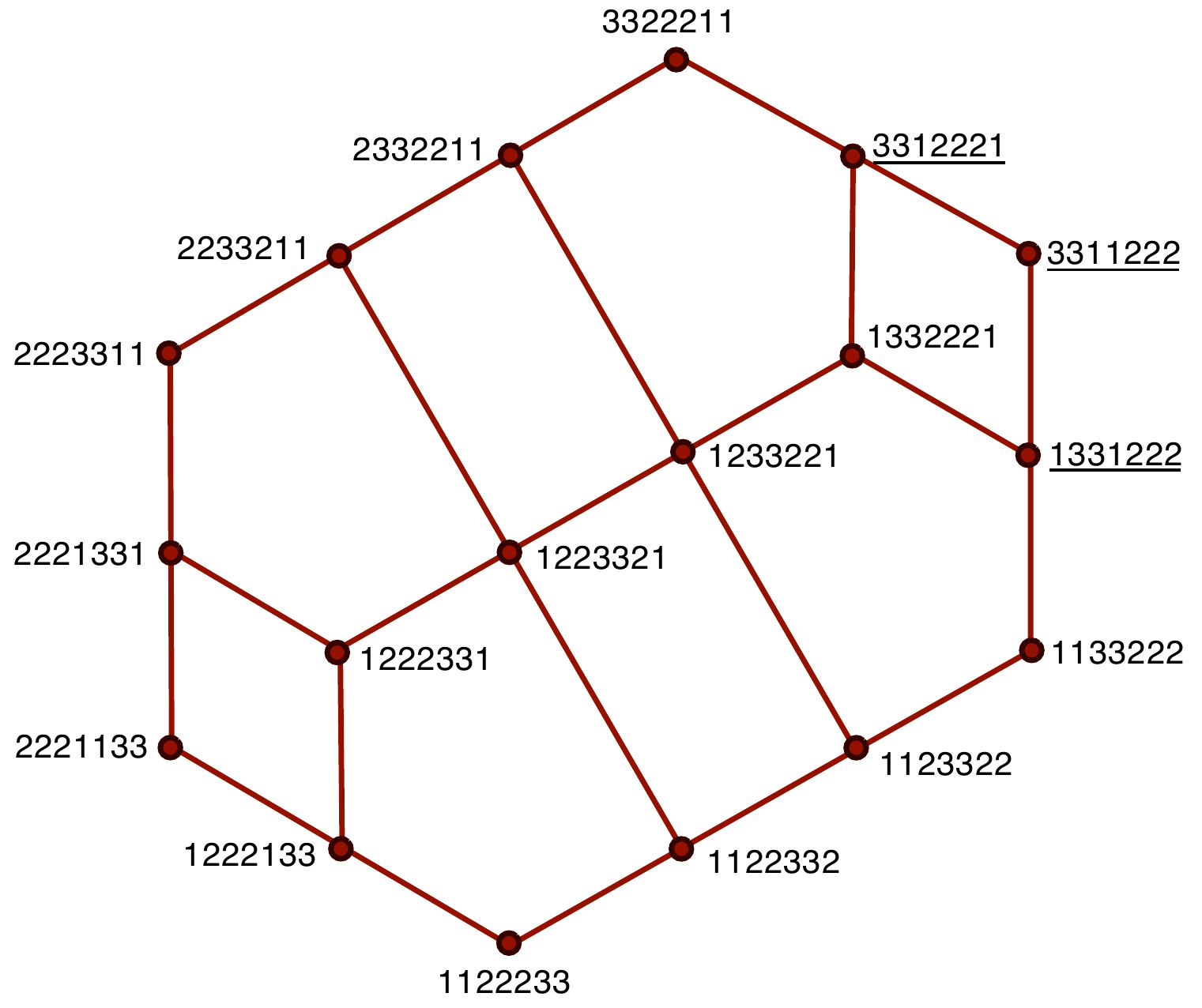}

\caption{All Stirling $s$-permutations for $s=(2,3,2)$ (drawn as the elements of a lattice 
similarly as in~\cite{Pons2015}). The $312$-avoiding are the non-underlined ones. There are 
15 of them, which are in bijection with the 15 $(3,4,3)$-trees.}
\label{figure:exampleof212312avoiding}

\end{figure}

\begin{remark}
Replacing the numbers from $1$ to $a$ for the numbers 
from $a$ to $1$ respectively and reversing the permutation transforms bijectively $(212, 
312)$-avoiding $s$-permutations into $(121, 231)$-avoiding $\overleftarrow s$-permutations,
where $\overleftarrow{s}=(s(a),\dots,s(1))$ denotes the composition obtained by reading $s$ in 
reverse order.
This family of $(121, 231)$-avoiding $\overleftarrow s$-permutations is a generalization of $231$-avoiding permutations which is also an $s$-Catalan family.
\end{remark}

\begin{figure}[htbp]
\centering
\begin{tikzpicture}[thick,scale=0.35]

\begin{scope}[xshift=0]
\tikzstyle{every node}=[circle, draw,
                        inner sep=0.5pt, minimum width=4pt,font=\small,scale=0.9]
\node (a1) at (2,6)[pin={[color=red,pin distance=3pt]135:$v_1$}]{};

\node (b1) at (-1,4){};
\node (b2) at (2,4)[pin={[color=red,pin distance=3pt]180:$v_2$}]{};
\node (b3) at (5,4){};

\node (c1) at (-0.5,2)[pin={[color=red,pin distance=3pt]145:$v_4$}]{};
\node (c2)  at (1,2){};
\node (c3)  at (3,2)[pin={[color=red,pin distance=3pt]0:$v_3$}]{};
\node (c4) at (7,2)[pin={[color=red,pin distance=3pt]45:$v_5$}]{};

\node (d1)  at  (-1,0){};
\node (d2) at (0,0) {};

\node (f1) at (1.5,0){};
\node (f2)  at (2.5,0){};
\node (f3)  at (3.5,0){};
\node (f4) at (4.5,0){};
    
\node (e1) at (5,0){};
\node (e2)  at (6,0){};
\node (e3)  at (7,0){};
\node (e4) at (8,0){};
\node (e5) at (9,0){};

\draw (a1)--(b1);
\draw (a1)--(b2);
\draw	 (a1)--(b3);
\draw (b2)--(c1);
\draw (b2)--(c2);
\draw (b2)--(c3);
\draw (b2)--(c4);
\draw (c1)--(d1);
\draw (c1)--(d2);
\draw (c3)--(f1);
\draw (c3)--(f2);
\draw (c3)--(f3);
\draw (c3)--(f4);
\draw (c4)--(e1);
\draw (c4)--(e2);
\draw (c4)--(e3);
\draw (c4)--(e4);
\draw (c4)--(e5);
\tikzstyle{every node}=[blue, inner sep=0.5pt, minimum width=4pt,font=\small,scale=1]
\node[] at (1,4.5) {$1$};
\node[] at (3,4.5) {$1$};
\node[] at (0.7,2.5) {$2$};
\node[] at (2,2.5) {$2$};
\node[] at (4,2.5) {$2$};
\node[] at (-.5,0.5) {$4$};
\node[] at (2.2,0.5) {$3$};
\node[] at (3,0.5) {$3$};
\node[] at (3.8,0.5) {$3$};
\node[] at (5.9,0.5) {$5$};
\node[] at (6.7,0.5) {$5$};
\node[] at (7.4,0.5) {$5$};
\node[] at (8.1,0.5) {$5$};
\tikzstyle{every node}=[ inner sep=0.5pt, minimum width=4pt,font=\small,scale=1.5]
\node[blue] (permu) at (3,-6) {$1422333255551$};
\draw[->] (3,-2) -- (permu);
\end{scope}

\begin{scope}[xshift=460,yshift=0,thick,scale=1]
\tikzstyle{every node}=[circle, draw,
                        inner sep=0.5pt, minimum width=4pt,font=\small]
\node (a1) at (2,6)[pin={[pin edge={black, dashed, -},pin distance=1pt, red]60:$v_1$}]{};
\node (b1) at (-1,4)[pin={[pin edge={black, dashed, -},pin distance=1pt, red]90:$v_2$}]{};
\node (b2) at (2,4){};
\node (b3) at (5,4)[pin={[pin edge={black, dashed, -},pin distance=1pt, red]60:$v_4$}]{};

\node (c1) at (-2.5,2) {};
\node (c2)  at (-1.5,2){};
\node (c3)  at (-0.5,2)[pin={[pin edge={black, dashed, -},pin distance=3pt, red]200:$v_3$}]{};
\node (c4) at (0.5,2){};
\node (d1)  at  (4.5,2)[pin={[pin edge={black, dashed, -},pin distance=1pt, red]120:$v_5$}]{};
\node (d2) at (5.5,2) {};

\node (f1) at (-2,0){};
\node (f2)  at (-1,0){};
\node (f3)  at (0,0){};
\node (f4) at (1,0){};
    
\node (e1) at (2.5,0){};
\node (e2)  at (3.5,0){};
\node (e3)  at (4.5,0){};
\node (e4) at (5.5,0){};
\node (e5) at (6.5,0){};

\draw (a1)--(b1);
\draw (a1)--(b2);
\draw	 (a1)--(b3);
\draw (b1)--(c1);
\draw (b1)--(c2);
\draw (b1)--(c3);
\draw (b1)--(c4);
\draw (b3)--(d1);
\draw (b3)--(d2);
\draw (c3)--(f1);
\draw (c3)--(f2);
\draw (c3)--(f3);
\draw (c3)--(f4);
\draw (d1)--(e1);
\draw (d1)--(e2);
\draw (d1)--(e3);
\draw (d1)--(e4);
\draw (d1)--(e5);

\tikzstyle{every node}=[circle, blue,scale=1,
                        inner sep=0.5pt, minimum width=4pt,font=\small]

\node at (1.2,4.5) {$1$};
\node at (2.8,4.5) {$1$};
\node at (-1.8,2.3) {$2$};
\node at (-1,2.3) {$2$};
\node at (0,2.3) {$2$};
\node[] at (5,2.3) {$4$};
\node[] at (-1.3,0.3) {$3$};
\node[] at (.3,0.3) {$3$};
\node[] at (-.5,0.3) {$3$};
\node[] at (3.2,0.3) {$5$};
\node[] at (4.1,0.3) {$5$};
\node[] at (5,0.3) {$5$};
\node[] at (6,0.3) {$5$};

\tikzstyle{every node}=[ inner sep=0.5pt, minimum width=4pt,font=\small,scale=1.5]
\node[blue] (permu) at (2,-6) {$2233321155554$};
\draw[->] (2,-2) -- (permu);

\end{scope}

\end{tikzpicture}
\caption{Example of the map $\Sigma:\IT_s\rightarrow \SP_{s-\bf{1}}$}
\label{fig:map_tree_to_permutation}
\end{figure}
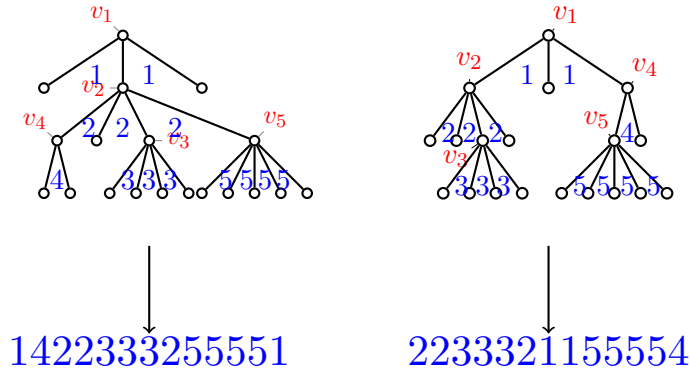

\subsubsection{Bijection between $s$-increasing trees and Stirling $(s-\bf{1})$-permutations}
\label{sec_sincreasing_stirling}
An \emph{increasing tree} is a planar rooted tree whose internal nodes have been decorated with the elements of $[a]$ (where $a$ is the number of internal nodes) such that any path of internal nodes away from the root has strictly increasing labels. 
An $s$-increasing tree is an increasing tree such that the internal node with label $i$ has exactly $s(i)$ children. 
We denote by~$\IT_s$ the set of~$s$-increasing trees. 
Figure \ref{fig:map_tree_to_permutation} shows two examples of  
$(3,4,4,2,5)$-increasing trees.

Gessel provided a simple bijection between the sets $\IT_s$ and $\SP_{s-\bf{1}}$ that we describe below (Gessel's 
result is unpublished but can be found for example in~\cite{JansonKubaPanholzer2011}).
Before explaining this bijection we need to introduce the notion of a cavern in a planar rooted tree.  
For a tree $T \in \T$ and an internal 
node $x$ of $T$ we call a \emph{cavern} of $T$ dermined by 
$x$ to the space between any two consecutive children of $x$ (see Figure \ref{fig:cavern}), i.e., 
if 
the ordered set of children of $v$ is $\{c_1,c_2,\dots,c_k\}$ then the set of caverns determined by
$v$ is
\[\caverns(v)=\{\{c_1,c_2\},\{c_2,c_3\},\dots,\{c_{k-1},c_k\}\}.\]
We also define $\caverns(T):=\cup_{i=1}^a\caverns(v_a)$.
Any internal node $v$ of $\deg(v)=k$ defines exactly $k-1$ caverns and hence 
there are in total $|\caverns(T)|=\sum_{i=1}^a \deg(v_i)-a$ caverns in a tree $T \in \T$ with $a$ 
internal nodes.

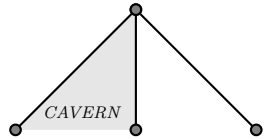
\begin{figure}
\centering
 \begin{tikzpicture}[thick,scale=0.8]
    \fill[fill=black!10] (2,2) -- (0,0) -- (2,0);
    \draw (0.3,0.3) node [anchor=west,scale=0.8]{\tiny \emph{CAVERN}};
    \tikzstyle{every node}=[circle, draw, fill=black!50,
                        inner sep=0pt, minimum width=4pt]
    \draw {
        (2,2)  node (a1){} --  (0,0)node (b1){}
	(a1)  --  (2,0)  node (b2){}
	(a1) node{} --  (4,0)node(b3){}
        };
\end{tikzpicture}
\caption{A cavern}
\label{fig:cavern}
\end{figure}

The set of caverns has a very natural ordering. We can associate each cavern 
$C=\{c_i,c_{i+1}\}$ of $x$ to the largest children $c_{i+1}$ of $x$ in $C$. In that way we can 
think of the caverns as the set of nonminimal children of internal nodes in $T$ (being $c_1$ the 
minimal child of $x$). The order in the set $\caverns(T)$ is the one induced by preorder 
in the set of nonminimal children in $T$. (see Figure \ref{fig:preordercavern}). 
We refer to this order as the \emph{preorder order of the caverns} of $T$.

For $T \in \IT_s$ we label all the caverns with the label of its corresponding internal node. Reading the caverns in preorder gives rise to an $(s-\bf{1})$-permutation 
$\Sigma(T)$. Since $T$ is an increasing tree, the caverns with label $i$ are read before or after all the caverns with 
label $j$ for $i<j$. So, the resulting $(s-\bf{1})$-permutation is $212$-avoiding, so a Stirling 
permutation. 

The process described above maps an increasing tree $T\in \IT_s$ to a Stirling 
$(s-\bf{1})$-permutation~$\Sigma(T)$. In fact, every Stirling $(s-\bf{1})$-permutation can be attained 
uniquely by exactly one such a tree. We construct the inverse map as follows: let $\sigma$ be a 
Stirling $(s-\bf{1})$-permutation. Put a root labeled 1 with $s(1)$ children labelled with the 
$s(1)$ permutations that are separated by the ones in $\sigma$ (some of which might be empty). 
Repeat the process with each of the non-empty permutations appearing in the process with the 
smallest value in the permutation instead of the one. The resulting tree~$\Lambda(\sigma)$ at the end of 
the process is the desired increasing tree, see Figure~\ref{fig:s-permutations_to_trees} for an 
example.  

\begin{figure}[htbp]
\begin{center}
\begin{tikzpicture}[thick,scale=0.35]

\begin{scope}[xshift=-460,yshift=0,thick,scale=1]
\tikzstyle{every node}=[circle,red,
                        inner sep=0.5pt, minimum width=4pt,scale=1.2]
\node (a1) at (2,5){$2233321155554$};

\end{scope}
\begin{scope}[xshift=-160,yshift=0,thick,scale=1]
\tikzstyle{every node}=[circle, draw,
                        inner sep=0.5pt, minimum width=4pt,font=\small]
\node (a1) at (2,6)[pin={[pin edge={black, dashed, -},pin distance=1pt, red]60:$v_1$}]{};
\node (b1) at (-1,4)[pin={[rectangle,pin edge={red},pin distance=4pt, red]270:$223332$}]{};
\node (b2) at (2,4){};
\node (b3) at (5,4)[pin={[rectangle,pin edge={red},pin distance=4pt, red]270:$55554$}]{};

\draw (a1)--(b1);
\draw (a1)--(b2);
\draw	 (a1)--(b3);

\tikzstyle{every node}=[circle, blue,scale=1,
                        inner sep=0.5pt, minimum width=4pt,font=\small]

\node at (1.2,4.5) {$1$};
\node at (2.8,4.5) {$1$};

\end{scope}

\begin{scope}[xshift=160,yshift=0,thick,scale=1]
\tikzstyle{every node}=[circle, draw,
                        inner sep=0.5pt, minimum width=4pt,font=\small]
\node (a1) at (2,6)[pin={[pin edge={black, dashed, -},pin distance=1pt, red]60:$v_1$}]{};
\node (b1) at (-1,4)[pin={[pin edge={black, dashed, -},pin distance=1pt, red]90:$v_2$}]{};
\node (b2) at (2,4){};
\node (b3) at (5,4)[pin={[pin edge={black, dashed, -},pin distance=1pt, red]60:$v_4$}]{};

\node (c1) at (-2.5,2) {};
\node (c2)  at (-1.5,2){};
\node (c3)  at (-0.5,2)[pin={[rectangle,pin edge={red},pin distance=4pt, red]270:$333$}]{};
\node (c4) at (0.5,2){};
\node (d1)  at  (4.5,2)[pin={[rectangle,pin edge={red },pin distance=4pt, red]270:$5555$}]{};
\node (d2) at (5.5,2) {};

\draw (a1)--(b1);
\draw (a1)--(b2);
\draw	 (a1)--(b3);
\draw (b1)--(c1);
\draw (b1)--(c2);
\draw (b1)--(c3);
\draw (b1)--(c4);
\draw (b3)--(d1);
\draw (b3)--(d2);

\tikzstyle{every node}=[circle, blue,scale=1,
                        inner sep=0.5pt, minimum width=4pt,font=\small]

\node at (1.2,4.5) {$1$};
\node at (2.8,4.5) {$1$};
\node at (-1.8,2.3) {$2$};
\node at (-1,2.3) {$2$};
\node at (0,2.3) {$2$};
\node[] at (5,2.3) {$4$};

\end{scope}

\begin{scope}[xshift=460,yshift=0,thick,scale=1]
\tikzstyle{every node}=[circle, draw,
                        inner sep=0.5pt, minimum width=4pt,font=\small]
\node (a1) at (2,6)[pin={[pin edge={black, dashed, -},pin distance=1pt, red]60:$v_1$}]{};
\node (b1) at (-1,4)[pin={[pin edge={black, dashed, -},pin distance=1pt, red]90:$v_2$}]{};
\node (b2) at (2,4){};
\node (b3) at (5,4)[pin={[pin edge={black, dashed, -},pin distance=1pt, red]60:$v_4$}]{};

\node (c1) at (-2.5,2) {};
\node (c2)  at (-1.5,2){};
\node (c3)  at (-0.5,2)[pin={[pin edge={black, dashed, -},pin distance=3pt, red]200:$v_3$}]{};
\node (c4) at (0.5,2){};
\node (d1)  at  (4.5,2)[pin={[pin edge={black, dashed, -},pin distance=1pt, red]120:$v_5$}]{};
\node (d2) at (5.5,2) {};

\node (f1) at (-2,0){};
\node (f2)  at (-1,0){};
\node (f3)  at (0,0){};
\node (f4) at (1,0){};
    
\node (e1) at (2.5,0){};
\node (e2)  at (3.5,0){};
\node (e3)  at (4.5,0){};
\node (e4) at (5.5,0){};
\node (e5) at (6.5,0){};

\draw (a1)--(b1);
\draw (a1)--(b2);
\draw	 (a1)--(b3);
\draw (b1)--(c1);
\draw (b1)--(c2);
\draw (b1)--(c3);
\draw (b1)--(c4);
\draw (b3)--(d1);
\draw (b3)--(d2);
\draw (c3)--(f1);
\draw (c3)--(f2);
\draw (c3)--(f3);
\draw (c3)--(f4);
\draw (d1)--(e1);
\draw (d1)--(e2);
\draw (d1)--(e3);
\draw (d1)--(e4);
\draw (d1)--(e5);

\tikzstyle{every node}=[circle, blue,scale=1,
                        inner sep=0.5pt, minimum width=4pt,font=\small]

\node at (1.2,4.5) {$1$};
\node at (2.8,4.5) {$1$};
\node at (-1.8,2.3) {$2$};
\node at (-1,2.3) {$2$};
\node at (0,2.3) {$2$};
\node[] at (5,2.3) {$4$};
\node[] at (-1.3,0.3) {$3$};
\node[] at (.3,0.3) {$3$};
\node[] at (-.5,0.3) {$3$};
\node[] at (3.2,0.3) {$5$};
\node[] at (4.1,0.3) {$5$};
\node[] at (5,0.3) {$5$};
\node[] at (6,0.3) {$5$};

\end{scope}

\end{tikzpicture}
\caption{Example of the map $\Lambda:\SP_{s-\bf{1}}\rightarrow\IT_s$}
\label{fig:s-permutations_to_trees}
\end{center}
\end{figure}
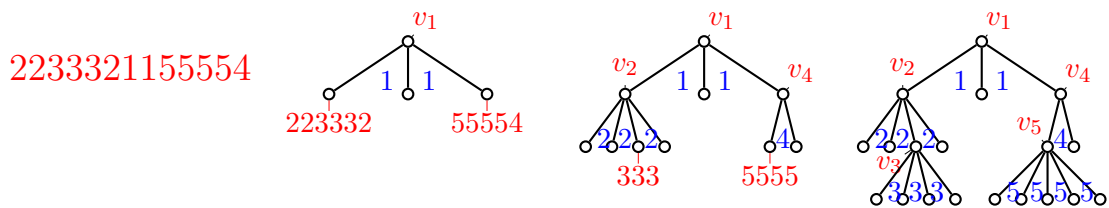

\begin{proposition}[Gessel c.f. \cite{JansonKubaPanholzer2011}]\label{prop:bijection-trees-permutations}
The map $\Sigma:\IT_s \rightarrow \SP_{s-\bf{1}}$ is a bijection between increasing trees and Stirling $(s-\bf{1})$-permutations. The inverse of $\Sigma$ is~$\Lambda$.
\end{proposition}

\subsubsection{Bijection between $s$-trees and $312$-avoiding Stirling $(s-\bf{1})$-permutations}
Note that if we decorate the internal nodes of an $s$-tree $T$ with the values of $[a]$ in preorder 
we obtain an increasing tree $\tilde T\in \IT_s$, hence there is an injection $\T_s\hookrightarrow 
\IT_s$. It turns out that the image of the composition $\T_s\hookrightarrow \IT_s \rightarrow 
\SP_{s-\bf{1}}$ is precisely the set $ \SP_{s-\bf{1}}(312)$. We denote the map determined by this composition by $\tilde \Sigma$. We also denote by $\tilde \Lambda (T)$ the $s$-tree obtained by removing the labels from the increasing tree $\Lambda(T)$. 
The Stirling permutation $2233321155554$ that appear in Figures~\ref{fig:map_tree_to_permutation} and~\ref{fig:s-permutations_to_trees} is 
$312$-avoiding, giving an example of the discussed bijection.

\begin{theorem}\label{prop:bijection-trees-permutations-restricted}
The map $\tilde \Sigma:\T_s\rightarrow \SP_{s-\bf{1}}(312)$ is a bijection between $s$-trees and $312$-avoiding Stirling $(s-\bf{1})$-permutations. The inverse of $\tilde \Sigma$ is~$\tilde \Lambda$.
\end{theorem}

\begin{proof}
It remains to show: 
\begin{enumerate}
\item For any $s$-tree $T$, its image $\tilde \Sigma (T)\in \SP_{s-\bf{1}}(312)$. 
\item For any $\sigma \in \SP_{s-\bf{1}}(312)$, the tree $\tilde \Lambda(\sigma)$ is an $s$-tree.
\end{enumerate}

{\it Proof of (1).}Label the internal nodes of an $s$-tree $T$ in preorder, and the caverns with the 
label of its  corresponding internal node. By Proposition 
\ref{prop:bijection-trees-permutations} we know that reading the caverns in preorder gives rise to 
a Stirling $(s-\bf{1})$-permutation $\Sigma(T)$.  Moreover, we will see 
next that it is also $312$-avoiding. We proceed by contradiction; assume the permutation has a 
subsequence $k\dots i \dots j$ with $i<j<k$. The internal node $k$ (labeled in preorder) 
necessarily has to be a descendant of the node $i$ that is not coming from its right most child. 
Similarly, the node $k$ is a descendant of the node $j$ that is not coming from its right most 
child. Drawing the unique path from node $k$ to the root visits both nodes $i$ and~$j$. Since $i<j$, 
then $j$ has to be a descendant of $i$. From this observations it is straight forward to see from 
the tree that no cavern $i$ appears in preorder between two caverns $k$ and $j$. Thus the pattern 
$312$ is impossible. 

{\it Proof of (2)} 
Let $\sigma \in \SP_{s-\bf{1}}(312)$. By Proposition 
\ref{prop:bijection-trees-permutations}, we know taht the tree $\Lambda(\sigma)$ is increasing.
We need to show that the labeling of its internal nodes is given by preorder. 
We proceed by contradiction; assume that the labeling in $\Lambda(\sigma)$ violates preorder. Consider the minimal pair $v_1,v_2$ in lexicographic preorder such that $v_1<v_2$ in preorder but the label of $v_1$ is larger than the label of $v_2$. Let $v_0$ be the smallest common ancestor of $v_1$ and~$v_2$. Since the tree is increasing, $v_2$ can not be a descendant of $v_1$. Therefore $v_0\neq v_1,v_2$, furthermore it has smaller label than $v_1$ and $v_2$ by lexicographic minimality of the pair $v_1,v_2$. If $r_0,r_1,r_2$ are the labels of $v_0,v_1,v_2$ respectively, then the permutation $\sigma$ corresponding to the tree contains a subsequence $r_1, r_0,r_2$ which satisfies $r_0<r_2<r_1$, and so it is not $312$-avoiding.
\end{proof}

\subsection{Noncrossing partitions}

A \emph{partition} $\pi=\{\pi_1,\pi_2,...,\pi_k\}$ of $[n]$ is
a collection of disjoint subsets $\pi_i \subseteq [n]$ such that $[n]=\cup_i \pi_i$. We call
each $\pi(i)$ a \emph{block} of $\pi$ and denote by $|\pi|$ the number of blocks in $\pi$. We are 
also defining a standard order in which to read the blocks of a partition $\pi$. We read the parts 
in 
increasing order of its minimal elements, that is, $\min(\pi_i) < \min(\pi_j)$ if and only if 
$i<j$. We 
call this order the \emph{minimal order} of the blocks of $\pi$.

We say that a partition $\pi$ of $[n]$ is \emph{noncrossing} if 
there are no $x<y<z<w \in [n]$ such that $x,z \in \pi_i$ and $y,w \in \pi_j$ with $i \neq 
j$. See Figure \ref{example:noncrossingpartition} for an example of noncrossing and crossing 
partitions of $8$; note that in Figure~\ref{fig:crossing} the subsets $\{3,5\}$ and 
$\{4,6\}$ belong to different blocks. 
We denote by $\NC_{n}$ the set of noncrossing partitions of 
$[n]$. In 
addition, let  $\pi_i$ and $\pi_j$ be two blocks of $\pi$, we say that 
$\pi(j)$ is \emph{nested} inside $\pi_i$ if there is a bipartition $\pi_i=B_1\cup B_2$ with both 
$B_1$ and $B_2$ 
nonempty such that $x \le y \le z$ for all $x \in B_1$, $y \in \pi_j$ and $z\in B_2$.

\begin{figure}
 \begin{subfigure}[h]{0.4\linewidth}
\centering
\begin{tikzpicture}[scale=0.8]
   \tikzstyle{every node}=[inner sep=0pt, minimum width=4pt,scale=0.8]
   \def\n{8}
   \foreach \x in {1,2,...,\n}
   \draw {(90-360/\n*\x+360/\n+1:2cm) node {\x}};
   
   \def\inicio{1} 
   \def\col{red}
      \path[fill=\col!50] (90-360/\n*\inicio+360/\n+1:1.8cm)
      \foreach \x in {1,6,7,8}{ -- (90-360/\n*\x+360/\n+1:1.8cm)};
      
    \def\inicio{2} 
    \def\col{gray}
      \path[fill=\col!50] (90-360/\n*\inicio+360/\n+1:1.8cm)
      \foreach \x in {2,3,4,5}{ -- (90-360/\n*\x+360/\n+1:1.8cm)};
\end{tikzpicture}
\caption{Noncrossing}
\end{subfigure}
\begin{subfigure}[h]{0.4\linewidth}
\centering
\begin{tikzpicture}[scale=0.8]
   \tikzstyle{every node}=[inner sep=0pt, minimum width=4pt,scale=0.8]
   \def\n{8}
   \foreach \x in {1,2,...,\n}
   \draw {(90-360/\n*\x+360/\n+1:2cm) node {\x}};
   
   \def\inicio{1} 
   \def\col{red}
      \path[fill=\col!50] (90-360/\n*\inicio+360/\n+1:1.8cm)
      \foreach \x in {1,4,6,8}{ -- (90-360/\n*\x+360/\n+1:1.8cm)};
      
    \def\inicio{2} 
    \def\col{gray}
      \path[fill=\col!50] (90-360/\n*\inicio+360/\n+1:1.8cm)
      \foreach \x in {2,3,5,7}{ -- (90-360/\n*\x+360/\n+1:1.8cm)};
\end{tikzpicture}
\caption{Crossing}\label{fig:crossing}
\end{subfigure}
\caption{Example of noncrossing and crossing partitions of $8$}
\label{example:noncrossingpartition}

\end{figure}
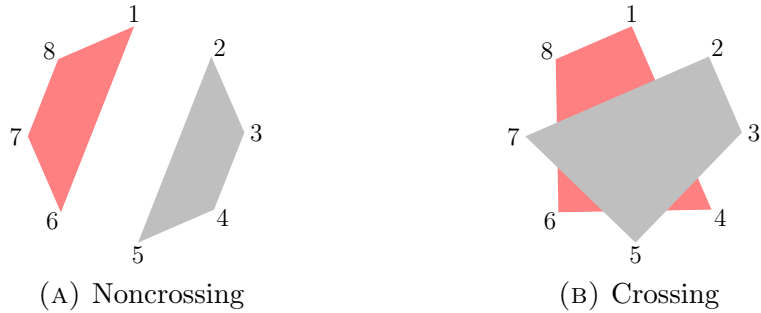

\begin{definition}[Noncrossing $s$-partitions]
\label{def_noncrossing_sPartitions}
To each partition $\pi=\{\pi_1,\pi_2,...,\pi_k\}$ of $[n]$ (where the blocks are ordered according its minimal order) we associate the 
composition $\mu(\pi):=(|\pi_1|,|\pi_2|,...,|\pi_k|)$. For $s \in \comp$, we say that 
$\pi$ is an \emph{$s$-partition} if $s\le\mu(\pi)$, that is, 
if $s$ is a refinement of~$\mu(\pi)$.
We denote the set of noncrossing $s$-partitions by $\NC_{s}$. 
\end{definition}

For example, 
$\NC_{(1^n)}$ is the set of noncrossing partitions and $\NC_{(k^n)}$ is the set of 
\emph{$k$-divisible} noncrossing partitions of $[nk]$, that is, the set of noncrossing partitions 
of $[nk]$ whose blocksizes are all divisible by $k$. The noncrossing partition in 
Figure \ref{fig:examplesnoncrossingpartition} is a $(2,3,3,1,4)$-partition, 
$(1,1,2,1,3,5)$-partition, $(1^{13})$-partition and so on; but for example is not a
$(3,3,3,4)$-partition since $(5,3,5)$ is not refined by $(3,3,3,4)$.

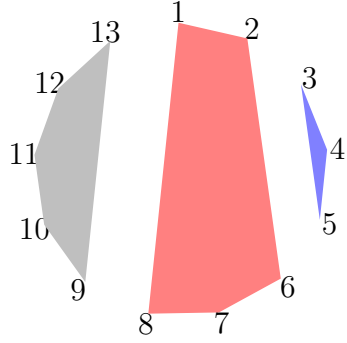
\begin{figure}
\centering
\begin{tikzpicture}[scale=0.7]
 
   \tikzstyle{every node}=[inner sep=0pt, minimum width=4pt]
   \def\n{13}
   
   \foreach \x in {1,2,...,13}
   \draw {(90-360/\n*\x+360/\n+1:3cm) node {\x}};
   
   \def\inicio{3} 
   \def\col{blue}
      \path[fill=\col!50] (90-360/\n*\inicio+360/\n+1:2.8cm)
      \foreach \x in {3,4,5}{ -- (90-360/\n*\x+360/\n+1:2.8cm)};
      
    \def\inicio{1} 
    \def\col{red}
      \path[fill=\col!50] (90-360/\n*\inicio+360/\n+1:2.8cm)
      \foreach \x in {1,2,6,7,8}{ -- (90-360/\n*\x+360/\n+1:2.8cm)};
      
      \def\inicio{9} 
    \def\col{gray}
      \path[fill=\col!50] (90-360/\n*\inicio+360/\n+1:2.8cm)
      \foreach \x in {9,10,11,12,13}{ -- (90-360/\n*\x+360/\n+1:2.8cm)};
      
\end{tikzpicture}
\caption{Example of a noncrossing $(2,3,3,1,4)$-partition}
\label{fig:examplesnoncrossingpartition}
\end{figure}

\subsubsection{Bijection between $s$-trees and noncrossing $(s-\bf{1})$-partitions}
\label{sec:bijection-trees-non-crossing}
Let $s \in \comp$ be such that $s_i \ge 2$ for all~$i$ and recall 
that $s-\mathbf{1}:=(s(1)-1,s(2)-1,\dots,s(\ell(s))-1)$. We also define $\emptyset-\mathbf{1}:=\emptyset$.
We will define a bijection $\phi: \T_s \rightarrow  \NC_{s-\mathbf{1}}$ by labeling the caverns of the tree in preorder and grouping the labels according to certain rule.

Recall that given two nodes $x$ and $y$ of $T$, $y$ is called a \emph{descendant} of~$x$ if~$x$ belongs to the 
unique 
path from $y$ to the root. A node $y$ is called a \emph{left descendant} of $x$ if $y$ is the 
minimal child of 
$x$ 
or the minimal child of a left descendant of $x$. We denote by $\leftDes(x)$ the set of left 
descendant internal nodes of $T$ union with $\{x\}$. 

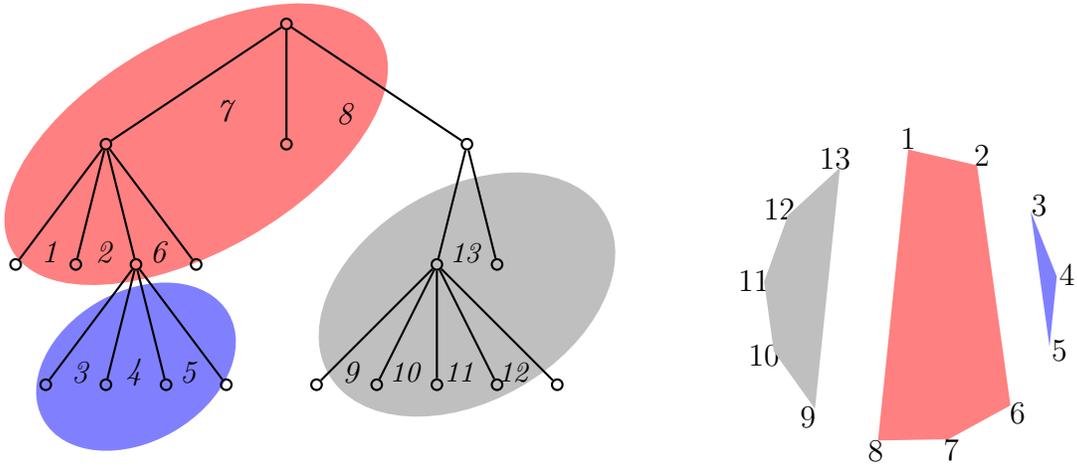
\begin{figure}
\centering
\begin{tikzpicture}[thick,scale=0.8]
\usetikzlibrary{arrows,shapes}

\node[fill=red!50,red!50,draw,ellipse, rounded corners,rotate=30, minimum width=200pt,minimum height=100pt,scale=0.8] at  (0.5,4) {};

\node[fill=blue!50,blue!50,draw,ellipse,rounded corners,rotate=30, minimum width=100pt,minimum height=70pt,scale=0.8] at  (-0.5,0.3) {};

\node[fill=gray!50,gray!50,draw, ellipse, rounded corners,rotate=30, minimum width=150pt,minimum height=100pt,scale=0.8] at  (5,1.5) {};

    \tikzstyle{every node}=[inner sep=0pt, minimum width=4pt]

    \draw (-1.9,2.2) node []{\emph{1}};
    \draw (-1,2.2) node []{\emph{2}};
    \draw (-1.4,0.2) node []{\emph{3}};
    \draw (-0.5,0.2) node []{\emph{4}};
    \draw (0.4,0.2) node []{\emph{5}};
    \draw (-0.1,2.2) node []{\emph{6}};
    \draw (1,4.5) node []{\emph{7}};
    \draw (3,4.5) node []{\emph{8}}; 
    \draw (3.1,0.2) node []{\emph{9}};
    \draw (4,0.2) node []{\emph{\small 10}};
    \draw (4.9,0.2) node []{\emph{\small 11}};
    \draw (5.8,0.2) node []{\emph{\small 12}};
 \draw (5,2.2) node []{\emph{\small 13}};

 \tikzstyle{every node}=[circle, draw,
                        inner sep=0.5pt, minimum width=4pt,font=\small]
\node (a1) at (2,6){};
\node (b1) at (-1,4){};
\node (b2) at (2,4){};
\node (b3) at (5,4){};

\node (c1) at (-2.5,2) {};
\node (c2)  at (-1.5,2){};
\node (c3)  at (-0.5,2){};
\node (c4) at (0.5,2){};
\node (d1)  at  (4.5,2){};
\node (d2) at (5.5,2) {};

\node (f1) at (-2,0){};
\node (f2)  at (-1,0){};
\node (f3)  at (0,0){};
\node (f4) at (1,0){};
    
\node (e1) at (2.5,0){};
\node (e2)  at (3.5,0){};
\node (e3)  at (4.5,0){};
\node (e4) at (5.5,0){};
\node (e5) at (6.5,0){};

\draw (a1)--(b1);
\draw (a1)--(b2);
\draw	 (a1)--(b3);
\draw (b1)--(c1);
\draw (b1)--(c2);
\draw (b1)--(c3);
\draw (b1)--(c4);
\draw (b3)--(d1);
\draw (b3)--(d2);
\draw (c3)--(f1);
\draw (c3)--(f2);
\draw (c3)--(f3);
\draw (c3)--(f4);
\draw (d1)--(e1);
\draw (d1)--(e2);
\draw (d1)--(e3);
\draw (d1)--(e4);
\draw (d1)--(e5);

\end{tikzpicture}\quad \quad \quad
\begin{tikzpicture}[scale=0.7]
 
   \tikzstyle{every node}=[inner sep=0pt, minimum width=4pt]
   \def\n{13}
   
   \foreach \x in {1,2,...,13}
   \draw {(90-360/\n*\x+360/\n+1:3cm) node {\x}};
   
   \def\inicio{3} 
   \def\col{blue}
      \path[fill=\col!50] (90-360/\n*\inicio+360/\n+1:2.8cm)
      \foreach \x in {3,4,5}{ -- (90-360/\n*\x+360/\n+1:2.8cm)};
      
    \def\inicio{1} 
    \def\col{red}
      \path[fill=\col!50] (90-360/\n*\inicio+360/\n+1:2.8cm)
      \foreach \x in {1,2,6,7,8}{ -- (90-360/\n*\x+360/\n+1:2.8cm)};
      
      \def\inicio{9} 
    \def\col{gray}
      \path[fill=\col!50] (90-360/\n*\inicio+360/\n+1:2.8cm)
      \foreach \x in {9,10,11,12,13}{ -- (90-360/\n*\x+360/\n+1:2.8cm)};
      
\end{tikzpicture}
\caption{Example of the bijection between $s$-trees and noncrossing $(s-\bf 1)$-partitions. The 
caverns of the tree are ordered in preorder.}
\label{fig:preordercavern}
\end{figure}

Let $C_1,C_2,\dots,C_{|s-\mathbf{1}|}$ be the preorder of the caverns of $T \in \T_s$. We 
define the function $\phi$ as follows:
Let $v_1,\dots,v_a$ be the internal nodes of $T$ in preorder, and let $v_{i_1},\dots,v_{i_k}$ be 
the internal nodes that are not a minimal child of another internal node. In particular, 
$v_{i_1}=v_1$ is the root of the tree.
Denote by 
\[
\pi_j := \bigcup_{v\in \leftDes(v_{i_j})}  \caverns(v),
\] 
where we are identifying a cavern $C_i$ with its index $i$.
The partition associated to $T$ is defined
\[
\phi(T):=\{\pi_1,\dots,\pi_k\}.
\]
An example of this map is illustrated in Figure~\ref{fig:preordercavern}.

For simplicity, we denote by $\{B_1,\dots, B_k\}$ the partition of $[a]$ determined by the indices $i_1,\dots,i_k$ in the following way:

\begin{equation}\label{eq:partitionBj}
B_j=     \left\{ \begin{array}{lcl}
         [i_j,i_{j+1}) & \mbox{for} & 1\leq j<a \\ 
         {[i_k,a]} & \mbox{for} & j=k. \\ 
                       \end{array}\right.
\end{equation}

\begin{lemma}\label{lem:leftDes_internalnodes}
For $1\leq j \leq k$, the internal nodes of $\leftDes(v_{i_j})$ are the nodes $\{v_\ell\}_{\ell\in 
B_j}$.
\end{lemma}
\begin{proof}
Let $1\leq j < k$. The set $\leftDes(v_{i_j})$ is obtained from $v_{i_j}$ by consecutively adding 
in preorder left minimal children whenever possible. The process stops when following preorder we 
find the first internal node that is not a minimal child. Since this internal node is 
$v_{i_{j+1}}$ then the internal nodes of $\leftDes(v_{i_j})$ are exactly the nodes 
$\{v_\ell\}_{\ell\in [i_j,i_{j+1})}$. 
If $j=k$, the process does not stop until covering all remaining internal nodes of $T$, and so, the 
internal nodes of $\leftDes(v_{i_k})$ are the nodes $\{v_\ell\}_{\ell\in [i_k,a]}$.
\end{proof}

\begin{lemma}
 The partition $\phi(T)$ is a noncrossing $(s-\mathbf{1})$-partition. 
\end{lemma}

\begin{proof}
To see that $\phi(T)$ is noncrossing consider $1\leq r < s \leq k$. Since  $v_{i_r}$ comes before 
$v_{i_s}$ in preorder we have two possibilities, either $v_{i_s}$ is a descendant of $v_{i_r}$ or 
they are unrelated. And note that in these two cases preorder implies that $\pi_s$ 
is either nested in $\pi_r$ or that all the caverns in $\pi_r$ come in preorder before the ones 
in $\pi_s$.

To see that $\phi(T)$ is an $(s-\mathbf{1})$-partition note that an internal node $v_\ell$ of $T$ 
has exactly $s(\ell)-1$ caverns. Lemma~\ref{lem:leftDes_internalnodes} then implies that 
$|\pi_j| =\sum_{\ell\in B_j} (s(\ell)-1)$ for $1\leq j \leq k$. Since the sets $B_j$ form a 
partition of $[a]$ with increasing adjacent consecutive parts, then $s-\bf 1$ is a refinement of 
$\mu(\pi)$. 
\end{proof}

Given a noncrossing $(s-\bf 1)$-partition $\pi=\{\pi_1<\pi_2<\cdots <\pi_k\}$ we will 
construct a tree  ${T=\psi(\pi)}$ such that $\phi(T)=\pi$. Since 
$\pi$ is an $(s-\bf 1)$-partition we have that $(|\pi_1|,|\pi_2|,\dots |\pi_k|)$ is refined by 
$(s-\bf 1)$ meaning that $s$ can be written as $s=\tilde s_1 \oplus \tilde s_2\oplus\cdots \tilde 
s_k$ with 
$|\pi_j|=|\tilde s_j - \bf 1|$.

We say that a tree is a \emph{left descendant tree} if every internal node is either the root or 
the minimal (leftmost) child of another internal node. For $1\leq 
j\leq k$, let $\tau_j$ be the (unique) left descendant tree with signature $\tilde s_j$, and label 
the caverns of $\tau_j$ in preorder with the elements of $\pi_j$ (note this is possible since the 
number 
of caverns of $\tau_j$ is exactly equal to $|\pi_j|$). 

\begin{lemma}\label{lemma:gluingleftdescendanttrees}
 Let $\{\tau_1,\dots,\tau_k\}$ be a set of left descendant trees such that the caverns of tree 
$\tau_j$ have been labeled in preorder with the set $\pi_j$ where $\pi=\{\pi_1<\pi_2<\cdots 
<\pi_k\}$ is a noncrossing $(s-\bf 1)$-partition of $[n]$. Then there is a unique way to glue 
together the trees $\tau_j$ such that they form a planar rooted tree $T$ satisfying:
\begin{enumerate}
 \item\label{condition:gluing1} No pair of trees $\tau_i$ and $\tau_j$ together form a larger left 
descendant subtree in $T$.
 \item\label{condition:gluing2} The caverns of $T$ are labeled in preorder.
 \item\label{condition:gluing3} $\signat(T)=s$.
\end{enumerate} 
\end{lemma}
\begin{proof}
 By condition (\ref{condition:gluing1}) we cannot glue two left 
descendant trees, say $\tau_i$ and $\tau_j$ with $i < j$, by attaching the root of $\tau_i$ to the 
left-most leaf of $\tau_j$ since they will form a larger left descendant tree. Now condition 
(\ref{condition:gluing2}) and the fact of the partition $\pi$ is noncrossing implies that for $i<j$ 
either  $\tau_j$ is nested in $\tau_i$ or all labels of $\tau_j$ are larger than all the labels of 
$\tau_i$. So given tree $T_1=\tau_1$ (whose $\min \pi_1=1$) there is a unique way of gluing 
$\tau_2$ to obtain $T_2$, its root needs to be glued in preorder to the leaf of $\tau_1$ 
immediately after the cavern labeled $\min \pi_2-1$ (if $\min \pi_2-1$ were not a cavern of 
$\tau_1$ then $\pi$ would not be a partition of $[n]$) we repeat this process inductively by gluing 
$\tau_j$ to $T_{j-1}$ to obtain $T_j$, its root glued in preorder to the leaf of 
$T_{j-1}$ immediately after the cavern labeled $\min \pi_j-1$. This construction preserves the 
preorder labeling and always glues $\tau_j$ to a leaf of $T_{j-1}$ that is not the leftmost leaf 
of another tree $\tau_i$. Now note that $\signat(\tau_1)=\tilde s_1$ and since $T_1=\tau_1$ is a 
left descendant tree then when attaching $\tau_2$ to form $T_2$ all the internal nodes of $T_1$ 
come in preorder before the internal nodes of $\tau_2$ hence $\signat(T_2)=\tilde s_1 \oplus \tilde 
s_2$. In general when attaching $\tau_j$ to $T_{j-1}$ to obtain $T_j$, all the internal nodes of 
$\tau_j$ come last in preorder and hence $\signat(T_j)=\signat(T_{j-1})\oplus\tilde s_j$, this 
inductively implies condition (\ref{condition:gluing3}).
\end{proof}

Given the conditions of Lemma \ref{lemma:gluingleftdescendanttrees} where the 
labels of the caverns in tree $\tau_j$ are the elements of $\pi_j$, we define $T:=\psi(\pi)$ where 
$T$ is the unique tree that the lemma predicts. 

\begin{figure}[tb]
\centering
\begin{tikzpicture}[thick,scale=0.8]
\usetikzlibrary{arrows,shapes}

\begin{scope}
\node[fill=red!50,red!50,draw,ellipse, rounded corners,rotate=30, minimum width=200pt,minimum height=100pt,scale=0.8] at  (0.5,4) {};



    \tikzstyle{every node}=[inner sep=0pt, minimum width=4pt]

    \draw (-1.9,2.2) node []{\emph{1}};
    \draw (-1,2.2) node []{\emph{2}};
    \draw (-0.1,2.2) node []{\emph{6}};
    \draw (1,4.5) node []{\emph{7}};
    \draw (3,4.5) node []{\emph{8}}; 

 \tikzstyle{every node}=[circle, draw,
                        inner sep=0.5pt, minimum width=4pt,font=\small]
\node (a1) at (2,6){};
\node (b1) at (-1,4){};
\node (b2) at (2,4){};
\node (b3) at (5,4){};

\node (c1) at (-2.5,2) {};
\node (c2)  at (-1.5,2){};
\node (c3)  at (-0.5,2){};
\node (c4) at (0.5,2){};

    

\draw (a1)--(b1);
\draw (a1)--(b2);
\draw	 (a1)--(b3);
\draw (b1)--(c1);
\draw (b1)--(c2);
\draw (b1)--(c3);
\draw (b1)--(c4);
\end{scope}

\begin{scope}[xshift=7.5cm, yshift=4cm]
\node[fill=blue!50,blue!50,draw,ellipse,rounded corners,rotate=30, minimum width=100pt,minimum height=70pt,scale=0.8] at  (0.5,0.5) {};

    \tikzstyle{every node}=[inner sep=0pt, minimum width=4pt]

    \draw (-0.4,0.2) node []{\emph{3}};
    \draw (0.5,0.2) node []{\emph{4}};
    \draw (1.4,0.2) node []{\emph{5}};

 \tikzstyle{every node}=[circle, draw,
                        inner sep=0.5pt, minimum width=4pt,font=\small]

\node (c4) at (0.5,2){};
\node (f1) at (-1,0){};
\node (f2)  at (0,0){};
\node (f3)  at (1,0){};
\node (f4) at (2,0){};

\draw (c4)--(f1);
\draw (c4)--(f2);
\draw (c4)--(f3);
\draw (c4)--(f4);

\end{scope}

\begin{scope}[xshift=8.5cm, yshift=2cm]
\node[fill=gray!50,gray!50,draw, ellipse, rounded corners,rotate=30, minimum width=150pt,minimum height=100pt,scale=0.8] at  (5,1.5) {};

    \tikzstyle{every node}=[inner sep=0pt, minimum width=4pt]

    \draw (3.1,0.2) node []{\emph{9}};
    \draw (4,0.2) node []{\emph{\small 10}};
    \draw (4.9,0.2) node []{\emph{\small 11}};
    \draw (5.8,0.2) node []{\emph{\small 12}};
 \draw (5,2.2) node []{\emph{\small 13}};

 \tikzstyle{every node}=[circle, draw,
                        inner sep=0.5pt, minimum width=4pt,font=\small]
\node (b3) at (5,4){};

\node (d1)  at  (4.5,2){};
\node (d2) at (5.5,2) {};

\node (e1) at (2.5,0){};
\node (e2)  at (3.5,0){};
\node (e3)  at (4.5,0){};
\node (e4) at (5.5,0){};
\node (e5) at (6.5,0){};

\draw (b3)--(d1);
\draw (b3)--(d2);

\draw (d1)--(e1);
\draw (d1)--(e2);
\draw (d1)--(e3);
\draw (d1)--(e4);
\draw (d1)--(e5);

\end{scope}

\end{tikzpicture}\quad \quad \quad
\caption{The recursive construction of the tree $T=\psi(\pi)$ for the noncrossing $(s-\bf 1)$-
partition $\pi=\{\{1,2,6,7,8\},\{3,4,5\},\{9,10,11,12,13\}\}$ with $s=(3,4,4,2,5)$.}
\label{fig:preordercavern2}
\end{figure}
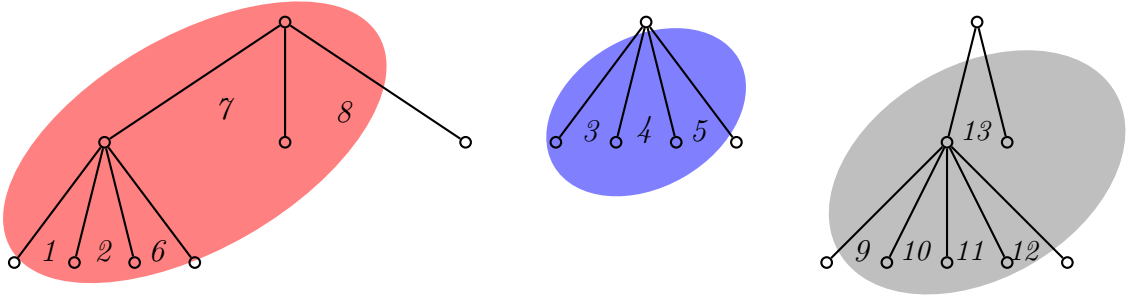

\begin{theorem}\label{thm:bijection_trees_paths}
Let $s \in \comp$ be such that $s_i \ge 2$ for all~$i$.
The map $\phi:\T_s\rightarrow \NC_{s- \bf 1}$ is a bijection between $s$-trees and 
noncrossing $(s-\mathbf{1})$-partitions. The inverse of $\phi$ is~$\psi$.
\end{theorem}

\begin{proof}
If $\pi=\phi(T)$ then by the definition of $\phi$ every block $\pi_j \in \pi$ gets its labels from 
the caverns of a maximal left descendant subtree $\tau_j$ in $T$. If $\tau_j$ has 
$\signat(\tau_j)=\tilde s_j$ then $s=\tilde s_1\oplus\tilde s_2\oplus\cdots \tilde s_k$ since all 
of these 
left descendant subtrees are consecutive in preorder by Lemma~\ref{lem:leftDes_internalnodes}. But 
with these conditions together $\tilde s_j$ and $\pi_j$ recover $\tau_j$ uniquely and Lemma 
\ref{lemma:gluingleftdescendanttrees} says that there is a unique way to glue the $\tau_j$ together 
so $\psi(\pi)=T$. Similarly, if $T=\psi(\pi)$ then by construction a maximal left descendant 
subtree of $T$ have its caverns labeled by a block of $\pi$ and by Lemma 
\ref{lemma:gluingleftdescendanttrees} the labels of the caverns of $T$ are in preorder. Since each 
block of $\phi(T)$ is determined by the set of labels in the caverns of a maximal left descendant 
subtree when the caverns of $T$ have been labeled in preorder then $\phi(T)=\pi$. 
\end{proof}

\subsection{Noncrossing  matchings}

A \emph{complete matching} in the \emph{complete graph} $K_{2n}$ (the graph $(V,E)$ with 
vertex set $V=[2n]$ and with edge set $E=\{S\subseteq [2n]\,\mid\, |S|=2\}$) is a partition of the 
set $[2n]$ in blocks of cardinality $2$. We say that the matching is \emph{noncrossing} if this 
partition is \emph{noncrossing}. It is known that the set of complete noncrossing matchings in 
$K_{2n}$ is a Catalan family (see for example \cite{Stanley2015}). For any $n \ge 1$ 
the \emph{complete 
hypergraph} 
$\K_{n}$ on $n$ vertices is the pair $(V,E)$ where $V=[n]$ is the set of \emph{vertices} and 
$E=\{S\subseteq [2n]\,\mid\, S\neq \emptyset \}$ the set of \emph{edges} (Note that in this
definition we consider vertices as edges of cardinality one). 

\begin{definition}[Complete noncrossing $s$-matchings]
\label{def_sMatchings}
For $s \in \comp$, a \emph{complete 
noncrossing $s$-matching} in 
$\K_{|s|}$ is a 
noncrossing partition $M$ of the set $[|s|]$ such that $M$ satisfies $|M_i|=s(i)$ when we order 
the blocks of $M=\{M_1<M_2<\dots <M_{a}\}$ in the minimal order (such that $\min M_i < \min M_j$ 
whenever $i<j$). We denote by $\CM_s$ the set of complete noncrossing $s$-matchings.
\end{definition}

Figure \ref{fig:examplessmatching} illustrates an example of a complete noncrossing $(3,4,4,2,5)$-matching.

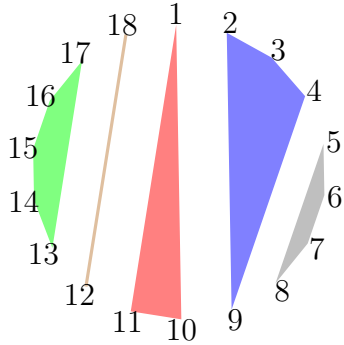
\begin{figure}
\centering
    \begin{tikzpicture}[scale=0.7]
     
    \tikzstyle{every node}=[inner sep=0pt, minimum width=4pt]
   \def\n{18}
   
   \def\inicio{2} 
   \def\col{blue}
      \path[fill=\col!50] (90-360/\n*\inicio+360/\n+1:2.8cm)
      \foreach \x in {2,3,4,9}{ -- (90-360/\n*\x+360/\n+1:2.8cm)};
      
    \def\inicio{1} 
    \def\col{red}
      \path[fill=\col!50] (90-360/\n*\inicio+360/\n+1:2.8cm)
      \foreach \x in {1,10,11}{ -- (90-360/\n*\x+360/\n+1:2.8cm)};
      
      \def\inicio{5} 
    \def\col{gray}
      \path[fill=\col!50] (90-360/\n*\inicio+360/\n+1:2.8cm)
      \foreach \x in {5,6,7,8}{ -- (90-360/\n*\x+360/\n+1:2.8cm)};
      
      \def\inicio{12} 
    \def\col{brown}
      \path[draw,very thick, fill=\col!50,\col!50] (90-360/\n*\inicio+360/\n+1:2.8cm)
      \foreach \x in {12,18}{ -- (90-360/\n*\x+360/\n+1:2.8cm)};

 \def\inicio{13} 
    \def\col{green}
      \path[fill=\col!50] (90-360/\n*\inicio+360/\n+1:2.8cm)
      \foreach \x in {13,14,15,16,17}{ -- (90-360/\n*\x+360/\n+1:2.8cm)};

   \foreach \x in {1,2,...,\n}
   \draw {(90-360/\n*\x+360/\n+1:3cm) node {\x}};

    \end{tikzpicture}
\caption{Example of a complete $(3,4,4,2,5)$-matching}
\label{fig:examplessmatching}
\end{figure}

\begin{remark}
 Note that for $s\in\comp$ both, $(s-\mathbf{1})$-partitions and complete $s$-matchings are 
partitions of different sets ($[|s|-\ell(s)+1]$ and $[|s|]$ respectively). We use the name of 
complete $s$-matchings to denote the later type of partitions because they provide 
with a suitable generalization of the concept of complete matching in $K_{2n}$. Also note that a 
given partition can be a $\nu$-paritition for different values of $\nu$ however a complete 
matching $M$ in $\K_{|s|}$ can only be a complete $s$-matching for the particular 
$s=(|M_1|,|M_2|,\dots,|M_a|)$.
\end{remark}

\subsubsection{Bijection between $s$-trees and complete noncrossing $s$-matchings}
For $T \in \T_s$ let $v_0, v_1,v_2,\dots,v_{|s|}$ be the listing of the nodes of $T$ in preorder 
(where $v_0$ is the root of $T$) and 
let $v_{i_1},v_{i_2},\dots,v_{i_{a}}$ be the sublist of internal nodes in preorder. For each 
internal node $x$ 
let $\child(x)$ be the set of children of $x$. The matching 
$\varphi(T):=\{M_1, M_2,\cdots,M_a\}$ is defined as the partition with blocks 
$M_k=\{j\,\mid\, v_j \in \child(v_{i_k})\}$. See Figure \ref{fig:treestomatchings} for an example 
of the map $\phi$.

\begin{figure}
\centering
\begin{tikzpicture}[]
\usetikzlibrary{arrows,shapes}

\begin{scope}[thick,scale=0.45]
\node[fill=gray!50,gray!50,draw, ellipse, rounded corners,rotate=0, minimum width=70pt,minimum height=30pt,scale=0.8] at  (-1,-0.3) {};

\node[fill=green!50,green!50,draw, ellipse, rounded corners,rotate=0, minimum width=80pt,minimum height=30pt,scale=0.8] at  (5,-0.5) {};

\node[fill=red!50,red!50,draw, ellipse, rounded corners,rotate=0, minimum width=130pt,minimum height=30pt,scale=0.8] at  (2,4.5) {};

\node[fill=blue!50,blue!50,draw, ellipse, rounded corners,rotate=0, minimum width=90pt,minimum height=30pt,scale=0.8] at  (-1,2) {};

\node[fill=brown!50,brown!50,draw, ellipse, rounded corners,rotate=0, minimum width=60pt,minimum height=30pt,scale=0.8] at  (5,2) {};

\tikzstyle{every node}=[circle, draw,
                        inner sep=0.5pt, minimum width=4pt,font=\small,scale=0.9]
\node (a1) at (2,6)[pin={[color=red,pin distance=3pt]135:$0$}]{};
\node (b1) at (-1,4)[pin={[color=red,pin distance=3pt]135:$1$}]{};
\node (b2) at (2,4)[pin={[color=red,pin distance=3pt]135:$10$}]{};
\node (b3) at (5,4)[pin={[color=red,pin distance=3pt]45:$11$}]{};

\node (c1) at (-2.5,2) [pin={[color=red,pin distance=3pt]225:$2$}]{};
\node (c2)  at (-1.5,2)[pin={[color=red,pin distance=3pt]225:$3$}]{};
\node (c3)  at (-0.5,2)[pin={[color=red,pin distance=3pt]220:$4$}]{};
\node (c4) at (0.5,2)[pin={[color=red,pin distance=3pt]-45:$9$}]{};
\node (d1)  at  (4.5,2)[pin={[color=red,pin distance=3pt]135:$12$}]{};
\node (d2) at (5.5,2) [pin={[color=red,pin distance=3pt]45:$18$}] {};

\node (f1) at (-2,0)[pin={[color=red,pin distance=3pt]225:$5$}]{};
\node (f2)  at (-1,0)[pin={[color=red,pin distance=3pt]225:$6$}]{};
\node (f3)  at (0,0)[pin={[color=red,pin distance=3pt]225:$7$}]{};
\node (f4) at (1,0)[pin={[color=red,pin distance=3pt]230:$8$}]{};
    
\node (e1) at (2.7,0)[pin={[color=red,pin distance=3pt]310:$13$}]{};
\node (e2)  at (3.5,0)[pin={[color=red,pin distance=3pt]315:$14$}]{};
\node (e3)  at (4.5,0)[pin={[color=red,pin distance=3pt]315:$15$}]{};
\node (e4) at (5.5,0)[pin={[color=red,pin distance=3pt]315:$16$}]{};
\node (e5) at (6.5,0)[pin={[color=red,pin distance=3pt]315:$17$}]{};

\draw (a1)--(b1);
\draw (a1)--(b2);
\draw	 (a1)--(b3);
\draw (b1)--(c1);
\draw (b1)--(c2);
\draw (b1)--(c3);
\draw (b1)--(c4);
\draw (b3)--(d1);
\draw (b3)--(d2);
\draw (c3)--(f1);
\draw (c3)--(f2);
\draw (c3)--(f3);
\draw (c3)--(f4);
\draw (d1)--(e1);
\draw (d1)--(e2);
\draw (d1)--(e3);
\draw (d1)--(e4);
\draw (d1)--(e5);

\end{scope}
\begin{scope}[xshift=6cm,yshift=1.5cm, scale=0.6]
 
   \tikzstyle{every node}=[inner sep=0pt, minimum width=4pt]
   \def\n{18}
   
   \def\inicio{2} 
   \def\col{blue}
      \path[fill=\col!50] (90-360/\n*\inicio+360/\n+1:2.8cm)
      \foreach \x in {2,3,4,9}{ -- (90-360/\n*\x+360/\n+1:2.8cm)};
      
    \def\inicio{1} 
    \def\col{red}
      \path[fill=\col!50] (90-360/\n*\inicio+360/\n+1:2.8cm)
      \foreach \x in {1,10,11}{ -- (90-360/\n*\x+360/\n+1:2.8cm)};
      
      \def\inicio{5} 
    \def\col{gray}
      \path[fill=\col!50] (90-360/\n*\inicio+360/\n+1:2.8cm)
      \foreach \x in {5,6,7,8}{ -- (90-360/\n*\x+360/\n+1:2.8cm)};
      
      \def\inicio{12} 
  \def\col{brown}
      \path[draw,very thick, fill=\col!50,\col!50] (90-360/\n*\inicio+360/\n+1:2.8cm)
      \foreach \x in {12,18}{ -- (90-360/\n*\x+360/\n+1:2.8cm)};

 \def\inicio{13} 
    \def\col{green}
      \path[fill=\col!50] (90-360/\n*\inicio+360/\n+1:2.8cm)
      \foreach \x in {13,14,15,16,17}{ -- (90-360/\n*\x+360/\n+1:2.8cm)};

   \foreach \x in {1,2,...,\n}
   \draw {(90-360/\n*\x+360/\n+1:3cm) node {\x}};

\end{scope}

  \end{tikzpicture}
\caption{Example of the bijection between $s$-trees and noncrossing $s$-matchings. The 
nodes of the tree are ordered in preorder.}
\label{fig:treestomatchings}
\end{figure}
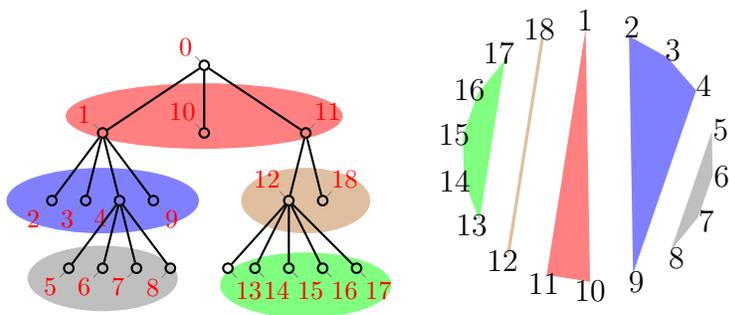

\begin{lemma}\label{lemma:phiminimal} Let $v_{i_k}$ be the $k$-th internal node of $T$ in preorder 
then $$\min M_k=i_k+1.$$ Consequently, $M_1, M_2,\dots,M_a$ is the minimal order listing of 
$\phi(T)$.
\end{lemma}
\begin{proof}
 An internal node and its leftmost child are consecutive in preorder, then $\min M(k)= 
\min \{j\,\mid\, v_j \in \child(v_{i_k})\}=i_k+1$.
\end{proof}

\begin{lemma}
 The partition $\varphi(T)$ is a complete noncrossing $s$-matching.
\end{lemma}
\begin{proof}
{\bf  $\varphi(T)$ is noncrossing:} Two different blocks of $\varphi(T)$ come from two different internal 
nodes 
$v_{i_{k}}$ and $v_{i_{l}}$, say with $v_{i_{k}} < v_{i_{l}}$ in preorder. If $v_{i_{k}}$ and 
$v_{i_{l}}$ are unrelated all the nodes in $\child(v_{i_{k}})$ occur before the nodes in 
$\child(v_{i_{l}})$ in preorder. If $v_{i_{k}}$ is an ancestor of $v_{i_{l}}$ then all the nodes 
in $\child(v_{i_{l}})$ are also descendant of exactly one child of $v_{i_{k}}$, hence they occur
in between two preorder consecutive elements of $\child(v_{i_{k}})$ and $M_l$ is nested in $M_k$.

{\bf  $\varphi(T)$ is a complete $s$-matching:}  By Lemma \ref{lemma:phiminimal} $M_1,M_2\dots,M_a$ is 
the minimal 
order listing of $\varphi(T)$ and $|M_k|=\deg(v_{i_k})=s(k)$  then 
\begin{align*}
    (|M_1|,|M_2|\dots,|M_a|)=(s(1),\dots,s(a))=s.
\end{align*}
\end{proof}

The inverse map $\gamma:\CM_s\rightarrow \T_s$ is defined as follows. Let 
$M=\{M_1,M_2,\dots,M_a\} \in \CM_{s}$ be a 
complete noncrossing $s$-matching whose blocks are ordered according to the minimal order and 
recall that by definition $|M_i|=s(i)$. Define 
the tree $\gamma(M)$ as follows (see Figure \ref{fig:map_matching_to_tree} for an example):

\begin{enumerate}
\item First, start with the root $r=v_0$ with $|M_1|$ children with increasing left-to-right 
labels  $\{v_j\,\mid\, j\in M_1\}$ and call the resulting tree $\tau_1$.
\item Attach an internal node to $\tau_1$ in $v_{\min{M_2}-1}$ with $s(2)$ children with increasing 
left-to-right labels $\{v_j\,\mid\, j\in M_2\}$ and call the resulting tree $\tau_2$.
\item In general, at the $j$th step of the process we have a tree $\tau_{j-1}$ with $j-1$ internal 
nodes $v_{i_1},\dots,v_{i_j}$ ordered in preorder. The $j$th tree $\tau_j$ is obtained from 
$\tau_{j-1}$ by attaching an internal node in $v_{\min{M_j}-1}$ with $s(j)$ children with 
increasing left-to-right labels $\{v_j\,\mid\, j\in M_j\}$.
\item The process finishes after attaching the $a$ internal nodes $v_{i_1},\dots 
,v_{i_a}$.
\end{enumerate}

\begin{figure}[htb]
\centering
 \begin{tikzpicture}[]
\usetikzlibrary{arrows,shapes}

\begin{scope}[xshift=0cm,thick,scale=0.35]

\tikzstyle{every node}=[circle, draw,
                        inner sep=0.5pt, minimum width=4pt,font=\small,scale=0.9]
\node (a1) at (2,6)[pin={[color=red,pin distance=3pt]135:$0$}]{};
\node (b1) at (-1,4)[pin={[color=red,pin distance=3pt]135:$1$}]{};
\node (b2) at (2,4)[pin={[color=red,pin distance=3pt]135:$10$}]{};
\node (b3) at (5,4)[pin={[color=red,pin distance=3pt]45:$11$}]{};

\draw (a1)--(b1);
\draw (a1)--(b2);
\draw	 (a1)--(b3);

\tikzstyle{every node}=[inner sep=0.5pt,font=\small,scale=0.9]
\node at (2,-2){
\begin{tabular}{c}
$M=\{\{1,10,11\}\}$\\
\end{tabular}
};
\end{scope}

\begin{scope}[xshift=4cm,thick,scale=0.35]

\tikzstyle{every node}=[circle, draw,
                        inner sep=0.5pt, minimum width=4pt,font=\small,scale=0.9]
\node (a1) at (2,6)[pin={[color=red,pin distance=3pt]135:$0$}]{};
\node (b1) at (-1,4)[pin={[color=red,pin distance=3pt]135:$1$}]{};
\node (b2) at (2,4)[pin={[color=red,pin distance=3pt]135:$10$}]{};
\node (b3) at (5,4)[pin={[color=red,pin distance=3pt]45:$11$}]{};

\node (c1) at (-2.5,2) [pin={[color=red,pin distance=3pt]225:$2$}]{};
\node (c2)  at (-1.5,2)[pin={[color=red,pin distance=3pt]225:$3$}]{};
\node (c3)  at (-0.5,2)[pin={[color=red,pin distance=3pt]225:$4$}]{};
\node (c4) at (0.5,2)[pin={[color=red,pin distance=3pt]45:$9$}]{};

\draw (a1)--(b1);
\draw (a1)--(b2);
\draw	 (a1)--(b3);
\draw (b1)--(c1);
\draw (b1)--(c2);
\draw (b1)--(c3);
\draw (b1)--(c4);

\tikzstyle{every node}=[inner sep=0.5pt,font=\small,scale=0.9]
\node at (2,-2){
\begin{tabular}{c}
$M=\{\{1,10,11\},\{2,3,4,9\}\}$\\
\end{tabular}
};
\end{scope}
\begin{scope}[xshift=8.5cm,thick,scale=0.35]

\tikzstyle{every node}=[circle, draw,
                        inner sep=0.5pt, minimum width=4pt,font=\small,scale=0.9]
\node (a1) at (2,6)[pin={[color=red,pin distance=3pt]135:$0$}]{};
\node (b1) at (-1,4)[pin={[color=red,pin distance=3pt]135:$1$}]{};
\node (b2) at (2,4)[pin={[color=red,pin distance=3pt]135:$10$}]{};
\node (b3) at (5,4)[pin={[color=red,pin distance=3pt]45:$11$}]{};

\node (c1) at (-2.5,2) [pin={[color=red,pin distance=3pt]225:$2$}]{};
\node (c2)  at (-1.5,2)[pin={[color=red,pin distance=3pt]225:$3$}]{};
\node (c3)  at (-0.5,2)[pin={[color=red,pin distance=3pt]215:$4$}]{};
\node (c4) at (0.5,2)[pin={[color=red,pin distance=3pt]45:$9$}]{};

\node (f1) at (-1,0)[pin={[color=red,pin distance=3pt]225:$5$}]{};
\node (f2)  at (0,0)[pin={[color=red,pin distance=3pt]225:$6$}]{};
\node (f3)  at (1,0)[pin={[color=red,pin distance=3pt]225:$7$}]{};
\node (f4) at (1.8,0)[pin={[color=red,pin distance=3pt]230:$8$}]{};
    
\draw (a1)--(b1);
\draw (a1)--(b2);
\draw	 (a1)--(b3);
\draw (b1)--(c1);
\draw (b1)--(c2);
\draw (b1)--(c3);
\draw (b1)--(c4);
\draw (c3)--(f1);
\draw (c3)--(f2);
\draw (c3)--(f3);
\draw (c3)--(f4);

\tikzstyle{every node}=[inner sep=0.5pt,font=\small,scale=0.9]
\node at (2,-2.5){
\begin{tabular}{c}
$M=\{\{1,10,11\},\{2,3,4,9\},$\\$\{5,6,7,8\}\}$\\
\end{tabular}
};
\end{scope}
\begin{scope}[xshift=2cm,yshift=-4cm,thick,scale=0.35]

\tikzstyle{every node}=[circle, draw,
                        inner sep=0.5pt, minimum width=4pt,font=\small,scale=0.9]
\node (a1) at (2,6)[pin={[color=red,pin distance=3pt]135:$0$}]{};
\node (b1) at (-1,4)[pin={[color=red,pin distance=3pt]135:$1$}]{};
\node (b2) at (2,4)[pin={[color=red,pin distance=3pt]135:$10$}]{};
\node (b3) at (5,4)[pin={[color=red,pin distance=3pt]45:$11$}]{};

\node (c1) at (-2.5,2) [pin={[color=red,pin distance=3pt]225:$2$}]{};
\node (c2)  at (-1.5,2)[pin={[color=red,pin distance=3pt]225:$3$}]{};
\node (c3)  at (-0.5,2)[pin={[color=red,pin distance=3pt]215:$4$}]{};
\node (c4) at (0.5,2)[pin={[color=red,pin distance=3pt]45:$9$}]{};
\node (d1)  at  (4.5,2)[pin={[color=red,pin distance=3pt]135:$12$}]{};
\node (d2) at (5.5,2) [pin={[color=red,pin distance=3pt]45:$18$}] {};

\node (f1) at (-2,0)[pin={[color=red,pin distance=3pt]225:$5$}]{};
\node (f2)  at (-1,0)[pin={[color=red,pin distance=3pt]225:$6$}]{};
\node (f3)  at (0,0)[pin={[color=red,pin distance=3pt]225:$7$}]{};
\node (f4) at (1,0)[pin={[color=red,pin distance=3pt]230:$8$}]{};
    
\draw (a1)--(b1);
\draw (a1)--(b2);
\draw	 (a1)--(b3);
\draw (b1)--(c1);
\draw (b1)--(c2);
\draw (b1)--(c3);
\draw (b1)--(c4);
\draw (b3)--(d1);
\draw (b3)--(d2);
\draw (c3)--(f1);
\draw (c3)--(f2);
\draw (c3)--(f3);
\draw (c3)--(f4);

\tikzstyle{every node}=[inner sep=0.5pt,font=\small,scale=0.9]
\node at (2,-2.5){
\begin{tabular}{c}
$M=\{\{1,10,11\},\{2,3,4,9\},$\\
$\{5,6,7,8\},\{12,18\}\}$\\
\end{tabular}
};
\end{scope}
\begin{scope}[xshift=7cm,yshift=-4cm,thick,scale=0.35]

\tikzstyle{every node}=[circle, draw,
                        inner sep=0.5pt, minimum width=4pt,font=\small,scale=0.9]
\node (a1) at (2,6)[pin={[color=red,pin distance=3pt]135:$0$}]{};
\node (b1) at (-1,4)[pin={[color=red,pin distance=3pt]135:$1$}]{};
\node (b2) at (2,4)[pin={[color=red,pin distance=3pt]135:$10$}]{};
\node (b3) at (5,4)[pin={[color=red,pin distance=3pt]45:$11$}]{};

\node (c1) at (-2.5,2) [pin={[color=red,pin distance=3pt]225:$2$}]{};
\node (c2)  at (-1.5,2)[pin={[color=red,pin distance=3pt]225:$3$}]{};
\node (c3)  at (-0.5,2)[pin={[color=red,pin distance=3pt]215:$4$}]{};
\node (c4) at (0.5,2)[pin={[color=red,pin distance=3pt]45:$9$}]{};
\node (d1)  at  (4.5,2)[pin={[color=red,pin distance=3pt]135:$12$}]{};
\node (d2) at (5.5,2) [pin={[color=red,pin distance=3pt]45:$18$}] {};

\node (f1) at (-2,0)[pin={[color=red,pin distance=3pt]225:$5$}]{};
\node (f2)  at (-1,0)[pin={[color=red,pin distance=3pt]225:$6$}]{};
\node (f3)  at (0,0)[pin={[color=red,pin distance=3pt]225:$7$}]{};
\node (f4) at (1,0)[pin={[color=red,pin distance=3pt]230:$8$}]{};
    
\node (e1) at (2.7,0)[pin={[color=red,pin distance=3pt]310:$13$}]{};
\node (e2)  at (3.5,0)[pin={[color=red,pin distance=3pt]315:$14$}]{};
\node (e3)  at (4.5,0)[pin={[color=red,pin distance=3pt]315:$15$}]{};
\node (e4) at (5.5,0)[pin={[color=red,pin distance=3pt]315:$16$}]{};
\node (e5) at (6.5,0)[pin={[color=red,pin distance=3pt]315:$17$}]{};

\draw (a1)--(b1);
\draw (a1)--(b2);
\draw	 (a1)--(b3);
\draw (b1)--(c1);
\draw (b1)--(c2);
\draw (b1)--(c3);
\draw (b1)--(c4);
\draw (b3)--(d1);
\draw (b3)--(d2);
\draw (c3)--(f1);
\draw (c3)--(f2);
\draw (c3)--(f3);
\draw (c3)--(f4);
\draw (d1)--(e1);
\draw (d1)--(e2);
\draw (d1)--(e3);
\draw (d1)--(e4);
\draw (d1)--(e5);

\tikzstyle{every node}=[inner sep=0.5pt,font=\small,scale=0.9]
\node at (2,-3){
\begin{tabular}{c}
$M=\{\{1,10,11\},\{2,3,4,9\},$\\
$\{5,6,7,8\},\{12,18\},$\\
$\{13,14,15,16,17\}\}$\\
\end{tabular}
};
\end{scope}
\end{tikzpicture}
\caption{The recursive construction of the tree $\gamma(M)$ for the complete noncrossing 
$(3,4,4,2,5)$-matching $M=\{\{1,10,11\},\{2,3,4,9\},\{5,6,7,8\},\{12,18\},\{13,14,15,16,17\}\}$.}
\label{fig:map_matching_to_tree}
\end{figure}
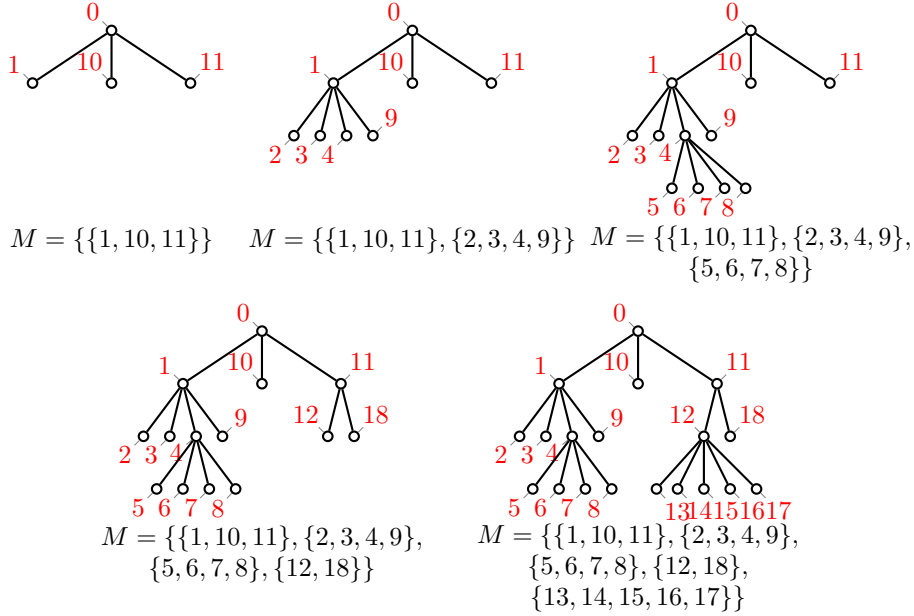

We will use the following simple lemma.
\begin{lemma}\label{lemma:noncrossing}
 If $\{M_1,M_2,\dots M_k\}$ is a noncrossing partition whose blocks are ordered according to the minimal order, then $M_k$ and 
$\bigcup_{i=1}^{k-1} M_i$ are also noncrossing.
\end{lemma}

\begin{lemma}
 $\gamma(M)$ is an $s$-tree with all its nodes labeled in preorder.
\end{lemma}
\begin{proof}To see that the map is well-defined note that in the step $j$ the node 
$v_{\min{M_j}-1}$ has already been attached to the tree $\tau_{j-1}$ otherwise there is a block 
$M_k$ with $\min{M_j}-1 \in M_k$ and $k>j$ implying $\min M_k < \min M_i$ contradicting the minimal 
ordering of $M$.

In the step $j$ we are attaching leaves with the labels in $M_j$. Since $M$ is noncrossing
by Lemma \ref{lemma:noncrossing} $M_j$ and 
$\bigcup_{i=1}^{j-1} M_i$ are noncrossing implying that the labels of $M_j$ are either nested or 
are all greater than the labels in $\bigcup_{i=1}^{j-1} M_i$. This implies that when we 
attach the vertices $\{v_{i}\,\mid\,i\in M_j\}$ as children of $v_{\min M_j-1}$ in 
$\tau_j$, the labels of the nodes in $\tau_j$ are still consistent with preorder.
\end{proof}

\begin{theorem}
 The map $\varphi:\T_s\rightarrow \CM_s$ is a bijection between $s$-trees and complete noncrossing 
$s$-matchings. The inverse of $\varphi$ is~$\gamma$.
\end{theorem}

\begin{proof}
After labeling the nodes of $T$ in preorder. Lemma \ref{lemma:phiminimal}  and the 
constructive definition of $\gamma$ imply that the labels of the internal nodes in $T$ and 
$\gamma(\varphi(T))$ are the same. Also, the constructive definition of both $\varphi$ and $\gamma$ imply 
that for each internal node $v_{i_k}$ the labeled set of children in $T$ and $\gamma(\varphi(T))$ 
are equal. These two conditions determine $T$ completely, therefore $T=\gamma(\varphi(T))$.	

Similarly, in the construction of the tree  $\gamma(M)$ the labels of a block $M_k$ are the 
children of an internal node. Since every block of $\varphi(T)$ is the set of labels of the children of 
an internal node then we have that $\varphi(\gamma(M))=M$.
\end{proof}

\section{The fundamental recurrence}
\label{sec_fundamentalRecurrence}
Trees have an inherent recursive nature, in the case of $s$-trees the recursion is 
very natural: Any $s$-tree $T$ can be constructed as $T=[T_1,T_2,...,T_{s(1)}]$ for $s(1)$ suitable 
trees $T_1,T_2,\dots,T_{s(1)}$ satisfying  
$$\signat(T)=(s(1))\oplus\signat(T_1)\oplus\cdots\oplus\signat(T_{s(1)}).$$
From this we can obtain a recursion for the number $C_s$ of $s$-Catalan structures:

\begin{align}\label{equation:fundamentalrecursion}
C_s=\sum_{(s)}C_{s_1}C_{s_2}\cdots C_{s_{s(1)}},
\end{align}
where the sum is over all sequences $(s_1,s_2,\dots,s_{s(1)})$ of compositions such that 
$s=(s(1))\oplus s_1\oplus s_2\oplus\cdots \oplus s_{s(1)}$ and where $C_{\emptyset}=1$.

Recursion (\ref{equation:fundamentalrecursion}) is at the heart of $s$-Catalan combinatorics and 
it is a generalization of the classical recursion for the number $C_n$ of Catalan objects. By 
letting $C_n:=C_{(2^n)}$ and $C_0=1$ in (\ref{equation:fundamentalrecursion}) we obtain the 
classical recursion:
\begin{align}\label{equation:catalanrecursion}
C_{n+1}=\sum_{k=0}^nC_{k}C_{n-k}.
\end{align}

\subsection{Catalan decompositions}
For a family $\A_s$ of $s$-Catalan objects recursion (\ref{equation:fundamentalrecursion}) can be 
realized by obtaining a rule of decomposition of an $s$-object $A=[A_1,A_2,\dots,A_{s(1)}]$ into 
$s(1)$ objects $A_1\in \A_{s_1}$, $A\in \A_{s_2}$, ..., $A_{s(1)}\in \A_{s_{s(1)}}$ such that 
$s=(s(1))\oplus s_1\oplus\cdots\oplus s_{s(1)}$. We call such a rule a \emph{Catalan decomposition}.

Whenever we have a pair of families of objects $\A_n$ and $\B_n$ with  
$\A_{\emptyset}=\{A_{\emptyset}\}$ and $\B_{\emptyset}=\{B_{\emptyset}\}$ both with known Catalan 
decomposition rules. Then a family of bijections $\xi_{s}:\A_s\rightarrow \B_s$ are easily 
defined recursively by:

 \begin{enumerate}
    \item[($\xi$1)]\label{definition:xi1} 
    	$\xi_{\emptyset}(A_{\emptyset})=B_{\emptyset}$.
    \item[($\xi$2)]\label{definition:xi2} 
    	Otherwise, if $A=[A_1 ,A_2,\cdots,A_k]$ with $A_i \in \A_{s_i}$ then
	\[\xi_{s}(A):=[\xi_{s_1}(A_1),\xi_{s_2}(A_2),\dots,\xi_{s_k}(A_k)]\]    
  \end{enumerate}

Such family of bijections sometimes can be described in a nicer and direct way without using 
recursion. We illustrate in Section~\ref{section:angulationofpolygon} and Section~\ref{sec_parenthesizations} the use of the recursion (\ref{equation:fundamentalrecursion}) with the examples of angulations of a convex polygon and parenthesizations of a word.

On the other hand if there is a family of bijections $\xi_{s}:\A_s\rightarrow \B_s$ where the 
Catalan decomposition rule for the objects in $\A_s$ is known, then we might be able to transport this rule to 
the objects in $\B_s$ using the maps $\xi_{s}$. This idea can 
be used to describe Catalan decompositions for other families of $s$-Catalan objects, 
like $s$-Dyck paths, $312$-avoiding Stirling $(s-{\bf 1})$-permutations, noncrossing $s$-partitions, and complete 
noncrossing $s$-matchings.

\subsection{Catalan decomposition for Dyck paths}
The Catalan decomposition of an $s$-Dyck path~$D$ is relatively simple. We remove the first north step, which corresponds to the root of the corresponding tree, and label the lattice points of the resulting path $\bar D$ (except the last one) by their horizontal distance to the Ribbon shape determined by $s$. That is, a lattice point~$p$ is labeled by the maximal number of east steps that can be added to $p$ without crossing the Ribbon shape.\footnote{We are adding an extra ``imaginary" box on top of the top-right box of the Ribbon shape to make this definition work for points on top of the grid.}  
For the example in Figure~\ref{fig:examplespathdecomposition} we get the sequence
\[\mathbf{\color{red}2},5,4,3,2,5,4,3,2,\mathbf{\color{red}1},\mathbf{\color{red}0},1,5,4,3,2,1,0.\] 
We cut the path $\bar D$ at the first appearances of $s(1)-1, s(1)-2,\dots, 1, 0$ (marked in red in the sequence and highlighted in Figure~\ref{fig:examplespathdecomposition}).  
The resulting sequence $(D_1,D_2,\cdots,D_{s(1)})$ of Dyck paths (possibly equal to the 
identity Dyck path $E$) form the Catalan decomposition of $D$, where the signature of $D_i$ is defined as the sub-sequence of $s$ determined by the rows of the north steps in $D_i$. Their composition is defined as  
\[
[D_1,D_2,\dots,D_k]:=ND_1D_2\cdots D_k.
\]
The validity of this decomposition follows from the correspondence between $s$-trees and $s$-Dyck paths, and the area labeling in Section~\ref{sec:trees-DyckPaths}.

\begin{figure}
\centering
\begin{tikzpicture}[scale=0.3]
\begin{scope}[xshift=4cm,yshift=12cm]
\edef \signature{3,4,4,2,5}

\def\row{0}
\def\col{0}

\foreach \part in \signature{
\pgfmathsetmacro \newcol {\col+\part-1}
\pgfmathparse{\col+\part-1}
\global\let\newcol\pgfmathresult
\pgfmathparse{\row+1}
\global\let\newrow\pgfmathresult
\draw[fill, color=gray!10] (\col,\row) rectangle (\newcol+1,\newrow);
\global\let\row\newrow
\global\let\col\newcol
}
\draw[very thin] (0, 0) grid (\col+1, \row);

\def\rowpa{0}
\def\colpa{0}
\edef \pa{0,3,5,0,6}
\foreach \var in \pa{
\pgfmathparse{\colpa+\var}
\global\let\newcolpa\pgfmathresult
\pgfmathparse{\rowpa+1}
\global\let\newrowpa\pgfmathresult
\draw [line width=3, color=red]
(\colpa, \rowpa)--(\colpa,\newrowpa)--(\newcolpa,\newrowpa);

\global\let\rowpa\newrowpa
\global\let\colpa\newcolpa
}
\tikzstyle{every node}=[scale=0.7]
\draw(7,-2.2) node {$s=(3,4,4,2,5)$};
\node[blue,draw,circle,inner sep=0] at (0.5,1.5) {$2$};
\node[blue] at (0.5,2.5) {$5$};
\node[blue] at (1.5,2.5) {$4$};
\node[blue] at (2.5,2.5) {$3$};
\node[blue] at (3.5,2.5) {$2$};
\node[blue] at (3.5,3.5) {$5$};
\node[blue] at (4.5,3.5) {$4$};
\node[blue] at (5.5,3.5) {$3$};
\node[blue] at (6.5,3.5) {$2$};
\node[blue,draw,circle,inner sep=0] at (7.5,3.5) {$1$};
\node[blue,draw,circle,inner sep=0] at (8.5,3.5) {$0$};
\node[blue] at (8.5,4.5) {$1$};
\node[blue] at (8.5,5.5) {$5$};
\node[blue] at (9.5,5.5) {$4$};
\node[blue] at (10.5,5.5) {$3$};
\node[blue] at (11.5,5.5) {$2$};
\node[blue] at (12.5,5.5) {$1$};
\node[blue] at (13.5,5.5) {$0$};

\end{scope}

\begin{scope}[xshift=0cm,yshift=0cm]
\edef \signature{4,4}

\def\row{0}
\def\col{0}

\foreach \part in \signature{
\pgfmathsetmacro \newcol {\col+\part-1}
\pgfmathparse{\col+\part-1}
\global\let\newcol\pgfmathresult
\pgfmathparse{\row+1}
\global\let\newrow\pgfmathresult
\draw[fill, color=gray!10] (\col,\row) rectangle (\newcol+1,\newrow);
\global\let\row\newrow
\global\let\col\newcol
}
\draw[very thin] (0, 0) grid (\col+1, \row);

\def\rowpa{0}
\def\colpa{0}
\edef \pa{3,4}
\foreach \var in \pa{
\pgfmathparse{\colpa+\var}
\global\let\newcolpa\pgfmathresult
\pgfmathparse{\rowpa+1}
\global\let\newrowpa\pgfmathresult
\draw [line width=3, color=red]
(\colpa, \rowpa)--(\colpa,\newrowpa)--(\newcolpa,\newrowpa);

\global\let\rowpa\newrowpa
\global\let\colpa\newcolpa
}
\tikzstyle{every node}=[scale=0.7]

\draw(3.5,-2) node {$s=(4,4)$};
\end{scope}
\begin{scope}[xshift=11cm]
\draw[line width=3,red] (-0.5,0)--(0.5,0);
\draw (-0.5,0)--(0.5,0);

\tikzstyle{every node}=[scale=0.7]

\draw(-0.2,-2) node {$s=\emptyset$};
\end{scope}
\begin{scope}[xshift=15cm]
\edef \signature{2,5}

\def\row{0}
\def\col{0}

\foreach \part in \signature{
\pgfmathsetmacro \newcol {\col+\part-1}
\pgfmathparse{\col+\part-1}
\global\let\newcol\pgfmathresult
\pgfmathparse{\row+1}
\global\let\newrow\pgfmathresult
\draw[fill, color=gray!10] (\col,\row) rectangle (\newcol+1,\newrow);
\global\let\row\newrow
\global\let\col\newcol
}
\draw[very thin] (0, 0) grid (\col+1, \row);

\def\rowpa{0}
\def\colpa{0}
\edef \pa{0,6}
\foreach \var in \pa{
\pgfmathparse{\colpa+\var}
\global\let\newcolpa\pgfmathresult
\pgfmathparse{\rowpa+1}
\global\let\newrowpa\pgfmathresult
\draw [line width=3, color=red]
(\colpa, \rowpa)--(\colpa,\newrowpa)--(\newcolpa,\newrowpa);

\global\let\rowpa\newrowpa
\global\let\colpa\newcolpa
}
\tikzstyle{every node}=[scale=0.7]

\draw(3,-2) node {$s=(2,5)$};
\end{scope}
\draw[line width=1, ->](11,7)--(4,3);
\draw[line width=1,->](11,7)--(11,3);
\draw[line width=1,->](11,7)--(18,3);
\end{tikzpicture}
\caption{Example of the Catalan decomposition of an $s$-Dyck path}
\label{fig:examplespathdecomposition}

\end{figure}
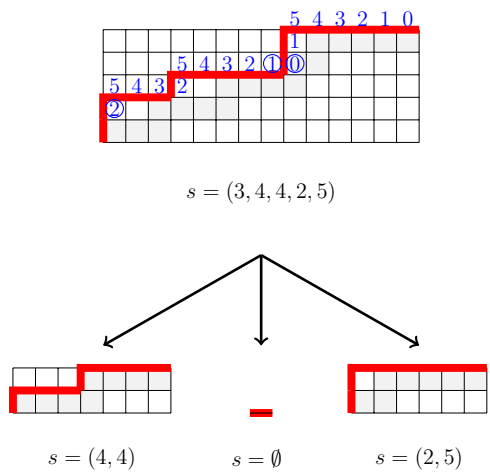

\subsection{Angulations of a polygon}\label{section:angulationofpolygon}

For $n \in \NN$ we denote by $P(n+2)$ a convex polygon with $n+2$ vertices. A \emph{diagonal} of 
$P(n+2)$ is a straight line joining any two nonadjacent vertices of the polygon. An 
\emph{angulation} of 
$P(n+2)$ is a partition of the interior of $P(n+2)$ into smaller convex polygons using a set of 
noncrossing diagonals. See Figure \ref{fig:examplerecursionsangulations} for an example of an 
angulation of $P(15)$ into $5$ parts.

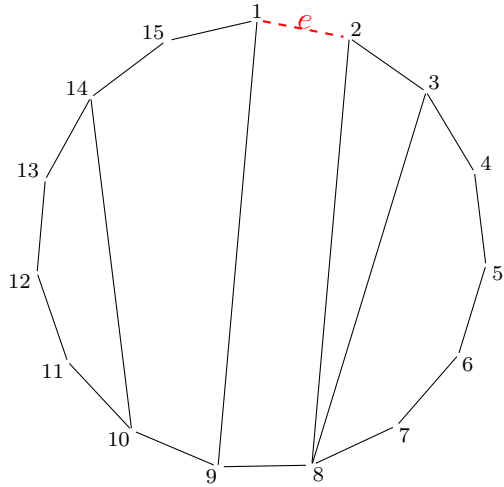
\begin{figure}
\centering
\begin{tikzpicture}[scale=1]
 
   \tikzstyle{every node}=[inner sep=0pt, minimum width=4pt]
   \def\n{15}
   
   \foreach \x in {1,2,...,\n}
   \draw {(90-360/\n*\x+360/\n+1:3cm) node (n\x)[label={90-360/\n*\x+360/\n+1:\tiny \x}]{}};

   
\draw[dashed, thick,red] (n1) -- (n2)node [midway, above, sloped,red] (TextNode) {$e$};
\draw (n2) -- (n3);
\draw (n3) -- (n4);
\draw (n4) -- (n5);
\draw (n5) -- (n6);
\draw (n6) -- (n7);
\draw (n7) -- (n8);
\draw (n8) -- (n9);
\draw (n9) -- (n10);
\draw (n10) -- (n11);
\draw (n11) -- (n12);
\draw (n12) -- (n13);
\draw (n13) -- (n14);
\draw (n14) -- (n15);
\draw (n15) -- (n1);
\draw (n1) -- (n9);
\draw (n2) -- (n8);
\draw (n3) -- (n8);
\draw (n10) -- (n14);
\end{tikzpicture}
\caption{An angulation of $P(15)$}
\label{fig:examplerecursionsangulations}
\end{figure}

Label all the vertices of $P(n+2)$ clockwise let $e$ be the edge between the vertices $1$ and 
$2$ and let $A$ be an angulation of $P(n+2)$. It is enough to remove $e$ to reveal the Catalan 
decomposition of $A$ into angulations of smaller polygons. An edge is considered as the polygon 
$P(2)$ with a unique triangulation. Then the set $\AP$ of angulations of polygons has the Catalan 
recursive structure. It is then a simple task to define the signature of an angulation $A$. Indeed, 
if removing $e$ gives angulations $A_1,A_2,\dots,A_k$ traveling counter-clockwise then
$$\signat(A):=(k)\oplus\signat(A_1)\oplus\signat(A_2)\oplus\cdots\oplus\signat(A_{k}),$$
where the signature of a single edge is $\signat(\text{\rule{1cm}{0.15cm}})=\emptyset$.
We denote by $\AP_s$ the set of $s$-angulations of $P(|s|-\ell(s)+2)$, that is the set of 
angulations $A$ with $\signat(A)=s$.

Using this recursive construction we can easily find a bijection between the $s$-trees and 
$s$-angulations of $P(|s|-\ell(s)+2)$ using the strategy described above. We leave the proof to the interested 
reader, but illustrate in Figure \ref{fig:examplebijectiontreesangulations} an example of this 
bijection when $s=(3,4,4,2,5)$ and the tree $T$ is the one in our example of Figure 
\ref{fig:examplestree}.

\begin{figure}
\centering
    \begin{tikzpicture}[scale=1]
     
       \tikzstyle{every node}=[inner sep=0pt, minimum width=4pt]
       \def\n{15}
       
       \foreach \x in {1,2,...,\n}
       \draw {(90-360/\n*\x+360/\n+1:3cm) node (n\x)[label={90-360/\n*\x+360/\n+1:\tiny \x}]{}};
    
       
    \draw (n1) -- (n2);
    \draw (n2) -- (n3);
    \draw (n3) -- (n4);
    \draw (n4) -- (n5);
    \draw (n5) -- (n6);
    \draw (n6) -- (n7);
    \draw (n7) -- (n8);
    \draw (n8) -- (n9);
    \draw (n9) -- (n10);
    \draw (n10) -- (n11);
    \draw (n11) -- (n12);
    \draw (n12) -- (n13);
    \draw (n13) -- (n14);
    \draw (n14) -- (n15);
    \draw (n15) -- (n1);
    \draw (n1) -- (n9);
    \draw (n2) -- (n8);
    \draw (n3) -- (n8);
    \draw (n10) -- (n14);
    
    \tikzstyle{every node}=[color=red]
    
    \node  at (0.6,3.4) {root};
    \draw[color=red,  very  thick] (0.6,3.2)--(0.4,1);
    \draw[color=red, dashed, very thick] (0.4,1)--(1.4,0.8);
    \draw[color=red, dashed, very thick] (0.4,1)--(0,-3.3);
    \draw[color=red, dashed, very thick] (0.4,1)--(-1,1.1);
    \draw[color=red, dashed, very thick] (1.4,0.8)--(1.8,2.8);
    \draw[color=red, dashed, very thick] (1.4,0.8)--(2,0);
    \draw[color=red, dashed, very thick] (2,0)--(2.7,2);
    \draw[color=red, dashed, very thick] (2,0)--(3.1,0.5);
    \draw[color=red, dashed, very thick] (2,0)--(3.1,-1);
    \draw[color=red, dashed, very thick] (2,0)--(2.3,-2.2);
    \draw[color=red, dashed, very thick] (2,0)--(1.2,-3);
    \draw[color=red, dashed, very thick] (-1,1.1)--(-0.8,3.1);
    \draw[color=red, dashed, very thick] (-1,1.1)--(-2,2.7);
    \draw[color=red, dashed, very thick] (-1,1.1)--(-1,-3);
    \draw[color=red, dashed, very thick] (-1,1.1)--(-2.3,-0.4);
    \draw[color=red, dashed, very thick] (-2.3,-0.4)--(-3.1,0.5);
    \draw[color=red, dashed, very thick] (-2.3,-0.4)--(-3.1,-1);
    \draw[color=red, dashed, very thick] (-2.3,-0.4)--(-2.4,-2.2);
    \draw[color=red, dashed, very thick] (-2.3,-0.4)--(-2.8,1.7);

    \end{tikzpicture}
\caption{Example of the bijection between $s$-trees and $s$-angulations}
\label{fig:examplebijectiontreesangulations}
\end{figure}
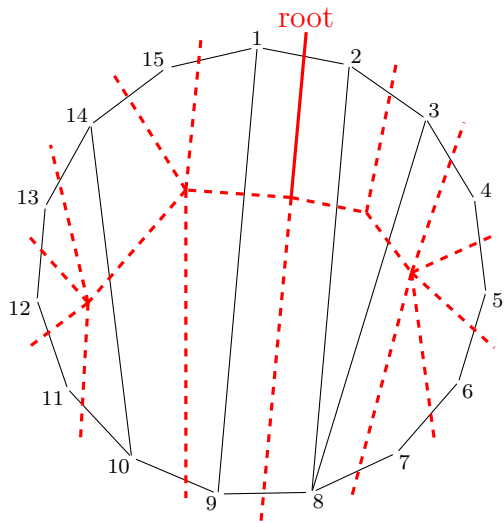

\subsection{Parenthesizations}
\label{sec_parenthesizations}
Starting with a word $w$ with $b$ letters consider the word $w^{\prime}$ of length $2a+b$ 
obtained by inserting $a$ left parenthesis and $a$ right parenthesis. For example let $a=3$ and 
$b=5$, so for the word $w=\ast\ast\ast\ast\ast$ we can insert left and right parenthesis in any 
way, 
say $w^{\prime}_1=)\ast((\ast)\ast(\ast\ast)$ or $w^{\prime}_2=(\ast((\ast\ast)\ast)\ast)$.
A \emph{(proper) parenthesization} of a word $w$ with $b$ letters is the insertion of the same 
number $a$ of left and right parenthesis with the conditions that 
 for any $i \in [2a+b]$ the prefix $w(1)w(2)\cdots w(i)$ contains at least as many left 
parenthesis as right parenthesis and that there cannot be an $i$ such that $w(i)w(i+1)=()$.

The word $w^{\prime}_2$ above is an example of a good parenthesization, another example is the word
\begin{align}\label{example:parenthesization}
w=(\ast\ast(\ast\ast\ast\ast)\ast)\ast((\ast\ast\ast\ast\ast)\ast). 
\end{align}

A \emph{segment} $u$ of a word $w=w(1)w(2)\cdots w(k)$ is a 
subword  of $w$ of the form $u=w(i+1)w(i+2)\cdots w(i+\ell)$, i.e., 
all the letters of $u$ are adjacent in $w$. A \emph{block} in a 
properly parenthesized word $w$ is a properly parenthesized segment $u$ that is inclusion maximal with respect to the property that any proper prefix in $u$ of the form $u(1)u(2)\cdots u(r)$ with $r 
<\ell$ contains strictly more left parenthesis than right parenthesis.
For the example (\ref{example:parenthesization}) there are three blocks, namely 
$(\ast\ast(\ast\ast\ast\ast)\ast)$, $\ast$ and $((\ast\ast\ast\ast\ast)\ast)$.

It is easy to see out of the definition above that every properly parenthesized word factors into 
blocks in the form $w=B_1B_2\cdots B_k$ (see example (\ref{example:parenthesization})) and that 
there is a unique parenthesization of the word $\ast$ with $b=1$ and $a=0$. So parenthesizations 
satisfy the recursion (\ref{equation:fundamentalrecursion}) and are then Catalan objects.
Define $\signat(\ast)=\emptyset$ and for the parenthesized word $w=B_1B_2\cdots B_k$ define its 
\emph{signature} $\signat(w)$ recursively as 
$$\signat(w)=(k)\oplus\signat(B_1)\oplus\cdots\oplus\signat(B_k).$$
For $w$ of example (\ref{example:parenthesization}) we have that $\signat(w)=(3,4,4,2,5)$. For 
$s \in \comp$, an \emph{$s$-parenthesization} is a properly parenthesized word $w$ such that 
$\signat(w)=s$. We denote $\PT_s$ the set of $s$-Parenthesizations of a word of length 
$|s|-\ell(s)+1$.

As in the case of $s$-angulations, $s$-trees can be recursively mapped bijectively to 
$s$-parenthesizations using their recursive constructions. Again, we leave the proof to the interested
reader and illustrate the bijection in Figure \ref{fig:exampletreesparenthesization} when 
$s=(3,4,4,2,5)$ and the tree $T$ is the one in our running 
example of Figure \ref{fig:examplestree}. Note from the Figure that at the end of the process we 
always omit the outermost parenthesis since it is redundant.

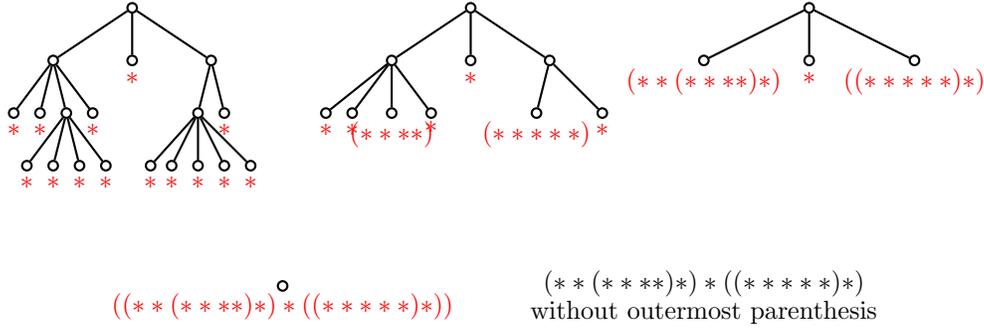
\begin{figure}
\centering
\begin{tikzpicture}[]
\begin{scope}[xshift=2cm,yshift=-3cm,thick,scale=0.35]
\tikzstyle{every node}=[ inner sep=0.5pt, minimum width=4pt,font=\small,scale=0.9]
\node[circle,draw] (a1) at (2,4)[pin={[rectangle,color=red,pin distance=0pt]270:$((\ast\ast(\ast\ast\ast\ast)\ast)\ast((\ast\ast\ast\ast\ast)\ast))$}]{};

\node (a1) at (18,3)[pin={[rectangle,color=black,pin distance=0pt]90:$(\ast\ast(\ast\ast\ast\ast)\ast)\ast((\ast\ast\ast\ast\ast)\ast)$}]{without outermost parenthesis};

\end{scope}

\begin{scope}[xshift=9cm,yshift=0cm,thick,scale=0.35]

\tikzstyle{every node}=[circle,draw,
                        inner sep=0.5pt, minimum width=4pt,font=\small,scale=0.9]
\node (a1) at (2,6){};
\node (b1) at (-2,4)[pin={[rectangle,color=red,pin distance=0pt]270:$(\ast\ast(\ast\ast\ast\ast)\ast)$}]{};
\node (b2) at (2,4)[pin={[color=red,pin distance=0pt]270:$\ast$}]{};
\node (b3) at (6,4)[pin={[rectangle,color=red,pin distance=0pt]270:$((\ast\ast\ast\ast\ast)\ast)$}]{};

\draw (a1)--(b1);
\draw (a1)--(b2);
\draw	 (a1)--(b3);

\end{scope}
\begin{scope}[xshift=4.5cm,yshift=0cm,thick,scale=0.35]

\tikzstyle{every node}=[circle,draw,
                        inner sep=0.5pt, minimum width=4pt,font=\small,scale=0.9]
\node (a1) at (2,6){};
\node (b1) at (-1,4){};
\node (b2) at (2,4)[pin={[color=red,pin distance=0pt]270:$\ast$}]{};
\node (b3) at (5,4){};

\node (c1) at (-3.5,2) [pin={[color=red,pin distance=0pt]270:$\ast$}]{};
\node (c2)  at (-2.5,2)[pin={[color=red,pin distance=0pt]270:$\ast$}]{};
\node (c3)  at (-1,2)[pin={[rectangle,color=red,pin distance=0pt]270:$(\ast\ast\ast\ast)$}]{};
\node (c4) at (0.5,2) [pin={[rectangle,color=red,pin distance=0pt]270:$\ast$}]{};
\node[circle] (d1)  at  (4.5,2)[pin={[rectangle,color=red,pin distance=0pt]270:$(\ast\ast\ast\ast\ast)$}]{};
\node[circle] (d2) at (7,2) [pin={[color=red,pin distance=0pt]270:$\ast$}] {};

\draw (a1)--(b1);
\draw (a1)--(b2);
\draw	 (a1)--(b3);
\draw (b1)--(c1);
\draw (b1)--(c2);
\draw (b1)--(c3);
\draw (b1)--(c4);
\draw (b3)--(d1);
\draw (b3)--(d2);

\end{scope}
\begin{scope}[xshift=0cm,yshift=0cm,thick,scale=0.35]

\tikzstyle{every node}=[circle, draw,
                        inner sep=0.5pt, minimum width=4pt,font=\small,scale=0.9]
\node (a1) at (2,6){};
\node (b1) at (-1,4){};
\node (b2) at (2,4)[pin={[color=red,pin distance=0pt]270:$\ast$}]{};
\node (b3) at (5,4){};

\node (c1) at (-2.5,2) [pin={[color=red,pin distance=0pt]270:$\ast$}]{};
\node (c2)  at (-1.5,2)[pin={[color=red,pin distance=0pt]270:$\ast$}]{};
\node (c3)  at (-0.5,2){};
\node (c4) at (0.5,2)[pin={[color=red,pin distance=0pt]270:$\ast$}]{};
\node (d1)  at  (4.5,2){};
\node (d2) at (5.5,2) [pin={[color=red,pin distance=0pt]270:$\ast$}] {};

\node (f1) at (-2,0)[pin={[color=red,pin distance=0pt]270:$\ast$}]{};
\node (f2)  at (-1,0)[pin={[color=red,pin distance=0pt]270:$\ast$}]{};
\node (f3)  at (0,0)[pin={[color=red,pin distance=0pt]270:$\ast$}]{};
\node (f4) at (1,0)[pin={[color=red,pin distance=0pt]270:$\ast$}]{};
    
\node (e1) at (2.7,0)[pin={[color=red,pin distance=0pt]270:$\ast$}]{};
\node (e2)  at (3.5,0)[pin={[color=red,pin distance=0pt]270:$\ast$}]{};
\node (e3)  at (4.5,0)[pin={[color=red,pin distance=0pt]270:$\ast$}]{};
\node (e4) at (5.5,0)[pin={[color=red,pin distance=0pt]270:$\ast$}]{};
\node (e5) at (6.5,0)[pin={[color=red,pin distance=0pt]270:$\ast$}]{};

\draw (a1)--(b1);
\draw (a1)--(b2);
\draw	 (a1)--(b3);
\draw (b1)--(c1);
\draw (b1)--(c2);
\draw (b1)--(c3);
\draw (b1)--(c4);
\draw (b3)--(d1);
\draw (b3)--(d2);
\draw (c3)--(f1);
\draw (c3)--(f2);
\draw (c3)--(f3);
\draw (c3)--(f4);
\draw (d1)--(e1);
\draw (d1)--(e2);
\draw (d1)--(e3);
\draw (d1)--(e4);
\draw (d1)--(e5);

\end{scope}
\end{tikzpicture}
\caption{Example of the bijection between $s$-trees and $s$-parenthesizations}
\label{fig:exampletreesparenthesization}
\end{figure}

\section{Enumeration}
\label{sec_enumeration}
Let $\lambda=(\lambda_1\geq \lambda_2 \geq \dots \geq \lambda_\ell)$ be a partition. We say that a 
partition $\lambda'=(\lambda_1'\geq \lambda_2' \geq \dots \geq \lambda_\ell')$ \emph{fits} 
$\lambda$ 
if the tableau of $\lambda'$ is contained in the tableau of $\lambda$, in other words, if $\lambda_i' \leq \lambda_i$ for all $i$. 

\begin{theorem}[Kreweras~{\cite{Kreweras1965}}]
\label{thm:kreweras}
The number $P(\lambda)$ of partitions fitting a partition $\lambda$ is counted by the determinant 
formula 
\[
P(\lambda) = \det {\lambda_j+1 \choose j-i+1}.
\]
\end{theorem}

In particular, for stair case partitions the numbers $P(\lambda)$ recover the classical Catalan 
numbers. 

\subsection{$s$-Catalan numbers}
Note that every composition $s=(s_1,\dots,s_a)$  defines a partition 
$\lambda^s=(\lambda^s_1 \geq \dots \geq \lambda^s_{a-1})$ such that 
$
\lambda_j^s = \sum_{i=1}^{a-j} (s_i-1),
$
 that is basically the shape of the upper shape that is left when in a $(|s|-1)\times a$ grid we 
remove the central ribbon determined by $s$. In the example of Figure \ref{fig:examplespath} we 
have that $\lambda^{(3,4,4,2,5)}=(9,8,5,2)$. It is not hard to see that in the same manner as the 
ribbon shape determined by $s$ describes a partition $\lambda^s$, every $s$-Dyck path describes a 
partition $\lambda^{\prime}\le \lambda^s$. Theorem \ref{thm:kreweras} gives then a 
determinantal formula for the number of $s$-Dyck paths.

\begin{corollary}[Kreweras]
The $s$-Catalan number is counted by the determinant formula 
\[
|\C_s| = \det {\sum_{i=1}^{a-j} (s_i-1)+1 \choose j-i+1}.
\]
\end{corollary}

\subsection{s-Narayana numbers}
The Catalan number $C_n=\frac{1}{n+1}{2n \choose n}$ can be refined by the Narayana numbers $N(n,k)$. These numbers are indexed by an extra parameter $1\leq k \leq n$, can be explicitly computed by the simple formula 
\[
N(n,k)= \frac{1}{n}{n \choose k}{n \choose k-1},
\]
and satisfy 
\[
C_n = N(n,1)+N(n,2)+\dots + N(n,n).
\]
The Narayana numbers count Catalan objects that satisfy certain property which depends on the family that they belong to. For instance, $N(n,k)$ is the number of (see \cite[Chapter 2]{Petersen2015} for further information):
\begin{enumerate}
\item planar binary trees with $n$ internal nodes and $k$ most left leaves;
\item Dyck paths in an $n\times n $ grid with $k$ peaks;
\item $312$-avoiding permutations of $[n]$ with $k-1$ ascents;
\item noncrossing partitions of $[n]$ with $k$ blocks;
\item noncrossing matchings of $[2n]$ containing $k$ matched pairs of the form $(i,i+1)$.
\end{enumerate}
All these combinatorial interpretations can be naturally extended to the generalized $s$-Catalan families. The \emph{$s$-Narayana number} is defined as the number of elements in the following families. We remark that they correspond to each other under the bijections presented in Section~\ref{sec_sCatalanzooAndBijections} and Appendix~\ref{sec_bijectionsWithDyckPaths}. 

 \begin{enumerate}
\item $s$-trees with $k$ most left leaves;
\item $s$-Dyck paths with $k$ peaks; \label{item_Dyckpaths}
\item $312$-avoiding Stirling $(s-\bf{1})$-permutations with $k-1$ ascents; 
\item noncrossing $(s-\bf{1})$-partitions with $k$ blocks;
\item Complete noncrossing $s$-matchings with $k$ parts satisfying $\operatorname{min}M_i+1\in M_i$.
\end{enumerate}

As far as we know there are no simple formulas to count these combinatorial objects in this general setting. However, refinements of the $s$-Narayana numbers have already been studied in the literature, for instance in the case of Fuss-Narayana numbers and rational Narayana numbers~\cite{ArmstrongRhoadesWilliams2013,lenczewski_limite_2014,lenczewski_FussNarayana_2013}.

\section{Signature generalizations of permutations and parking functions}

\subsection{$s$-generalized permutations}

As mentioned in the introduction, permutations of~$[n]$ are in bijection with increasingly labeled planar binary trees with $n$ internal nodes. A natural generalization of permutations in terms of signatures is the collection of permutations corresponding to $(s+\bf{1})$-increasing trees. As explained in Section \ref{sec_sincreasing_stirling}, such permutations are in correspondence with Stirling $s$-permutations, which were already considered and studied by Gessel and Stanley back in the seventies~\cite{GesselStanley1978} for the special case~$s=(2,\dots,2)$. As described above, these are multipermutations of the multiset $\{1^{s(1)},2^{s(s(2))},\dots, a^{s(a)}\}$ that avoid the pattern $212$; see~\cite{KubaPanholzer2011,JansonKubaPanholzer2011,Park1994-1,Park1994-2,Park1994-3, 
GrahamKnuthPatashnik1994,Dleon2015,RemmelWilson2015} for further studies on such permutations. 

The cardinality of the set~$\SP_s$ of Stirling $s$-permutations generalizes the factorial numbers. To be more precise, 
from a permutation in $\SP_{(s(1),s(2),\dots,s(k-1))}$ we
can obtain a permutation in $\SP_{s\oplus (s(k))}$ by inserting the $s(k)$ consecutive occurrence 
of the  label $kk\cdots k$ in any of the $s(1)+s(2)+\cdots + s(k-1)+1$ possible
positions. Inductively and starting in $\SP_{(s(1))}=\{11\cdots 1\}$ this implies that 
$$|\SP_s|=1 \cdot (s(1)+1) \cdot (s(1)+s(2)+1)\cdot \cdots \cdot 
(s(1)+s(2)+\cdots+s(a-1)+1)=:|s|!^{s},$$
and we read $|s|!^{s}$ as \emph{$|s|$ $s$-factorial}. Note that if $s=(1^n)$ then $|s|!^{s}=n!$ and 
if $s=(2^n)$ then  $|s|!^{s}=1\cdot 3 \cdots (2n-1)=(2n-1)!!$, the factorial and double-factorial 
numbers.

 In forthcoming work~\cite{ceballos_sweakorder}, the first author and Viviane Pons study a generalization of the weak order on permutations for the set of Stirling $s$-permutations. They call this generalized order the \emph{$s$-weak order} and show that it has the structure of a lattice. They also investigate geometric realizations of the $s$-weak order in terms of polytopal subdivisions of the classical permutahedron.

\subsection{$s$-generalized parking functions}
A \emph{parking function} of $n$ is a sequence of nonnegative integers $(p_1,p_2,\dots,p_n)$ with 
the property that after being rearranged in weakly increasing order $p_{j_1}\le p_{j_2}\le \cdots 
\le p_{j_n}$ they satisfy $p_{j_i}<i$ for all $i\in[n]$. We denote by $\PF_n$ the set of 
parking functions of $n$. Parking functions were studied by 
Konheim and Weiss \cite{KonheimWeiss1966} were they used the following equivalent description to 
define them: 

Let $C_1,\dots,C_n$ be a collection of cars to be parked in order, car $C_i$ 
before car $C_{i+1}$, in a one-way street in a consecutive sequence of parking spots marked 
$P_0,\dots,P_{n-1}$. The driver of car $C_i$ has a preferred spot $p_i$ where he would like to 
park. In step $i$ the driver of car $C_i$ drives up to the spot $p_i$, parking there if the spot is 
available, or otherwise continuing until the next available spot. If the spot $p_i$ 
and all the following parking spots are occupied then $C_i$ is unable to park. A 
\emph{parking function} is a sequence $(p_1,p_2,\dots,p_n)$ of parking preferences where all the 
cars are able to park.

It is not difficult to verify that the two definitions of parking function stated above are 
equivalent. There is another way to construct parking functions. We can obtain parking functions by 
decorating Dyck paths with the elements of $[n]$. 
 
A \emph{decorated Dyck path} is a Dyck path in which the north steps have been decorated 
from left to right with a permutation of $[n]$ in such a way that consecutive north steps have 
increasing labels (see Figure \ref{figure:example_decorated_Dyck_path}). Let $\DDP_n$ denote the 
set of decorated Dyck paths of $[n]$. In Figure \ref{figure:example_type_A_combinatorics} the 
reader can see the example of $\DDP_3$.

It turns out that the sets $\PF_n$ and $\DDP_n$ are in bijection and that decorated Dyck paths 
are a convenient way to encode parking functions. The bijection we will describe is common in the 
literature, see for example \cite{GarsiaHaiman1996}. The function $P:\DDP_n\rightarrow \PF_n$ that 
provides a parking function $P(D)$ starting with a decorated Dyck path $D$ is defined as follows: 
the value of $p_i$ is the number of east steps that are in $D$ to the left of the north step 
decorated with the letter $i$. The function $P$ is bijective and its inverse is also easy to 
describe. In the example of Figure \ref{figure:example_decorated_Dyck_path} we have for example 
that $p_1=0$ because there are no east steps before the north step labeled $1$ and $p_4=5$ since 
there are $5$ east step before the north step labeled $4$.

 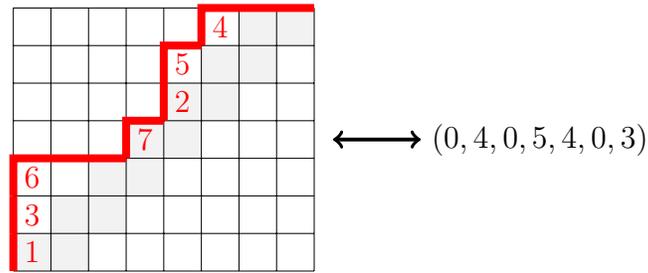
\begin{figure}
\centering
\begin{tikzpicture}[scale=0.5]
\edef \signature{2,2,2,2,2,2,2}

\def\row{0}
\def\col{0}

\foreach \part in \signature{
\pgfmathsetmacro \newcol {\col+\part-1}
\pgfmathparse{\col+\part-1}
\global\let\newcol\pgfmathresult
\pgfmathparse{\row+1}
\global\let\newrow\pgfmathresult
\draw[fill, color=gray!10] (\col,\row) rectangle (\newcol+1,\newrow);
\global\let\row\newrow
\global\let\col\newcol
}
\draw[very thin] (0, 0) grid (\col+1, \row);

\def\rowpa{0}
\def\colpa{0}
\edef \pa{0,0,3,1,0,1,3}
\foreach \var in \pa{
\pgfmathparse{\colpa+\var}
\global\let\newcolpa\pgfmathresult
\pgfmathparse{\rowpa+1}
\global\let\newrowpa\pgfmathresult
\draw [line width=3, color=red]
(\colpa, \rowpa)--(\colpa,\newrowpa)--(\newcolpa,\newrowpa);

\global\let\rowpa\newrowpa
\global\let\colpa\newcolpa
}

\node[color=red] at (0.5,0.5) {$1$};
\node[color=red] at (0.5,1.5) {$3$};
\node[color=red] at (0.5,2.5) {$6$};
\node[color=red] at (3.5,3.5) {$7$};
\node[color=red] at (4.5,4.5) {$2$};
\node[color=red] at (4.5,5.5) {$5$};
\node[color=red] at (5.5,6.5) {$4$};

\node[](pf) at (14,3.5) {$(0,4,0,5,4,0,3)$};
\draw[<->,line width=.02in] (8.5,3.5) -- (pf);

\end{tikzpicture}
\caption{Example of a decorated Dyck path and its corresponding parking function}
\label{figure:example_decorated_Dyck_path}
\end{figure}

Perhaps the easiest way to generalize the concept of parking function is using the concept of a 
decorated Dyck path. For a composition $s$ let $\DDP_s$ be the set of $s$-Dyck paths whose north 
steps have been decorated with a permutation of $[a]$ in such a way that consecutive north steps 
have increasing labels. The discussion above illustrates that the set $\DDP_s$ is also in bijection 
with a more general set of parking functions. For a weak composition $\mu$  a \emph{$\mu$-parking 
function} is a sequence of nonnegative integers $(p_1,p_2,\cdots,p_a)$ such that after being 
rearranged in weakly increasing order $p_{j_1}\le p_{j_2}\le \cdots \le p_{j_a}$ they satisfy 
$p_{j_1}=0$ and
\begin{align*}
 p_{j_i}\le\sum_{i=1}^{i-1}\mu(i)\text{ for all }i>1.
\end{align*}
We denote by $\PF_{\mu}$ the set of $\mu$-parking functions. By the same reasoning as above we have 
that there is a bijection $\DDP_s\rightarrow \PF_{s-\bf{1}}$.  $\mu$-parking 
functions have already appear in the literature with the name of generalized parking functions, 
vector parking functions or $\mu$-parking functions. They were originally studied by Pitman and Stanley~\cite{StanleyPitman2002} and Yan \cite{Yan2000,Yan2001} in a slightly more general form. This more general form can be obtained within this framework if we allow $0$ values in our signatures and what we call a $\mu$-parking function is what appears as a $(1)\oplus \mu$-vector 
parking function in the literature. The case 
$s-\bf{1}=(1^n)$ returns the classical concept of parking functions counted by $(n+1)^{n-1}$ and in 
general if $s$ is a rational signature for $(a,b)$ relatively prime then Armstrong, Loehr and 
Warrington \cite{ArmstrongLoehrWarrington2016} show that $|\PF_{a,b}|=b^{a-1}$, this includes the 
case $a=n$ and $b=kn+1$ or $s-\bf{1}=(k^n)$ of Fuss-Catalan parking functions. The rational 
signatures are basically the cases were a nice product formula is known and to find formulas for 
general signatures seems to be a difficult task. In \cite{GaydarovHopkins2016} Gaydarov and Hopkins 
generalize the work in \cite{PerkinsonYangYu2016} to provide a bijection between vector parking 
functions and a set of trees that they call rooted plane trees. 

We finish this section with the remark that even thought in the literature it is most common to 
represent parking functions as decorated Dyck paths, we argue that is more natural and probably 
more convenient from the recursive point of view to use the perspective of $s$-trees to represent a 
parking function. A \emph{decorated $s$-tree} is an $s$-tree such that all the internal nodes are 
decorated with distinct elements of $[a]$ in such a way that an internal node that is the left-most 
child of its parent has a larger label. Let $\DT_s$ be the set of decorated 
$s$-trees, then is clear from Theorem \ref{thm:bijection_trees_paths} that we have a bijection 
$\DT_s\cong \DDP_s\cong \PF_{s-\bf{1}}$. In Figure \ref{figure:example_decorated_s_tree} we 
illustrate an example of a decorated $s$-tree corresponding to the $(1^7)$-parking function 
$(0,4,0,5,4,0,3)$. Note that the associated parking function can be obtained by labeling the leaves 
left-to-right from $0$ to $|s-\mathbf{1}|+1$, then $p_i$ is equal to the label of the leaf that 
is the leftmost descendant of the internal node decorated $i$. 

 \begin{figure}
\centering
\begin{tikzpicture}[thick,scale=0.6]
\tikzstyle{every node}=[circle, draw,
                        inner sep=0.5pt, minimum width=4pt,font=\small]
\node (i1) at (0,0){\color{red}$1$};
\node (i2) at (2,-2){\color{red}$2$};
\node (i3) at (-1,-1){\color{red}$3$};
\node (i4) at (2,-4){\color{red}$4$};
\node (i5) at (1,-3){\color{red}$5$};
\node (i6) at (-2,-2){\color{red}$6$};
\node (i7) at (1,-1){\color{red}$7$};
\node (l0) at (-3,-3)[pin={[pin distance=2]270:$0$}]{};
\node (l1) at (-1,-3)[pin={[pin distance=2]270:$1$}]{};
\node (l2) at (-.2,-2)[pin={[pin distance=2]270:$2$}]{};
\node (l3) at (0.2,-2)[pin={[pin distance=2]270:$3$}]{};
\node (l4) at (0,-4)[pin={[pin distance=2]270:$4$}]{};
\node (l5) at (1,-5)[pin={[pin distance=2]270:$5$}]{};
\node (l6) at (3,-5)[pin={[pin distance=2]270:$6$}]{};
\node (l7) at (3,-3)[pin={[pin distance=2]270:$7$}]{};
\draw (i1) -- (i3);
\draw (i3) -- (i6);
\draw (i6) -- (l0);
\draw (i6) -- (l1);
\draw (i3) -- (l2);
\draw (i1) -- (i7);
\draw (i7) -- (i2);
\draw (i7) -- (l3);
\draw (i2) -- (i5);
\draw (i5) -- (i4);
\draw (i5) -- (l4);
\draw (i4) -- (l5);
\draw (i4) -- (l6);
\draw (i2) -- (l7);

\end{tikzpicture}
\caption{Example of the decorated $s$-tree corresponding to the parking function of Figure 
\ref{figure:example_decorated_Dyck_path}}
\label{figure:example_decorated_s_tree}
\end{figure}
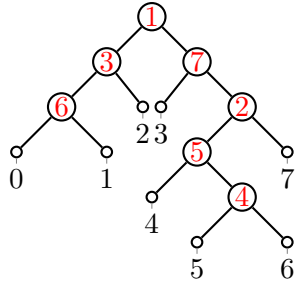

\section{Relation to the rational constructions of Armstrong, Rhoades and Williams}
The constructions in this article generalize the rational Dyck paths of Armstrong, 
Rhodes and Williams~\cite{ArmstrongRhoadesWilliams2013} to more general signatures $s$. A natural question 
that arises is whether our constructions for noncrossing $s$-partitions, complete noncrossing 
$s$-matchings and $s$-angulations of a $(|s|-\ell(s)+2)$-gon specialize to their rational 
constructions for the same type of objects in 
\cite{ArmstrongRhoadesWilliams2013} when $s$ is a rational signature. The answer to this question 
happens to be on the negative side. Even though in some examples the two constructions coincide 
they are in general different. In this section we discuss the difference between the 
constructions and in particular we illustrate the situation in the case of rational 
noncrossing partitions. 

Let $(a,b)$ be a pair of relatively prime positive numbers and $s$ the corresponding 
rational signature defined as in equation \ref{equation:definitionofrationalsignature}. In 
\cite{ArmstrongRhoadesWilliams2013}, the authors define a family of inhomogeneous $(a,b)$-noncrossing 
partitions (that we will denote here by $\widetilde \NC_{(a,b)}$), a family of homogeneous 
$(a,b)$-noncrossing partitions (denoted here by $\widetilde \CM_{(a,b)}$) and  
a family of rational dissections of a polygon (denoted here by $\widetilde \AP_{(a,b)}$).
The constructions of $\widetilde \NC_{(a,b)}$, $\widetilde \CM_{(a,b)}$ and $\widetilde 
\AP_{(a,b)}$ rely on a ``laser'' construction. A \emph{laser} through the point $(x_0,y_0)$ is the 
half-ray in the grid $b\times a$ with starting point $(x_0,y_0)$ and slope $\frac{a}{b}$, i.e., a 
half-ray given by the equation $y=y_0+\frac{a}{b}(x-x_0)$ with $x\ge x_0$. See the Figure 
\ref{fig:difference_in_rational_noncrossing_partitions} for examples of lasers that go through the 
points $(0,0)$, $(2,2)$, $(6,3)$ and $(10,4)$ with slope $\frac{5}{13}$.

We now describe the construction of $\widetilde \NC_{(a,b)}$. For a given $(a,b)$-Dyck path $D$ 
we label the right ends of the east steps of $D$ from left to right with the labels 
$1,2,\dots,(b-1)$. For every consecutive sequence of north steps in $D$ fire a laser through the 
lower point where the sequence starts. The lasers give a topological decomposition of the reqion 
between $D$ and the diagonal $y=\frac{a}{b}x$ and then we can define a partition $\pi(D)$ by 
letting $i$ and $j$ to be in the same block if they belong to the same connected component. In 
\cite[Proposition 6.1]{ArmstrongRhoadesWilliams2013} it is shown that $\pi(D)$ is a noncrossing 
partition of $[b-1]$ and that $\pi$ is an injective map, which allow them to define rational 
noncrossing partitions. In Figure \ref{fig:difference_in_rational_noncrossing_partitions}  we 
illustrate an example of this construction starting with the $(5,13)$-Dyck path (or 
$(3,4,3,4,3$-Dyck path in our language) $D=N^2E^2NE^4NE^4NE^3$. The noncrossing partition 
obtained from the construction in \cite{ArmstrongRhoadesWilliams2013} is 
$\pi(D)=\{\{1,2,5,6,9,10\},\{3,4\},\{7,8\},\{11,12\}\}$  while the noncrossing 
partition obtained from $D$ using our constructions by first identifying the $(3,4,3,4,3)$-tree 
$\zeta(D)$ associated to $D$ and then finding the noncrossing $(3,4,3,4,3)$-partition 
$\phi(\zeta(D))=\{\{1,2,5,6,10\},\{3,4\},\{7,8,9\},\{11,12\}\}$, showing  
that our constructions arrive to different objects. Note that not only the 
noncrossing partitions associated to $D$ are different but also the sets $\widetilde 
\NC_{(5,13)}$ and $\NC_{(2,3,2,3,2)}$ of rational noncrossing partitions are different. 
The partition $\pi(D)=\{\{1,2,5,6,9,20\},\{3,4\},\{7,8\},\{11,12\}\}$ has parts of sizes 
$\mu(\pi(D))=(6,2,2,2)$, so in particular, it is not refined by $s-\mathbf{1}=(2,3,2,3,2)$ and 
hence $\pi(D)\notin \NC_{(2,3,2,3,2)}$. A similar situation occurs with the families $\widetilde 
\CM_{(a,b)}$ and $\widetilde \AP_{(a,b)}$ since they are also defined by the laser construction and 
our constructions have definitions that reflect the underlying tree structure of the objects 
instead.

In \cite{ArmstrongRhoadesWilliams2013} Armstrong, 
Rhodes and Williams  pose as an open problem to give characterizations of 
the families $\widetilde \NC_{(a,b)}$, $\widetilde \CM_{(a,b)}$ and $\widetilde \AP_{(a,b)}$ that 
are somewhat more natural and do not rely in a laser construction. One advantage of our 
constructions is precisely that they are independent of the laser construction and that they happen 
to reflect the natural tree structure of $s$-Catalan families. One advantage in the laser 
constructions however is that they happen to have 
certain rotational symmetry. Indeed, it was proven in 
\cite[Proposition 5.2]{ArmstrongRhoadesWilliams2013} that $\widetilde \CM_{(a,b)}$ is closed under 
rotation and the same property is conjecturally true for $\widetilde \NC_{(a,b)}$. It is not 
difficult to check that the families $\CM_s$ and $\NC_{s-\mathbf{1}}$ are in general not closed 
under rotation unless $s=(k^a)$ for some value of $k$, since the rotation operation also have the 
effect of permuting the entries of the signature $s$. Hence if we want to have rotational symmetry 
we need to go to a larger family of objects that includes all the possible permutations of the 
signature. Finally, a clear advantage of our constructions is that we produce, as a byproduct, a 
family of $(a,b)$-trees that is central in the new rational picture and that it was not defined 
before in the literature. We claim that this central object could play an important role when 
studying families of rational Catalan objects for its inherent recursive nature. For example the 
$(a,b)$-trees also allow us to obtain the notion of Stirling $(a,b)$-permutations that were not 
considered before inside this picture.

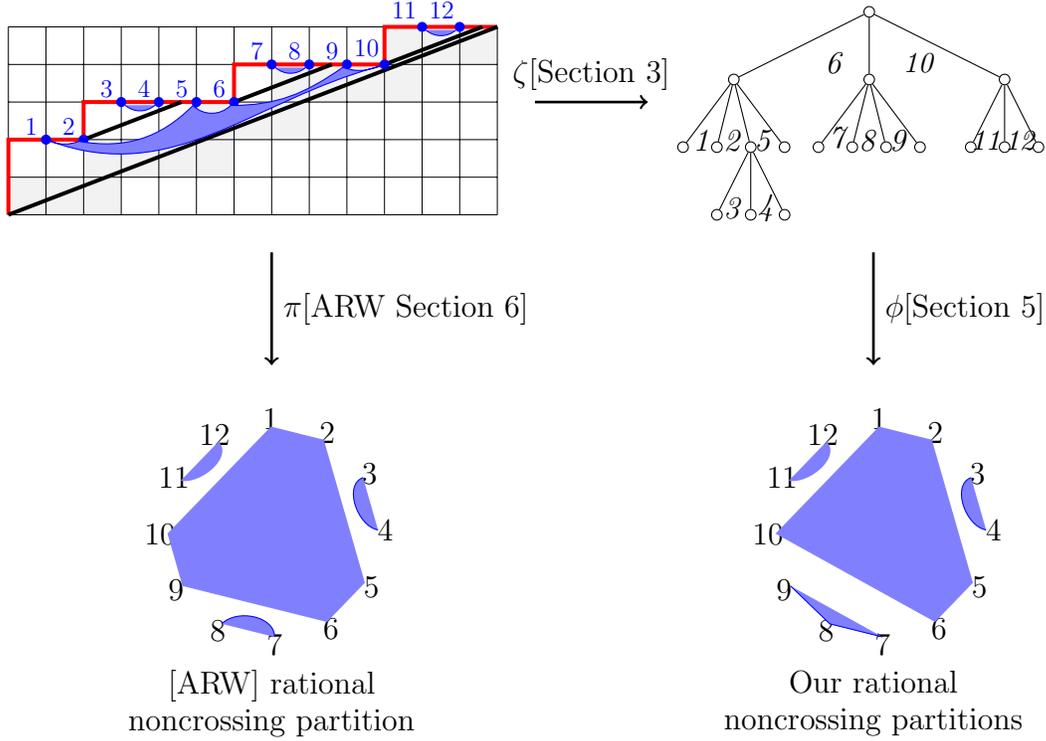
\begin{figure}[htb]
\centering
 \begin{tikzpicture}[]

\draw[line width=1,->] (3.5,-0.5) -- (3.5,-2)node [midway, right]  {$\pi$[ARW Section 6]};
\draw[line width=1,->] (11.5,-0.5) -- (11.5,-2)node [midway, right]  {$\phi$[Section 5]};
\draw[line width=1,->] (7,1.5) -- (8.5,1.5)node [midway, above]  {$\zeta$[Section 3]};
\node[align=center] at (3.5,-6.5) {[ARW] rational\\ noncrossing partition};
\node[align=center] at (11.5,-6.5) {Our rational\\ noncrossing partitions};

\begin{scope}[xshift=0,yshift=0,scale=0.5]
\edef \signature{3,4,3,4,3}

\def\row{0}
\def\col{0}

\foreach \part in \signature{
\pgfmathsetmacro \newcol {\col+\part-1}
\pgfmathparse{\col+\part-1}
\global\let\newcol\pgfmathresult
\pgfmathparse{\row+1}
\global\let\newrow\pgfmathresult
\draw[fill, color=gray!10] (\col,\row) rectangle (\newcol+1,\newrow);
\global\let\row\newrow
\global\let\col\newcol
}
\draw[very thin] (0, 0) grid (\col+1, \row);

\def\rowpa{0}
\def\colpa{0}
\edef \pa{0,2,4,4,3}
\foreach \var in \pa{
\pgfmathparse{\colpa+\var}
\global\let\newcolpa\pgfmathresult
\pgfmathparse{\rowpa+1}
\global\let\newrowpa\pgfmathresult
\draw [line width=1.5, color=red]
(\colpa, \rowpa)--(\colpa,\newrowpa)--(\newcolpa,\newrowpa);

\global\let\rowpa\newrowpa
\global\let\colpa\newcolpa
}

\draw[line width=1.5] (0,0) -- (13,5);
\draw[line width=1.5] (2,2) -- (2+13/5,3);
\draw[line width=1.5] (6,3) -- (6+13/5,4);
\draw[line width=1.5] (10,4) -- (10+13/5,5);

\tikzstyle{every node}=[blue,fill,draw,circle,inner sep=1.5,scale=0.8]
\node (n1) at (1,2)[pin={[pin edge=-,pin distance=0]145:1}] {};
\node (n2) at (2,2)[pin={[pin edge=-,pin distance=0]145:2}] {};
\node (n3) at (3,3)[pin={[pin edge=-,pin distance=0]145:3}] {};
\node (n4) at (4,3)[pin={[pin edge=-,pin distance=0]145:4}] {};
\node (n5) at (5,3)[pin={[pin edge=-,pin distance=0]145:5}] {};
\node (n6) at (6,3)[pin={[pin edge=-,pin distance=0]145:6}] {};
\node (n7) at (7,4)[pin={[pin edge=-,pin distance=0]145:7}] {};
\node (n8) at (8,4)[pin={[pin edge=-,pin distance=0]145:8}] {};
\node (n9) at (9,4)[pin={[pin edge=-,pin distance=0]145:9}] {};
\node (n10) at (10,4)[pin={[pin edge=-,pin distance=0]145:10}] {};
\node (n11) at (11,5)[pin={[pin edge=-,pin distance=0]145:11}] {};
\node (n12) at (12,5)[pin={[pin edge=-,pin distance=0]145:12}] {};

\draw [blue,fill=blue!50] (n1) to[out=-20,in=-160] (n2)to [out=-10,in=-135] (n5)  to[out=-45,in=-135] (n6) to[out=0,in=-145] (n9) to[out=-20,in=-160] (n10) to[out=-165,in=-20] (n1);
\draw [blue,fill=blue!50] (n3) to[out=-40,in=-140] (n4);
\draw [blue,fill=blue!50] (n7) to[out=-40,in=-140] (n8);
\draw [blue,fill=blue!50] (n11) to[out=-40,in=-140] (n12);

\end{scope}
\begin{scope}[xshift=100,yshift=-120,scale=0.5]
 \tikzstyle{every node}=[inner sep=0pt, minimum width=4pt]
   \def\n{12}
   
   \foreach \x in {1,2,...,12}
   \draw {(90-360/\n*\x+360/\n+1:3cm) node {\x}};
   
   \def\inicio{1} 
   \def\col{blue}
      \path[fill=\col!50] (90-360/\n*\inicio+360/\n+1:2.8cm)
      \foreach \x in {1,2,5,6,9,10}{ -- (90-360/\n*\x+360/\n+1:2.8cm)};

 \def\inicio{3} 
   \def\col{blue}
      \draw[\col,fill=\col!50] (90-360/\n*\inicio+360/\n+1:2.8cm)
      \foreach \x in {3,4}{ to[out=-160,in=-190] (90-360/\n*\x+360/\n+1:2.8cm)};
      
    \def\inicio{7} 
    \def\col{blue}
      \draw[\col,fill=\col!50] (90-360/\n*\inicio+360/\n+1:2.8cm)
      \foreach \x in {7,8}{  to[out=90,in=45](90-360/\n*\x+360/\n+1:2.8cm)};
      
      \def\inicio{11} 
    \def\col{blue}
      \path[blue,fill=\col!50] (90-360/\n*\inicio+360/\n+1:2.8cm)
      \foreach \x in {11,12}{  to[out=-10,in=-45] (90-360/\n*\x+360/\n+1:2.8cm)};
\end{scope}

\begin{scope}[xshift=330,yshift=-120,scale=0.5]
 \tikzstyle{every node}=[inner sep=0pt, minimum width=4pt]
   \def\n{12}
   
   \foreach \x in {1,2,...,12}
   \draw {(90-360/\n*\x+360/\n+1:3cm) node {\x}};
   
   \def\inicio{1} 
   \def\col{blue}
      \path[fill=\col!50] (90-360/\n*\inicio+360/\n+1:2.8cm)
      \foreach \x in {1,2,5,6,10}{ -- (90-360/\n*\x+360/\n+1:2.8cm)};

 \def\inicio{3} 
   \def\col{blue}
      \draw[\col,fill=\col!50] (90-360/\n*\inicio+360/\n+1:2.8cm)
      \foreach \x in {3,4}{ to[out=-160,in=-190] (90-360/\n*\x+360/\n+1:2.8cm)};
      
    \def\inicio{7} 
    \def\col{blue}
      \draw[\col,fill=\col!50] (90-360/\n*\inicio+360/\n+1:2.8cm)
      \foreach \x in {7,8,9}{ --(90-360/\n*\x+360/\n+1:2.8cm)};
      
      \def\inicio{11} 
    \def\col{blue}
      \path[blue,fill=\col!50] (90-360/\n*\inicio+360/\n+1:2.8cm)
      \foreach \x in {11,12}{  to[out=-10,in=-45] (90-360/\n*\x+360/\n+1:2.8cm)};
\end{scope}

\begin{scope}[xshift=300,yshift=0,scale=0.45]
    \tikzstyle{every node}=[inner sep=0pt, minimum width=4pt]

    \draw (-2.9,2.2) node []{\emph{1}};
    \draw (-2,2.2) node []{\emph{2}};
    \draw (-2,0.2) node []{\emph{3}};
    \draw (-1,0.2) node []{\emph{4}};
    \draw (-1.1,2.2) node []{\emph{5}};
    \draw (1,4.5) node []{\emph{6}};
    \draw (1.1,2.2) node []{\emph{7}};
    \draw (2,2.2) node []{\emph{8}}; 
    \draw (2.9,2.2) node []{\emph{9}};
    \draw (3.5,4.5) node []{\emph{\small 10}};
    \draw (5.5,2.2) node []{\emph{\small 11}};
    \draw (6.5,2.2) node []{\emph{\small 12}};

 \tikzstyle{every node}=[circle, draw,
                        inner sep=0.5pt, minimum width=4pt,font=\small]
\node (a1) at (2,6){};
\node (b1) at (-2,4){};
\node (b2) at (2,4){};
\node (b3) at (6,4){};

\node (c1) at (-3.5,2) {};
\node (c2)  at (-2.5,2){};
\node (c3)  at (-1.5,2){};
\node (c4) at (-0.5,2){};
\node (d1)  at  (5,2){};
\node (d2) at (6,2) {};
\node (d3) at (7,2) {};

\node (f1) at (-2.5,0){};
\node (f2)  at (-1.5,0){};
\node (f3)  at (-0.5,0){};

\node (e1) at (0.5,2){};
\node (e2)  at (1.5,2){};
\node (e3)  at (2.5,2){};
\node (e4) at (3.5,2){};

\draw (a1)--(b1);
\draw (a1)--(b2);
\draw(a1)--(b3);
\draw (b1)--(c1);
\draw (b1)--(c2);
\draw (b1)--(c3);
\draw (b1)--(c4);
\draw (b3)--(d1);
\draw (b3)--(d2);
\draw (b3)--(d3);
\draw (c3)--(f1);
\draw (c3)--(f2);
\draw (c3)--(f3);
\draw (b2)--(e1);
\draw (b2)--(e2);
\draw (b2)--(e3);
\draw (b2)--(e4);

\end{scope}

\end{tikzpicture}
\caption{Difference between $\widetilde \NC_{(a,b)}$ and $\NC_s$ for a rational $s$.}
\label{fig:difference_in_rational_noncrossing_partitions}
\end{figure}

\section{Algebraic structures on signature objects}
In a forthcoming article we study various algebraic structures that are related to the set 
$\T=\bigcup_{s\in \comp}\T_s$ of planar rooted trees and such that after applying the 
corresponding algebraic operations some properties of signatures are preserved.

The structure that is most notably related to $\T$ is the one that models the category of 
\emph{nonsymmetric operads}, see  \cite{LodayVallette2012} for the context and 
notation. Let $\NN\text{-Mod}$ be the category of $\NN$-graded (or commonly known as arity graded) 
vector spaces and graded preserving linear maps over some field $\kk$. Over the monoidal 
category of endofunctors $F:\NN\text{-Mod}\rightarrow \NN\text{-Mod}$ with composition of 
endofunctors and whose identity $\II$ is the identity endofunctor we have a \emph{monad} 
$\T:\NN\text{-Mod}\rightarrow \NN\text{-Mod}$ (we are abusing notation using $\T$ here also but 
hope the reader will find every meaning clear from the context) such that for an $\NN$-module 
$M=\bigoplus_{n\ge 1}M_n$
we have that 

\begin{align*}\label{equation:compositiontrees}
\T(M)=\bigoplus_{n\ge 1}\bigoplus_{T\in \T_n}\bigotimes_{v\in 
\intvert(T)}M_{\indeg(v)}, 
\end{align*}
where $\T_n$ is the set of planar rooted trees with $n$ leaves, for $T\in \T_n$ we have that $\intvert(T)$ is 
the set of internal vertices of $T$ and for $v\in \intvert(T)$ we have that $\deg(v)$ is the number of children 
of $v$. The rule of composition $\gamma:\T\circ\T \rightarrow \T$ is given by \emph{composition of 
trees}, i.e., if $T\in \T_k$ and $T_i\in \T_{m_i}$ for $i\in [k]$ 
then $$\gamma(T;T_1,T_2,\dots,T_k)$$ is the tree in $\T_{\sum_i m_i}$ constructed by attaching the 
root of $T_i$ to the $i$-th leaf of $T$ from left to right. The identity $\eta:\II\rightarrow \T$ 
is the natural transformation that maps $\II(M)=M$ to  $\eta(\II(M))=M$ that can be seen as having 
the same form of equation \label{equation:compositiontrees} but for each value of $n$ the only tree 
in $\T_n$ to be considered is the \emph{corolla} with arity $n$, that is the, unique planar rooted 
tree with signature $(n)$. Algebras over the monad $\T$ are known as \emph{nonsymmetric operads} 
and for any $M\in\NN\text{-Mod}$ we have that $\T(M)$ is the \emph{free nonsymmetric operad} 
generated by $M$. In particular when $M=\bigoplus_{n\ge 1} \kk$ then we obtain a free nonsymmetric 
operad structure on planar rooted trees, see \cite{LodayVallette2012}.

Note that if $\signat(T)=s$ and $\signat(T_i)=s_i$ for $i\in [k]$ then it is not possible to 
determine exactly $\signat(\gamma(T;T_1,T_2,\dots,T_k)$ since it will depend on the tree structure 
of $T$. However, since as multisets the components of the signatures are preserved we 
can define a multiset grading on $\T$. Let $[s]$ indicate the underlying multiset associated to a 
composition $s$. Then for a multiset $S$ we have define 

$$T_S=\{T\in \T\,\mid\, [\signat(T)]\in S\}.$$

We then have that  $$\T=\bigcup_{S \text{ a multisubset of }\NN} T_S,$$
with the property that if $T \in T_S$ and $T_i\in T_{S_i}$ for $i\in [k]$ then 
$$\gamma(T;T_1,T_2,\dots,T_k) \in T_{S\cup \bigcup_i S_i}.$$

There is another operadic structure that can be associated to the set $\T$ where composition has a 
more predictable behavior on signatures. Following a similar idea than in the work of Lopez, 
Preville-Ratelle and Ronco in \cite{LopezPrevilleRatelleRonco2015} on composition on what they 
call $m$-Dyck paths, we can define a composition of $s$-Dyck paths or their corresponding $s$-trees. 
In \cite{LopezPrevilleRatelleRonco2015} an $m$-Dyck path is defined as a lattice path from $(0,0)$ 
to $(nm,nm)$, taking only north steps $(m,0)$ and east steps $(0,1)$ and such that $m$ times the 
number of north steps taken at any given point is larger than the number of east steps taken. 
Let us denote the set of $m$-Dyck paths by $\Dyck^m$. There is an almost trivial bijection between 
$\Dyck^m$ and $\DP_{((m+1)^n)}$ by just keeping the same word in the letters $N$ and $E$ that 
describes $D\in \Dyck^m$ and adding a trailing $E$ to get a path in $\DP_{((m+1)^n)}$, i.e., if 
$D=NE^{i_1}NE^{i_2}\cdots NE^{i_a} \in \Dyck^m$ if and only if $DE=NE^{i_1}NE^{i_2}\cdots NE^{i_a}E 
\in \DP_{((m+1)^n)}$. Composition is defined in~\cite{LopezPrevilleRatelleRonco2015} as 
follows: let~$D,D_0,D_1\cdots,D_k\in \Dyck^m$ such that the NE word associated to 
$D=NE^{i_1}NE^{i_2}\cdots NE^{i_a}$ has $i_a=k$, then the composition is defined as

$$\tilde \gamma (D;D_0,D_1,\cdots,D_k)=NE^{i_1}NE^{i_2}\cdots ND_0ED_1E\cdots ED_k.$$

From the perspective of $s$-trees under the bijection of Theorem 
\ref{thm:bijection_trees_paths}, it can be shown that this composition can be seen as a special case of the composition 
\[\gamma(T;\bullet,\bullet,\cdots,\bullet,T_0,T_2,\dots,T_k),\] 
where the only non-identity trees 
substituted in $T$ are attached to leaves of $T$ that follow the last internal node when the nodes 
are ordered in preorder. Under this conditions it then follows that
$$\signat(\gamma(T;\bullet,\bullet,\cdots,\bullet,T_0,T_2,\dots,
T_k))=\signat(T)\oplus\signat(T_0)\oplus\cdots\oplus\signat(T_k),$$
and hence the composition in this operadic structure behaves well under concatenation of 
signatures. Here $\T=\bigcup_{s\in \comp}\T_s$ does not have an arity grading anymore but 
instead a grading in terms of the length $\ell(s)$ of the signature.

Finally, algebraic structures can be associated to the two operadic structures discussed 
above to give constructions related to the Hopf algebras of Loday and Ronco for planar binary 
trees \cite{LodayRonco1998}, the Hopf algebras of Novelli and Thibon for 
$(m+1)$-ary trees \cite{NovelliThibon2014} and the operadic constructions on $\Dyck^m$ 
\cite{LopezPrevilleRatelleRonco2015}.

\appendix
\section{Bijections with $s$-Dyck paths}
\label{sec_bijectionsWithDyckPaths}

In this appendix we describe further bijections between $s$-Catalan objects and $s$-Dyck paths. As in the case of $s$-trees, these bijections turn out to be elegant and simple (see Figure~\ref{fig_bijections_sDyckpaths}). They are obtained by assigning some labels to an $s$-Dyck path, and then reading or grouping the labels according to a certain rule. The bijections presented here are in accordance to those presented in Section~\ref{sec_sCatalanzooAndBijections}. More precisely, if $D$ is the $s$-Dyck path corresponding to an $s$-tree~$T$ (according to Section~\ref{sec:trees-DyckPaths}), then an $s$-Catalan object associated to $D$ in this appendix is the same $s$-Catalan object associated to $T$ according to the corresponding bijection in Section~
\ref{sec_sCatalanzooAndBijections}. All the details about these bijections are simple and left to be filled by the interested reader. 

\begin{figure}[htbp]
\begin{center}
\input{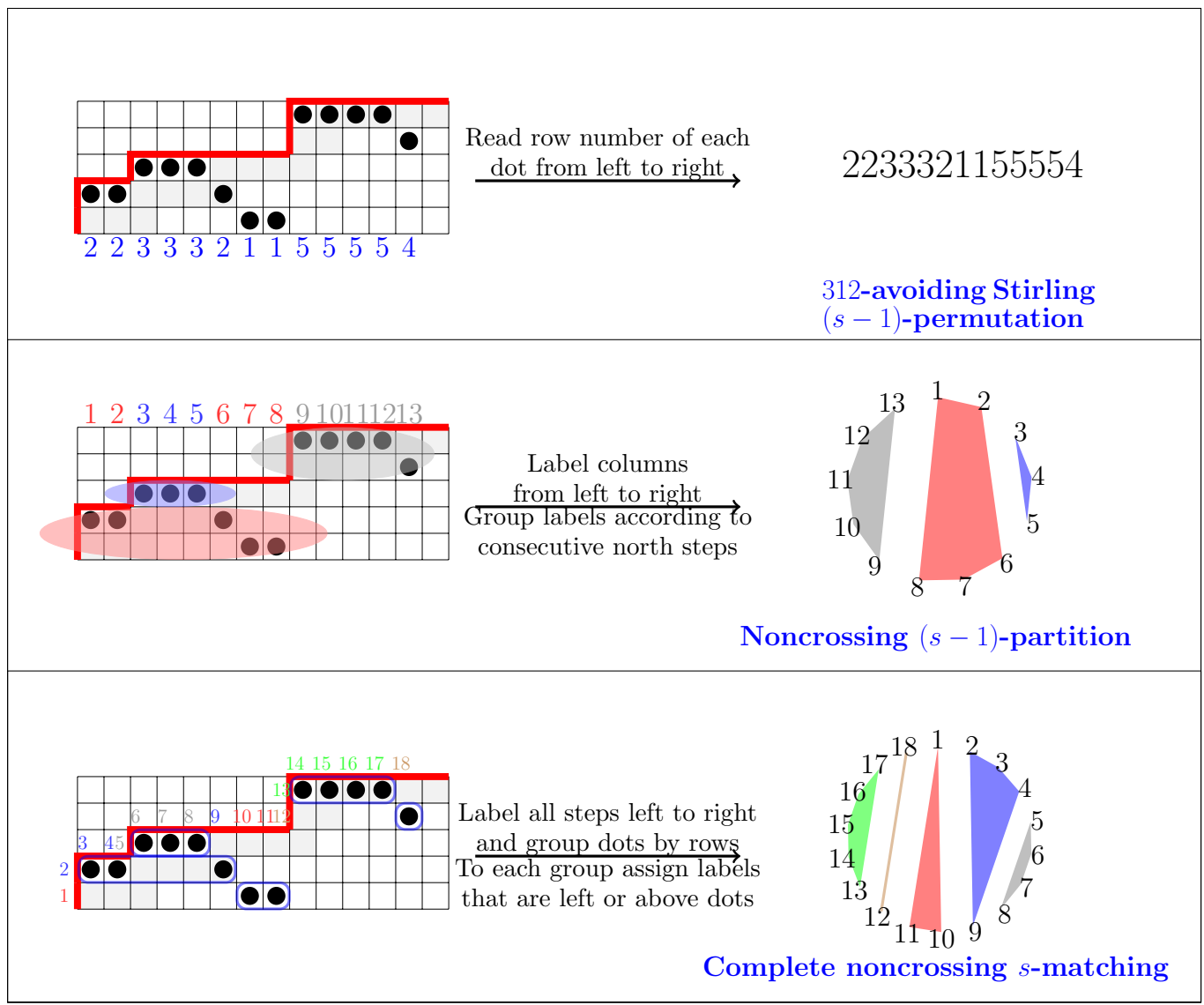}
\caption{Examples of the bijections with $s$-Dyck paths.}
\label{fig_bijections_sDyckpaths}
\end{center}

\end{figure}

\subsection{$s$-Dyck paths and $312$-avoiding Stirling $(s-\bf{1})$-permutations}

Let $\sigma$ be an $(s-\bf{1})$-permutation. We represent the permutation in a $(|s|-\ell(s)+1) \times \ell(s)$ 
rectangle by placing a dot in the square at position $(i,j)$ if $\sigma(i)=j$. Note that all 
the columns of the rectangle except the last one are occupied. 
The top part of Figure~\ref{fig_bijections_sDyckpaths} illustrates an example for the 
$(2,3,3,1,4)$-permutation $\sigma=2233321155554$.  
Given such representation of $\sigma$ we shade all the boxes in the rectangle that are weakly below 
and weakly to the right of some dotted box. It is not hard to see that the shaded boxes bound an 
$s$-Dyck path~$D(\sigma)$. Note that many permutations may give rise the same Dyck path. However, if we restrict to $312$-avoiding Stirling $(s-\bf{1})$-permutations there is exactly one permutation for each Dyck path. This permutation is constructed as follows: place $s(j)-1$ dots, from left to right, in the highest unoccupied row $j$ that is below the path 
$D$, such that no previously occupied columns are taken. The permutation~$\sigma(D)$ is defined as 
the permutation with the resulting rectangular representation.

\subsection{$s$-Dyck paths and noncrossing $(s-\bf{1})$-partitions}
Given an $s$-Dyck path $D$ we place $s(j)-1$ dots in row $j$ for each $j$ as above, and number the columns of the grid (except the last one) from~$1$ to~$|s|-\ell(s)$.   
We make an horizontal strip to the right of each maximal consecutive chain of north steps in $D$. The columns of the dots in each strip form the blocks of the noncrossing $(s-\bf{1})$-partition $\pi(D)$ associated to $D$. The middle part of Figure~\ref{fig_bijections_sDyckpaths} illustrates an example. Note that the number of blocks of the partition is equal, by definition, to the number of peaks of the Dyck path. The inverse map can be easily described as follows: let $\pi=\{\pi_1,\dots,\pi_k\}$ be a noncrossing $(s-\bf{1})$-partition whose blocks are ordered according to the order of their minimal elements. Partition the sequence $s-\bf{1}$ in $k$ consecutive blocks such that the sum of the values in block $i$ is equal to the size $|\pi_i|$.
In the middle of Figure~\ref{fig_bijections_sDyckpaths}, $s-\bf{1}=(2,3,3,1,4)$ and 
\[\pi=\{\{1,2,6,7,8\}, \{3,4,5\},\{9,10,11,12,13\}\}.\]
The sizes of the blocks of $\pi$ are $(5,3,5)$, and so we partition $s-\bf{1}$ as $|2,3|3|1,4|$. 
For $1\leq i \leq k$, define $x_i=\min \pi_i -1$ and $y_i$ equal to the position of the last number in block~$i$ of the partitioned~$s-\bf{1}$. In our running example, 
\[
(x_1,x_2,x_3)=(1-1,3-1,9-1)=(0,2,8),
\]     
\[
(y_1,y_2,y_3)=(2,3,5).
\]
The $s$-Dyck path corresponding to $\pi$ is the path with peaks at coordinates $(x_i,y_i)$ for $i\in [k]$. 

\subsection{$s$-Dyck paths and complete noncrossing $s$-matchings}
\label{sec_bijection_sDyckpaths_smatchings}
Given an $s$-Dyck path~$D$, we place $s(j)-1$ dots in row $j$ for each $j$ as above, and number the steps of $D$ (except the last one) from~$1$ to~$|s|$. For each $1\leq i \leq \ell(s)$, define $M_i$ to be the set containing the label of the north step of $D$ in row $i$, together with the labels of the east steps in the columns of the dots in that row. The complete noncrossing $s$-matching associated to $D$ is defined by~$M(D):=\{M_1,\dots ,M_{\ell(s)}\}$. The bottom part of Figure~\ref{fig_bijections_sDyckpaths} illustrates an example. The inverse is determined by the positions of the north steps in the path, which are the minimal elements of the blocks in $M$.

\section*{Acknowledgements}
This project started in November 2014, during a visit of the second author to the first author at York University and the Fields Institute. We thank the Banting Fellowships Program of the Government of Canada and York University for their financial support and for making this visit possible. 
We also thank Robin Sulzgruber for his input in the description of the bijection between $s$-Dyck paths and complete noncrossing $s$-matchings in Section~\ref{sec_bijection_sDyckpaths_smatchings}; and Nantel Bergeron for the many useful conversations.

\bibliographystyle{plain}
\bibliography{scatalan}

\end{document}